\documentclass{article}

\usepackage[utf8]{inputenc}
\usepackage{amssymb,amsmath,mathrsfs,amsthm,bbm,xcolor}
\usepackage{verbatim}%for comment
\usepackage{indentfirst}
\usepackage{geometry}
\usepackage{enumerate} \usepackage{graphicx} 
\usepackage{booktabs,caption}
\usepackage[flushleft]{threeparttable}
\usepackage{hyperref} %navigation
\hypersetup{colorlinks=true,allcolors=blue}
\usepackage{hypcap}
\usepackage{soul}
\usepackage{import}
\usepackage[capitalize]{cleveref} % cleve reference

\newtheorem{thm}{Theorem}[section]
\newtheorem{prop}[thm]{Proposition}
\newtheorem{cor}[thm]{Corollary}
\newtheorem{exam}[thm]{Example}
\newtheorem{lem}[thm]{Lemma}
\newtheorem*{lem*}{Lemma}

\newtheorem*{claim*}{Claim}

\theoremstyle{remark}
\newtheorem{rem}[thm]{Remark}
\newtheorem*{rem*}{Remark}

\theoremstyle{definition}
\newtheorem{defi}[thm]{Definition}

\makeatletter
\newtheorem*{rep@theorem}{\rep@title}
\newcommand{\newreptheorem}[2]{%
\newenvironment{rep#1}[1]{%
 \def\rep@title{#2 \ref{##1}}%
 \begin{rep@theorem}}%
 {\end{rep@theorem}}}
\makeatother
\newreptheorem{thm}{Theorem}
\newreptheorem{lem}{Lemma}

\numberwithin{figure}{section}

\input{macros}

\title{Selberg, Ihara and Berkovich}

\author{Jialun Li, Carlos Matheus, Wenyu Pan and Zhongkai Tao}
\date{}

\begin{document}

\maketitle

\begin{abstract}
We use the Selberg zeta function to study the limit behavior of resonances in a degenerating family of Kleinian Schottky groups. We prove that, after a suitable rescaling, the Selberg zeta functions converge to the Ihara zeta function of a limiting finite graph associated to the relevant non-Archimedean Schottky group acting on the Berkovich projective line.

Moreover, we show that these techniques can be used to get an exponential error term in a result of McMullen (recently extended by Dang and Mehmeti) about the asymptotics for the vanishing rate of the Hausdorff dimension of limit sets of certain degenerating Schottky groups generating symmetric three-funnel surfaces. Here, one key idea is to introduce an intermediate zeta function capturing \emph{both} non-Archimedean and Archimedean information (while the traditional Selberg, resp. Ihara zeta functions concern only Archimedean, resp. non-Archimedean properties). 
\end{abstract}

\setcounter{tocdepth}{2} % for table of contents, until the level of subsection.
\tableofcontents

\section{Introduction}

A \emph{Schottky group} of rank $g$ \emph{over} $\mathbb{C}$ (also called Kleinian Schottky group) is a discrete subgroup of $\mathrm{SL}_2(\mathbb{C})$ which is purely loxodromic and isomorphic to the free group of rank $g$. These groups were introduced by Schottky in 1877 and they were classically studied in connection with the famous uniformization theorem for Riemann surfaces: indeed, Koebe proved in 1910 that any closed Riemann surface is uniformized by some Kleinian Schottky group. 

Partly motivated by the great success of the study of degenerating families of Riemann surfaces (ultimately leading to the celebrated Deligne--Mumford compactification of moduli spaces and its numerous applications), we shall investigate in this paper the analytical and geometrical properties of \emph{degenerating} families of Schottky groups over $\mathbb{C}$. 

More concretely, given a family $\Gamma_z$ of Schottky groups over $\mathbb{C}$ generated by $2\times 2$ matrices whose entries are meromorphic functions of $z\in\mathbb{D}^*$, it is sometimes possible to introduce a Schottky group $\Gamma^{na}$ over the non-Archimedean field $\Ct$ (of Laurent series on $t$) whose features can be used to extract asymptotic information on $\Gamma_z$ as $|z|\to 0$. In fact, this philosophy of studying degenerating families of complex objects via non-Archimedean limiting objects is nowadays widely spread in the literature (cf. \cite{Kiwi2015}, \cite{BoJo2017}, \cite{dujardinDegenerationsProtectMathrm2019}, \cite{favreDEGENERATIONENDOMORPHISMSCOMPLEX2020} and \cite{Luo2022} for some recent works) and, in the context of Schottky groups, this idea was exploited by Dang and Mehmeti \cite[Theorem 3]{dangHausdorffDimension2024} to prove that certain degenerating families $\Gamma_z$ of Schottky groups over $\mathbb{C}$ possess \emph{limit sets} $L(\Gamma_z)\subset\mathbb{P}^1_{\mathbb{C}}$ whose Hausdorff dimensions satisfy 
\begin{equation}\label{equ-haus conv}
\textrm{dim}(L(\Gamma_z))\sim \frac{d_0}{\log(1/|z|)}
\end{equation}
as $|z|\to 0$, where $d_0$ is the Hausdorff dimension of the limit set of a non-Archimedean Schottky group $\Gamma^{na}$ acting on the \emph{Berkovich projective line} $\mathbb{P}^{1,an}_{\Ct}$. (We also recommend to the reader the very recent paper \cite{CourtoisGuilloux2024} of Courtois and Guilloux containing a partial extension of the work of Dang and Mehmeti to \emph{higher} dimensions via an \emph{alternative method}, namely, the study of the infinite dimensional hyperbolic space.) 

In the present paper, we pursue this kind of philosophy in order to describe the Selberg zeta functions of certain degenerating families of Schottky groups over $\mathbb{C}$ in terms of Ihara zeta functions of finite graphs naturally attached to non-Archimedean Schottky groups: see Subsection \ref{ss.intro2} below for precise statements. As a consequence, we are able to significantly refine the result of Dang and Mehmeti mentioned above. In particular, we improve upon the work of McMullen \cite{mcmullenHausdorffDimensionConformal1998} on the Hausdorff dimension of limit sets of Schottky reflection groups associated to symmetric three-funnel hyperbolic surfaces and the theoretical and numerical works of Weich \cite{weichResonanceChainsGeometric2015}, Borthwick \cite{borthwickSpectralTheoryInfiniteArea2016} and Pollicott--Vytnova \cite{pollicottZerosSelbergZeta2019a} on the convergence of rescaled Selberg zeta functions of symmetric three-funnel hyperbolic surfaces: see Subsection \ref{ss.intro1} below for concrete statements. 

\subsection{Example: symmetric three-funnel surfaces}\label{ss.intro1}
    Let $X(\ell)$ be the symmetric three-funnel hyperbolic surface with parameter $\ell$, that is, the non-compact hyperbolic surface of genus zero that has three funnels of widths $\ell > 0$. This gives an example of a family of degenerating surfaces when $\ell$ tends to infinity. Denote by $\Gamma(\ell)$ the Schottky group uniformizing $X(\ell)$, and let $L(\Gamma(\ell))$ be the limit set of $\Gamma(\ell)$ on the visual boundary $\partial\H^2$.
    McMullen \cite{mcmullenHausdorffDimensionConformal1998} studied the asymptotics of the Hausdorff dimension of $L(\Gamma(\ell))$ as $\ell\to\infty$ and he obtained $\dim (L(\Gamma(\ell)))\approx \frac{\log 4}{\ell}$. 

    For this particular example, as it will be explained in Example \ref{exa.borth} below, the main results of this paper imply an exponential error term in McMullen's asymptotic formula above: 
\begin{thm}\label{thm:three-funnel}

      For the case of symmetric three-funnel hyperbolic surfaces, we have for any $\epsilon>0$,
    \[\dim (L(\Gamma(\ell)))=\frac{\log 4}{\ell}+O_\epsilon(e^{-(1-\epsilon)\ell/4}). \]

    Moreover, for the Selberg zeta function (see \cref{equ:selberg-zeta} for definition), we have
      \[|Z(\Gamma(\ell),s/\ell)-(1-6e^{-s}+9e^{-2s}-4e^{-3s})|=O_{K,\epsilon}(e^{-(1-\epsilon)\ell/4}), \]
    for any $s$ in a fixed compact set $K$ and any $\epsilon>0$.
\end{thm} 

The convergence of the Selberg zeta function without speed is due to Weich \cite{weichResonanceChainsGeometric2015}. 
Pollicott--Vytnova \cite{pollicottZerosSelbergZeta2019a} proved a polynomial error term $O(1/\sqrt{\ell})$ in the convergence of the Selberg zeta functions and $O(1/\sqrt{\ell^3})$ in the convergence of Hausdorff dimensions (and, in fact, Example \ref{exa.PV19} below shows how to recover their results from our techniques). The exponential error term above also explains that the numerics of resonances in Borthwick \cite{borthwickSpectralTheoryInfiniteArea2016} behave well even for small $\ell$ such as $\ell=10$: see Example \ref{exa.funnel-torus} below for more details. 

Our method of proof is different from previous works on symmetric three-funnel surfaces. Indeed, this family is just an example of a degenerating Schottky family and, as it is explained below, we actually study the limit behavior of all such families.

\subsection{Main results: degenerating Schottky families}\label{ss.intro2}
The study of the degenerations at the origin $z=0$ of meromorphic families $\Gamma_z$ of Kleinian Schottky (and, more generally, non-elementary) groups depending on a complex parameter $z\in \bb D^*$ in the unit punctured disc of $\mathbb{C}$ is a topic of intense investigations with a vast literature around it. For instance, such families are useful to build compactifications of representation variety, character variety or moduli spaces (see, e.g., \cite{cullerVarietiesGroupRepresentations1983} \cite{otalCompactificationSpacesRepresentations2015} and \cite{poineauBerkovichCurvesSchottky2021}), their Lyapunov exponents for the associated random walks satisfy fascinating asymptotics (see, e.g., \cite{AvronCraigSimon} and \cite{dujardinDegenerationsProtectMathrm2019}), and, as we already mentioned above, Hausdorff dimensions of their limit sets attracted the attention of many authors. 

In the present paper, the central objects of discussion will be the Selberg zeta functions. Recall that the Selberg zeta function for a Schottky group $\Gamma$ of $\SL_2(\C)$ is defined as follows. Let $\calP$ be the set of oriented primitive periodic geodesics on $\mathrm{T}^1(\Gamma \backslash \mathbb{H}^3)$.
 \begin{equation}\label{equ:selberg-zeta}
  Z(\Gamma,s)=\prod_{\gamma\in \calP }\prod_{k_1,k_2\geq 0}(1-e^{-(s+k_1+k_2)\ell(\gamma)}e^{-i\theta_{\gamma}(k_1-k_2)}).      
 \end{equation}
Here, by a slight abuse of notation, we also denote by $\gamma$ an oriented closed geodesic representing a loxodromic element $\gamma\in \SL_2(\CC)$, so that $\ell(\gamma)$ is the length of the periodic geodesic $\gamma$ in $\Gamma\backslash \HH^3$, and $e^{i\theta_{\gamma}}$ is the holonomy given by $\gamma'(\gamma_+) = e^{-\ell(\gamma)-i\theta_{\gamma}} $ with $\gamma_+$ being the attracting fixed point of $\gamma$. Similarly, for $\Gamma<\SL_2(\RR)$, the classical Selberg zeta function is defined as
 \begin{equation*}
      Z(\Gamma,s)=\prod_{\gamma\in \calP}\prod_{k\geq 0}(1-e^{-(s+k)\ell(\gamma)}).
 \end{equation*}
Note that the definitions for $\SL_2(\RR)$ and $\SL_2(\CC)$ do not coincide, and, for this reason, we will deal with the two cases separately.

The function $Z(\Gamma, s)$ is convergent and analytic in the region $\Re s> \delta_\Gamma$, where $\delta_{\Gamma}$ is the so-called critical exponent. It admits an analytic continuation to the whole complex plane $\mathbb{C}$, and hence we obtain an entire function $Z(\Gamma, s)$. Due to \cite{pattersonDivisorSelbergZeta2001}, we know that the poles of the resolvent of the Laplacian on $\Gamma\backslash \mathbb{H}^3$ (also called resonances) correspond to a subset of the zeros of the Selberg zeta function $Z(\Gamma,s)$.   

For a degenerating family of Schottky groups, we define the \textit{intermediate zeta function}. We will prove in \cref{prop:expansion of length function} and \cref{cor:analytic-middle} that it is well-defined and has a holomorphic extension to $s\in\CC$. 
\begin{defi}[Intermediate zeta function]
   Let $\gamma\in \SL_2(\Ct)$ and $M\in \Z_{\geq 0}$, we define the approximate length function
   \begin{equation*}
       \ell_M(\gamma,z)=\ell^{na}(\gamma)+\Re\sum\limits_{j=0}^{M}a_j(\gamma)z^j/\log(1/|z|)
   \end{equation*}
   where $a_j(\gamma)$'s are the coefficients in the expansion
   \begin{equation*}
       \ell(\gamma_z)/\log(1/|z|)=\ell^{na}(\gamma)+\Re\sum\limits_{j=0}^{\infty}a_j(\gamma)z^j/\log(1/|z|).
   \end{equation*}
   The intermediate $M$-zeta function is defined by
   \begin{equation*}
       Z_M(\Gamma,z,s)=\prod_{[\gamma]\in\calP}(1-e^{-s\ell_M(\gamma,z)})
   \end{equation*}
   for $\Re s$ sufficiently large.
   \footnote{See \cref{exa.funnel-torus} for one explicit but nontrivial intermediate zeta function for hyperbolic funneled torus. }
\end{defi}

Here is our main result:
 \begin{thm}\label{thm:zeta-conv} Let $\Gamma_z$ be a degenerating Schottky family satisfying condition $(\bigstar)$ in Subsection \ref{Degeneration of Schottky groups}, and denote by $\Gamma<\SL_2(\Ct)$ be the corresponding non-Archimedean Schottky group. Then for any $s\in \C$, as $|z|\rightarrow 0$, we have
      \[ Z(\Gamma_z,s/\log(1/|z|))\rightarrow Z_{I}(\Gamma,s), \]
      where $Z_I(\Gamma,s)$ is the Ihara zeta function associated to $\Gamma$ (defined in \cref{prop:na-Ihara}) and the convergence on any compact set is uniform.

Moreover, for any compact set $K$ and $\epsilon>0$, we have for all $0<|z|<1/e$ and for any $s\in K$,
\[ |Z(\Gamma_z,s/\log(1/|z|))-Z_0(\Gamma,z,s)|\lesssim_{K,\epsilon} |z|^{1-\epsilon}. \]

Furthermore, for any $C,\epsilon>0$ and $M\in \mathbb{Z}_{\geq 0}$, we have for $s\in[-C,C]+i[-C|z|^{-M},C|z|^{-M}]$ and $0<|z|<1/e$,
    \begin{equation*}
    |Z(\Gamma_z,s/\log(1/|z|))-Z_M(\Gamma,z,s)|\lesssim_{C,M,\epsilon} |z|^{1-\epsilon}.
    \end{equation*}
     
 \end{thm}

As a direct corollary, we obtain
\begin{cor}\label{cor:dimension}
    Let $\Gamma_z$ be a degenerating Schottky family satisfying the condition $(\bigstar)$ in Subsection \ref{Degeneration of Schottky groups}. Let $P(|z|)$ be the first zero of the intermediate zeta function $Z_0(\Gamma, z,s)$ of the variable $s$. We have for any $\epsilon>0$ and $|z|<1/e$
    \[ \dim (L(\Gamma_z))=P(|z|)+O_\epsilon(|z|^{1-\epsilon}).  \]
\end{cor}

For special examples, like symmetric three-funnel surfaces and \cref{exa.o-o}, the main term has the form $P(|z|)=d_0/\log(1/|z|)$. For general examples, the main term $P(|z|)$ is a zero of a polynomial of exponential functions of the variable $s$ like $\sum a_je^{-s(b_j+c_j/\log(1/|z|))}$.  
It would be interesting to study the Taylor expansion of $P(|z|)$ on $1/\log(1/|z|)$. 

\begin{rem} Due to the definition of $Z_0$, we know $P(|z|)$ and the coefficients of the expansion only depend on the coefficients of $\tr(\gamma)$ with $\gamma\in\Gamma$. For example, if we write $Z_0(\Gamma, z,s)=Z_0(s,u)$ with $u=1/\log(1/|z|)$, then by the implicit function theorem, we obtain 
\[P(|z|)=\frac{d_0}{\log(1/|z|)}-\frac{\partial_2Z_0}{\partial_1Z_0}(d_0,0)\frac{1}{(\log(1/|z|))^2}+O\left(\frac{1}{(\log(1/|z|))^3}\right).\]
Can we find other interpretations of the term $\frac{\partial_2Z_0}{\partial_1Z_0}(d_0,0)$, such as the variations of Hausdoff dimension \cite{mcmullenThermodynamicsDimensionWeilPetersson2008} or the expansion of Lyapunov exponent \cite{AvronCraigSimon}?
\end{rem}

We also obtain a corollary on the zeros of zeta functions and, consequently, on resonances of the Laplacian.
\begin{cor}\label{cor-resonance converge}
    The zeros of $Z_{I}(\Gamma,s)$ have the structure $\rm{Res}_I(\Gamma)=\{\mu_j\}+2\pi i \ZZ$,
    where $\{\mu_j\}$ is a finite set. Outside $s=0$, the resonances of $\Gamma_z$ converge to $\rm{Res}_{I}(\Gamma)$ (counted with multiplicity) after rescaling, i.e.
    \begin{equation}\label{eq:res-minus-0}
        \log(1/|z|)\,\rm{Res}(\Gamma_z)\setminus \{0\}\to \rm{Res}_{I}(\Gamma)\setminus \{0\},\quad |z|\to 0
    \end{equation}
    on compact sets. In particular, the spectral gap of $\Gamma_z$ converges to $0$ after rescaling by $\log(1/|z|)$.
\end{cor} 

\begin{proof}
    It follows from \cref{thm:zeta-conv} that the zeros of $Z(\Gamma_z,s/\log(1/|z|))$ converges to the zeros of $Z_I(\Gamma,s)$ in any compact set. See \cref{cor:res-conv} for a more detailed version. The structure of $\rm{Res}_I(\Gamma)$ follows from the fact that $Z_I(\Gamma,s)$ is a polynomial of $e^{-s}$ (see \cref{prop:ZI}). The convergence \cref{eq:res-minus-0} follows from the fact that the resonances of the Laplacian coincides with the zeros of $Z(\Gamma_z,s)$ outside $s\in \mathbb{Z}_{\leq 0}$ (see \cite[Chapter 10]{borthwickSpectralTheoryInfiniteArea2016} for the two dimensional case and \cite{pattersonDivisorSelbergZeta2001} for the general case). 
\end{proof}

\begin{rem}
    1. By \cref{prop:na-Ihara}, the Ihara zeta function $Z_I(\Gamma,s)$ associated to $\Gamma$ is the Ihara zeta function of a finite graph which comes from some Berkovich space quotient by $\Gamma$.

    2. It is classic that the Hausdorff dimension of the limit set equals the first zero of the Selberg zeta function. For the non-Archimedean case, see \cref{sec.appendix} for the equality between Hausdorff dimension and the first zero of the Ihara zeta function. We recover the convergence theorem of the Hausdorff dimension of the limit set of \cite{dangHausdorffDimension2024} \cref{equ-haus conv} from \cref{cor-resonance converge}.

    3.      
    The work \cite{OhWinter} proves that the family of congruence Schottky surfaces has a uniform spectral gap (in both low and high frequencies).
    The works \cite{MageeNaud}, \cite{calderonMageeNaud} study the optimal size of the low frequency spectral gap under random covers. Our work shows that the gap becomes small under degenerate situations. It is interesting to ask if the gap also gets smaller in high frequency, i.e. when $\Im s$ is big. The conjecture of Jakobson--Naud \cite{JakobsoNaud} suggests that this does not happen. 

   4. When the corresponding Mumford curve $\Sigma_X$ of $\Gamma$ is a $q$-regular graph with each edge length $1$ (see \cref{sec:na-schottky} for definition), the Ihara zeta function is given by $Z_I(\Gamma,s)=(1-e^{-2s})^{r-1}\det(I-A_Xe^{-s}+(q-1)e^{-2s})$ where $r$ is the number of generators of $\Gamma$ and $A_X$ is the adjacency matrix of $\Sigma_X$ (see for example \cite{hortonWhatAreZeta2006}). Let $\lambda_n\leq \cdots \leq \lambda_2\leq \lambda_1=q$ be the eigenvalues of $A_X$. Therefore the set of resonances $\rm{Res}_I(\Gamma)$ is determined by $\mu_j$'s, which satisfy $e^{\mu_j}=(\lambda_i\pm\sqrt{\lambda_i^2-4(q-1)})/2(q-1) $ or $e^{\mu_j}=\pm 1$. This is similar to the relation between $\rm{Res}(\Gamma_z)$ and eigenvalues/resonances of the Laplacian operator on the quotient manifold.
    
    Moreover, if the graph $\Sigma_X$ is a Ramanujan graph, that is $-2\sqrt{q-1}\leq \lambda_n\leq \cdots\leq \lambda_2\leq 2\sqrt{q-1}$, then all $\mu_j$'s have real part $\log(q-1)/2$ except one $\mu_j$ equal to $\log(q-1)$ and some $\mu_j$ equal to $0$ or $\pi i$.

    5. The fractal Weyl law conjecture \cite[Conjecture 5]{zworski-survey} states that the number of resonances in a strip grows according to the dimension of the limit set at $T\to \infty$:
    \begin{equation*}
        \#\rm{Res}(\Gamma)\cap \{\Re s>-C, \, |\Im s|\leq T\}\sim T^{1+\delta},\quad \delta=\dim(L(\Gamma)).
    \end{equation*}
    Moreover, it is also conjectured in \cite[Conjecture 7]{zworski-survey} that the resonances concentrate near the axis $\Re s=\delta/2$.
    The upper bound is known by \cite{guillopeSelbergZetaFunction2004}, but the sharp lower bound is not known in any nontrivial example of Schottky groups. Our work provides a perspective on this conjecture via the intermediate zeta functions, which we conjecture to have a similar growth pattern as the fractal Weyl law, see the discussion after \cref{exa.PV19}.
\end{rem}

\subsection*{Acknowledgment}

We would like to thank the hospitality and the support of Centre Mathématiques Laurent Schwartz, École polytechnique and Coll\`ege de France.
We want to thank Bac Dang and Vler\"e Mehmeti for their talk on their work.
We would also like to thank Peter Sarnak for helpful and inspiring discussion.
JL was partially supported by the starting grant of CNRS.
ZT was partially supported by the NSF grant DMS-1952939 and by the Simons Targeted Grant Award No. 896630.

\subsection{Notations}\label{ss.notations}
Throughout the paper, we adopt the following conventions.
\begin{itemize}
\item We write $A \lesssim B$ if there exists a positive constant $C > 0$ (that may differ from expression to expression) such that $A \leq C B$. We specify the dependencies of $C$ by a subscript, e.g., $A \lesssim_{\Gamma} B$. We use $A\ll B$ to mean that $A$ is much smaller than $B$.
\item $D(a,r):=\{w\in\CC: |w-a|<r\}$ is a disc of radius $r$ centered at $a\in\CC$.
\item $\mathbb{D}=D(0,1)$ is the unit disc and $\mathbb{D}^*=\mathbb{D}\setminus \{0\}$. Similarly, $\mathbb{D}_r=D(0,r)$ and $\mathbb{D}_r^* = \mathbb{D}_r \setminus\{0\}$.
\item The variable $t$ usually means the formal variable in the field of Laurent series $\Ct$. This field is equipped with its usual $t$-adic norm $|f|_{na}=\exp(-\textrm{ord}_{t=0}(f)) = e^{-n}$ for $f=t^n\sum\limits_{m=0} f_m t^m\in \Ct$ with $f_0\neq 0$. The variable $z$ usually means a complex number that parametrizes a degenerating family of Schottky groups over $\CC$ or $\RR$.
\item The Banach ring $A_r\subset \Ct$ is defined by $f=\sum a_n t^n\in\Ct$ such that the hybrid norm
\begin{equation*}
    \|f\|_{A_r}=\sum \max\{|a_n|_{\infty}, |a_n|_0\} r^{n}
\end{equation*}
is finite ($|\cdot|_0$ is the trivial norm). We will usually take $r=1/e$ and define $\|\cdot\|_{\mathrm{hyb}}:=\|\cdot\|_{A_{1/e}}$.
\item Let $M\in \ZZ_{\geq 0}$ be a nonnegative integer and $a=t^m\sum_{n=0}^{\infty} a_n t^n\in \Ct$ with $a_0\neq 0$. We define the leading $M$ terms
\begin{equation*}
    \lt_M(a):= t^m\sum_{n=0}^{M} a_n t^n.
\end{equation*}
\item We use Fraktur letters $\mathfrak{a}, \mathfrak{b}$ to mean an element of the (fixed) set of generators on $\Gamma$. We use bold letters $\mathbf{a},\mathbf{b}$ to mean words of the generators.
\item For an element $\gamma\in\SL(2,\Ct)$, we use $\gamma_z$ to denote an element in $\SL(2,\C)$ by evaluating at $t=z$ if all the coefficients of $\gamma$ converge at $z$.
\item  $\ell(\gamma)\in\RR_{\geq 0}$ means the length of a closed geodesic associated to $\gamma\in \Gamma$. $\ell^{na}(\gamma)\in\ZZ_{\geq 0}$ is the non-Archimedean length defined in \cref{defi:na length}. $l(\gamma)\in\ZZ_{\geq 0}$ is the word length of $\gamma$ for a fixed set of generators $\mathfrak{a}_1,\cdots,\mathfrak{a}_{2g}$ with $\frak{a}_{i+g}=\frak{a}_i^{-1}$, i.e. $l(\gamma)=n$ whenever $\gamma=\mathfrak{a}_{i_1}\mathfrak{a}_{i_2}\cdots \mathfrak{a}_{i_n}$ with  $\mathfrak{a}_{i_{j+1}}\neq \mathfrak{a}_{i_j}^{-1}$.

\end{itemize}

\section{Preliminaries of Schottky groups} 
\label{sec: preliminaries}
We recall some preliminaries on the Schottky groups.

\subsection{Schottky groups of $\textrm{PGL}_2(\mathbb{C})$} 
An element $\gamma=\left(\begin{array}{cc} a & b \\ c & d   \end{array}\right)\in \textrm{PGL}_2(\mathbb{C})$ is called \emph{loxodromic} if it induces a M\"obius transformation $z\mapsto \frac{az+b}{cz+d}$ on $\C$ which is conjugated to $z\mapsto \lambda z$ for a complex number $\lambda$ with $0<|\lambda|< 1$.

\begin{defi}\label{defi:schottky} We say that a free subgroup $\Gamma=\langle M_1,\dots, M_g\rangle\subset \textrm{PGL}_2(\mathbb{C})$ with $g\geq 2$ generators is a \emph{Schottky group of rank $g$} if
\begin{itemize} 
\item every $\gamma\in\Gamma\setminus\{id\}$ is loxodromic; 
\item $\Gamma$ admits a non-empty domain of discontinuity, i.e., it acts freely and properly on a non-empty connected invariant open subset of $\mathbb{P}^1(\mathbb{C})$.
\end{itemize}
\end{defi}

Given a Jordan curve $C$ in $\mathbb{P}^1(\mathbb{C})$ and a reference point $o\notin C$, we say that the exterior (resp. interior) of $C$ is the connected component of  $\mathbb{P}^1(\mathbb{C})\backslash C$ containing  (resp. not containing) $o$. After Maskit \cite{maskitCharacterizationSchottkyGroups1967}, a Schottky group has the form $\langle T_1,\dots, T_g\rangle$, where $T_j$'s induce M\"obius transformations with the following property: there are a reference point $o$ and Jordan curves $C_1, C_1',\dots, C_g, C_g'$ with disjoint interiors such that each $T_j$ maps the exterior of $C_j$ to the interior of $C_j'$.  If all the Jordan curves $C_j$, $C_j'$ can be chosen to be circles, then the corresponding Schottky group is called a \textit{classical Schottky group}.

\subsection{Kleinian Schottky spaces} 
A loxodromic element $\gamma\in \textrm{PGL}_2(\mathbb{C})$ is determined by three parameters: its attracting fixed point $\alpha\in\mathbb{P}^1(\mathbb{C})$, its repelling fixed point $\beta\in\mathbb{P}^1(\mathbb{C})$, and its contraction rate $\lambda\in\mathbb{D}^*$. In the sequel, we shall denote by $M(\alpha,\beta,\lambda)$ the loxodromic matrix parametrized by $\alpha$, $\beta$ and $\lambda$. 

The \emph{Schottky space $S_g(\mathbb{C})$ of rank $g\geq 2$}  
is the space of $g$ elements $M_1,\dots, M_g$ of $\textrm{PGL}_2(\mathbb{C})$ that generate a Schottky group up to the natural equivalence given by simultaneous conjugation by M\"obius transformations (cf. Bers \cite{bersAutomorphicFormsSchottky1975}). These elements are parametrized by $3g$ quantities $\alpha_1,\beta_1,\dots, \alpha_g,\beta_g,\lambda_1,\dots,\lambda_g$. Since the action of $\textrm{PGL}_2(\mathbb{C})$ on $\mathbb{P}^1(\mathbb{C})$ is $3$-point transitive, we can set $\alpha_1=0$, $\beta_1=2$, $\alpha_2=1$. In this way, we can regard $S_g(\mathbb{C})$ as a subset of $\mathbb{C}^{3g-3}$: each $p\in S_g(\mathbb{C})$ is parametrized by the so-called \emph{Koebe coordinates}    
$$(\alpha_3,\dots,\alpha_g,\beta_2,\dots,\beta_g,\lambda_1,\dots,\lambda_g)=(\underline{\alpha},\underline{\beta},\underline{\lambda})\in (\mathbb{P}^1(\mathbb{C})\setminus\{0,1,2\})^{2g-3}\times (\mathbb{D}^*)^{g}.$$ 
As it was shown by Hejhal \cite{hejhalSchottkyTeichmullerSpaces1975} (see also \cite{bersAutomorphicFormsSchottky1975}), $S_g(\mathbb{C})$ is a connected complex manifold of dimension $3g-3$.

\subsection{Limit sets of Kleinian Schottky groups}

Given a point $p\in S_g(\mathbb{C})$, let $L(\Gamma_p)$ be the limit set of the Schottky group $\Gamma=\Gamma_p$ associated to $p$. Sullivan \cite{sullivanDensityInfinityDiscrete1979} showed that the Hausdorff dimension $\textrm{dim}(L(\Gamma_p))$ coincides with the critical exponent $\delta_{\Gamma_p}$ of $\Gamma$ which is given by
$$\delta_{\Gamma_p} = \inf\{s>0:\mathcal{P}(s)<\infty\}.$$ 
Here $\mathcal{P}(s) = \sum\limits_{\gamma\in\Gamma_p}|\gamma'(o)|_{\infty}^{s}$ is the usual Poincar\'e series, $o$ is a point outside $L(\Gamma_p)$ and $|.|_{\infty}$ is the norm with respect to the Euclidean metric on $\mathbb{C}$. Due to the work of Patterson \cite{Pat1976,Pat1988,Pat1989}, we know that the critical exponent is also equal to the first zero of the Selberg zeta function.

As established by Anderson and Rocha \cite{andersonAnalyticityHausdorffDimension1997}, $\textrm{dim}(L(\Gamma_p))$ is a real analytical function of $p\in S_g(\mathbb{C})$. 

\subsection{Degenerations of Kleinian Schottky groups}
\label{Degeneration of Schottky groups}

Similarly to the case of moduli space of complex algebraic curves, one can try to approach the boundary of $S_g(\mathbb{C})$ by studying degenerations of one-parameter families of Schottky groups. More concretely, we always consider a meromorphic function $p:\mathbb{D}\to \mathbb{C}^{2g-3}\times \mathbb{D}^g, z\mapsto  p(z)=(\underline{\alpha}(z),\underline{\beta}(z),\underline{\lambda}(z))$ such that $p(z)\in S_g(\mathbb{C})$ for all $z\in\mathbb{D}^*$. From now on, when talking about Schottky family, we mean a family given by this meromorphic map $p(z)$.

Since it is not easy to handle \emph{general} degenerations of Schottky groups, we shall focus on the degenerations $p(z)$ leading to nice (Schottky) non-Archimedean limits (in the sense of Poineau and Turchetti \cite{poineauSchottkySpacesUniversal2022}) as $z\to 0$. For this sake, one follows Dang and Mehmeti \cite{dangHausdorffDimension2024} by considering the quantity 
$$K_{\min} = \min\limits_{\substack{i\neq j, i\neq k \\ u\in\{\alpha,\beta\}}} \textrm{ord}_{z=0}\left(\lambda_i(z)\cdot [u_j(z):u_k(z);\alpha_i(z):\beta_i(z)]\right)$$ 
(based on the order of vanishing of certain holomorphic functions) measuring how close the attracting and repelling fixed points can get in terms of cross-ratios versus the contraction rates\footnote{For example, by taking $i=1$, $j=k=2$, one gets $\lambda_1(z)\cdot[1:\beta_2(z);0:2] = \lambda_1(z)(2-\beta_2(z))/\beta_2(z)$. 
}.

\begin{defi}
    We say that a one-parameter family of Schottky groups satisfies condition $(\bigstar)$ if 
\begin{itemize}
\item $\lambda_j(z)$ vanishes at $z=0$ for all $j=1,\dots, g$ and 
\item $K_{\min} > 0$.
\end{itemize} 
\end{defi}
For the explanation this condition using non-Archimedean norm on $\Ct$, please see \cref{rem:schottky-laurent}.

In this setting, Dang and Mehmeti showed that 
\begin{thm}
There exists $d_0 > 0$ with 
$$\delta_{\Gamma_z} = \textrm{dim}(L(\Gamma_{z})) \sim d_0/\log(1/|z|)$$ 
as $z\to 0$ whenever the condition $(\bigstar)$ holds.
\end{thm}

\begin{rem} Similar results for Lyapunov exponents at the place of Hausdorff dimension were obtained by Favre \cite{favreDEGENERATIONENDOMORPHISMSCOMPLEX2020} (and, more recently, an error term was derived by Ingram et al. \cite{ingramAsymptoticSumsLyapunov2022} in arithmetic situations). 
\end{rem}

Let $\M(\D)$ is the ring of meromorphic function on $\D$ with possible pole at $0$.
From a Koebe coordinate $(\alpha(z),\beta(z),\lambda(z))$, we get a matrix $\gamma_z=\begin{pmatrix}
   \alpha(z) & \beta(z) \\ 1 & 1
\end{pmatrix}\begin{pmatrix}
   1 & 0 \\ 0 & \lambda(z)
\end{pmatrix}\begin{pmatrix}
   \alpha(z) & \beta(z) \\ 1 & 1
\end{pmatrix}^{-1}$ in $\GL_2(\M(\D))$.  
In order to simplify later computations, we want a matrix in $\SL_2(\M(\D))$. The only difficulty is that $\sqrt{\lambda(z)}=\sqrt{t^{-n}f(z)}$, for some $n\in\Z$ and holomorphic $f(z)$ nonzero at $0$, may not be meromorphic. We can do a change of variable. Let $z=w^2$ for $w\in \D$, then due to $\D$ simply connected, the square root $\sqrt{\lambda(w^2)}=w^{-n}\sqrt{f(w^2)}$ is still meromorphic. Let $\beta_w=\begin{pmatrix}
   \alpha(w^2) & \beta(w^2) \\ 1 & 1
\end{pmatrix}\begin{pmatrix}
   w^n(f(w^2))^{-1/2} & 0 \\ 0 & w^{-n}(f(w^2))^{1/2}
\end{pmatrix}\begin{pmatrix}
   \alpha(w^2) & \beta(w^2) \\ 1 & 1
\end{pmatrix}^{-1} $, which is in $\SL_2(\M(\D))$. For $z=w^2$, from the definition we have $\ell(\gamma_z)=\ell(\beta_w) $ and $2\ell^{na}(\gamma)=\ell^{na}(\beta)$, where $\ell^{na}$ will be defined in \cref{defi:na length} and with respect to $\Ct$ and $\C(\!(w)\!)$, respectively. Denote the Schottky group generated by $\beta_w$ as $\Delta_w$. Then the corresponding Selberg zeta functions and Ihara zeta functions satisfy
\[ Z(\Gamma_z,s/\log(1/|z|))=Z(\Delta_w,s/2\log(1/|w|))\text{ and }Z_I(\Gamma,s)=Z_I(\Delta,s/2). \]
We can obtain the result of $\Gamma_z$ from that of $\Delta_w$.
Hence, up to a possible change of variable $z=w^2$, the family of Schottky groups in $\PGL_2(\M(\D))$ can be replaced by a family in $\SL_2(\M(\D))$.
We don't give the detailed results in $\PGL_2$. The readers are encouraged to write for themselves.

\subsection{Non-Archimedean Schottky groups}
\label{sec:na-schottky}

Let $(k,|\cdot|_k)$ be a complete normed non-Archimedean field. See \cite{poineauSchottkySpacesUniversal2022} for more details of this part. 
\begin{defi}
    For $n\in\N$, the \emph{Berkovich analytification $\A_k^{n,an}$} of the $n$-dimensional affine space $\A_k^n$ is the set of multiplicative seminorms on the polynomial ring $k[T_1,\cdots,T_n]$ which extend the norm on $k$. The topology on $\A_k^{n,an}$ is the coarsest topology such that for any polynomial $P$ in $k[T_1,\cdots, T_n]$, the evaluation map from $\A_k^{n,an}$ to $\R$ given by $|\cdot|_x\mapsto |P|_x$ is continuous.

    The \emph{Berkovich analytification} $\P_k^{1,an}$ of the projective line $\P_k^1$ is constructed by gluing two copies of $\A_k^{1,an}$. More precisely, let $|\cdot|_x$ and $|\cdot|_y$ be two seminorms of two copies. They are equivalent if the map $T\rightarrow S^{-1}$ maps one seminorm on $k[T,T^{-1}]$ to the other on $k[S,S^{-1}]$. 
\end{defi}

\begin{rem}

When $k=\mathbb{C}$, $\mathbb{A}^{n,an}_{\mathbb{C}}$ is homeomorphic to the usual analytic $\mathbb{A}^{n}_{\mathbb{C}}$, and hence $\mathbb{P}^{1,an}_{\mathbb{C}}$ is homeomorphic to $\mathbb{P}^{1}_{\mathbb{C}}$. When $k$ is non-Archimedean, $\P^{1,an}_k$ has the structure of a real tree.
\end{rem}

\begin{exam}
Set $n=1$. When $k$ is non-Archimedean, for $a\in k$ and $r\in \mathbb{R}_{\geq 0}$, we can define the point $\eta_{a,r}\in \mathbb{A}^{1,an}_k$ by $\sum_{n}a_n(T-a)^n\mapsto \max_n|a_n|_k r^n$. In this setting, $a\in k$ corresponds to $\eta_{a,0}\in \mathbb{A}^{1,an}_k$ and the element $\eta_{0,1}\in \mathbb{A}^{1,an}_k$ is called the \emph{Gauss point}.
\end{exam}

\begin{defi}
An open disk (resp. a closed disk) in $\mathbb{P}^{1,an}_k$ is an open subset (resp. a closed subset) isomorphic to a set 
\begin{align*}
D^-(a,r)&=\{x\in \mathbb{A}^{1,an}_k| |T-a|_x<r\}\\
(\text{resp.}\,\,D^+(a,r)&=\{x\in \mathbb{A}^{1,an}_k| |T-a|_x\leq r\} )
\end{align*}
where $a\in k$ and $r\in \mathbb{R}_{>0}$, and $\mathbb{A}^{1,an}_k$ is an affine chart of $\mathbb{P}^{1,an}_k$.

The \textit{Shilov boundaries} of both $D^{-}(a,r)$ and $D^{+}(a,r)$ are defined to be the singleton $\{\eta_{a,r}\}$.
\end{defi}

The action of an element $\gamma=\begin{pmatrix}
    a & b \\ c & d
\end{pmatrix}$ of $\PGL_2(k)$ on $\A_k^{1,an}-\{-d/c\}$ is given by
\[ |P(T)|_{\gamma(x)}=\left|P\left(\frac{aT+b}{cT+d}\right)\right|_x=\frac{|P_1(T)|_x}{|P_2(T)|_x} \]
for $P,P_1,P_2\in k[T]$,  and $P_1(T)/P_2(T)=P(\frac{aT+b}{cT+d})$. Moreover, it is possible to naturally extend this formula to obtain an action of $\gamma$ on $\mathbb{P}^{1,an}_k$: see \cite[II.1.3]{poineauBerkovichCurvesSchottky2021} for more details. 

An element $\gamma\in \PGL_2(k)$ is called \emph{loxodromic} of given a representative in $\GL_2(k)$, its eigenvalues in an algebraic closure $k^{\mathrm{alg}}$ of $k$ have different absolute values (by \cite[II.1.4]{poineauBerkovichCurvesSchottky2021}, the eigenvalues are in $k$). Schottky groups of $\PGL_2(k)$ (\cite[Definition 3.5.1]{poineauSchottkySpacesUniversal2022}) are defined similarly to Schottky groups of $\PGL_2(\C)$
(\cref{defi:schottky}) with the projective space $\P^1(\C)$ replaced by $\P^{1,an}_k$.  

\begin{defi} For a loxodromic element $\gamma=\begin{pmatrix}
    a & b\\ c & d
\end{pmatrix}\in \PGL_2(k)$ with $c\neq 0$ and $\lambda\in \mathbb{R}_{>0}$, let
\[ D_{\gamma,\lambda}^-=\{ x\in \A_k^{1,an}|\ |(cZ+d)(x)|^2<\lambda |ad-bc|\},  \]
\[ D_{\gamma,\lambda}^+=\{ x\in \A_k^{1,an}|\ |(cZ+d)(x)|^2\leq \lambda |ad-bc|\},  \]
where we only consider an affine chart $\A_k^{1,an}$ of $\P^{1,an}_k$ and $Z$ is the variable of $k[Z]$. They are called open and closed \emph{twisted Ford disks} respectively. 
\end{defi}

\begin{rem}
    The original definition in \cite[Definition 3.5.5]{poineauSchottkySpacesUniversal2022} contains a typo. We need to define for all $x\in \A_k^{1,an}$ to have the desired property $\gamma D_{\gamma,\lambda}^-=\P_k^{1,an}-D_{\gamma^{-1},\lambda^{-1}}^+$. The definition relies on the choice of the affine chart $\A_k^{1,an}$. In all the argument, we fix the affine chart once for all, such that the $\infty$ is not in the limit set.  
\end{rem}

\begin{defi}[Schottky figure]
Let $\Gamma$ be a Schottky group in $\PGL_2(k)$. Let $\{\gamma_1,\cdots, \gamma_g\}$ be a set of free generators of $\Gamma$. Let $\mathcal{B}=\{D^+_{\gamma_i},D^+_{\gamma_i^{-1}}:i=1,\cdots g\}$ be a collection of disjoint closed disks in $\mathbb{P}^{1,an}_k$. Then $\mathcal{B}$ is called a \emph{Schottky figure} adapted to $\{\gamma_1,\cdots, \gamma_g\}$ if each $i\in \{1,\cdots, g\}$ and $\epsilon\in \{1,-1\}$,
\[ D^-_{\gamma_i^{\epsilon}}:=\gamma_i^{\epsilon}(\P^{1,an}_k-D^+_{\gamma_i^{-\epsilon}}) \]
is the maximal open disk inside $D^+_{\gamma_i^{\epsilon}}$.
\end{defi}

The following theorem is due to Gerritzen and is stated in the notations of Poineau--Turchetti (\cite[Theorem 3.5.9]{poineauSchottkySpacesUniversal2022}).
\begin{thm}[Gerritzen]\label{thm:schottky-figure}
     Let $\Gamma$ be a Schottky group of $\PGL_2(k)$. Then there exists a set of free generators $\{\gamma_1,\cdots,\gamma_g\}$ of $\Gamma$ and $\lambda_i\in(0,1)$ for $i=1,\cdots, g$ such that the collection of twisted Ford disks
    \[\mathcal{B}=\{D^+_{\gamma_i^{-1},\lambda_i^{-1}},D^+_{\gamma_i,\lambda_i}:i=1,\cdots, g \}  \]
    is a Schottky figure adapted to $\{\gamma_1,\cdots, \gamma_g\}$.
\end{thm}

\begin{rem} There is no analogue of this theorem in the Archimedean setting: indeed, it is known that some Kleinian Schottky groups are not classical, i.e., they cannot be described in terms of Schottky figures involving only round discs.   
\end{rem}

\begin{rem}\label{rem:schottky-laurent}
    A meromorphic family $(\Gamma_z)_{z\in\mathbb{D}^*}$ of Kleinian Schottky groups can be seen as a subgroup $\Gamma\subset \PGL_2(\Ct)$. In this context, as it is shown in \cite[Proposition 4.4.2]{poineauSchottkySpacesUniversal2022}, the condition $(\bigstar)$ is \emph{equivalent} to $\Gamma$ being a Schottky subgroup of $\PGL_2(\Ct)$ (with $\Ct$ equipped with the non-Archimedean norm $|\sum_{j\geq m}a_jt^j|_{na}=e^{-m}$ for $a_m\neq 0$).
\end{rem}

\paragraph{Mumford curve and its skeleton}
For an analogue of periodic geodesics in non-Archimedean case, we need to use the description of the limit set in $\P^{1,an}_k$. The advantage is that $\Pkan$ is a compact real tree, while $\P^1_k$ is discrete.

Let $\Gamma$ be a Schottky group of rank $g$ in $\PGL_2(k)$. Similarly to the special case where $k=\mathbb{C}$, a point $x\in \P^{1,an}_k$ is a \emph{limit point} if there exist $x_o\in \Pkan$ and a sequence $\{\gamma_n\}_n\subset \Gamma$ of distinct elements such that $\lim_{n\rightarrow \infty}\gamma_nx_0=x$.
Let $L(\Gamma)$ be the limit set of $\Gamma$ on $\Pkan$, that is the set of limit points. We have that $L(\Gamma)\subset \P^1_k$ \cite[Corollary II.3.14]{poineauBerkovichCurvesSchottky2021}. 

Moreover, let $O=\Pkan-L(\Gamma)$. It is shown that $O$ is a domain of discontinuity of $\Gamma$, i.e., it is $\Gamma$-invariant and the action of $\Gamma$ on $O$ is free and proper. The quotient space $X=\Gamma\backslash O$ is a Mumford curve. (See \cite[Theorem II.3.18]{poineauSchottkySpacesUniversal2022}.)
Let $\Sigma=O\cap\cup_{x,y\in L}[x,y]$ where $\Pkan$ has a structure of real tree and $[x,y]$ is the unique injective path from $x$ to $y$, see \cite[Section 2.5]{poineauSchottkySpacesUniversal2022} \cite[Proposition I.6.12]{poineauBerkovichCurvesSchottky2021}.\footnote{The path is the union of two intervals given by $\{\eta_{x,r},\ r\in[0,|x-y|_k]\}$ and $\{\eta_{y,r},\ r\in[0,|x-y|_k]\}$ with $\eta_{x,|x-y|_k}=\eta_{y,|x-y|_k}$,} Let $\Sigma_X=\Gamma\backslash\Sigma$, which is called the skeleton of the Mumford curve and is a finite graph of genus $g$. (See \cite[Notation 4.2.2]{poineauSchottkySpacesUniversal2022}) The skeleton $\Sigma_X$ can also be obtained by identifying the Shilov boundary points  
$p_i,q_i$ of $D_{\gamma_i}^+,D_{\gamma_i^{-1}}^+$, which is a Schottky figure adapted to $\{\gamma_1,\cdots,\gamma_g \}$.  
\begin{rem}
This skeleton $\Sigma_X$ is the analogue of the convex core of a hyperbolic manifold for Archimedean case. This can be seen from the fact that $\Pkan$ contains the Bruhat--Tits building (Bruhat--Tits tree in $\P^1_k$ case), which is the analogue of the simply connected hyperbolic manifold in Archimedean case.
\end{rem}

\paragraph{Metric structure}
Recall from \cite[Section I.8]{poineauBerkovichCurvesSchottky2021} \cite[Section 2.5]{poineauSchottkySpacesUniversal2022} the definition of multiplicative length. For $\alpha,\ \beta\in k$ and $0<r\leq s$, we define the \textit{additive} length of the elements $\eta_{\alpha,r},\eta_{\beta,s}$ by: \begin{itemize}
    \item if $\max\{r,s\}\geq |\alpha-\beta|_k$, $d_a(\eta_{\alpha,r},\eta_{\beta,s}):=\log(s/r)$; 
    \item if $ \max\{r,s\}< |\alpha-\beta|_k$, $d_a(\eta_{\alpha,r},\eta_{\beta,s}):=\log(|\alpha-\beta|_k/r)+\log(|\alpha-\beta|_k/s)$. 
\end{itemize}
This length is invariant under the action of M\"obius transformations. The tree $\Sigma$ consists of points of $\eta_{a,r}$ with $a\in k$ and $r\in\R_{>0}$.\footnote{The points like $\eta_{x,0}=x$ are not in $\Sigma=(\Pkan-L)\cap\cup _{x,y\in L}[x,y] $.} Using this length, we can define length on $\Sigma$. The length of the edge of the graph $\Sigma_X=\Gamma\backslash \Sigma$ comes from the quotient of this additive length. 

\begin{defi}\label{defi:na length}
    For a loxodromic element $\gamma$, we define the non-Archimedean length by
    \[ \ell^{na}(\gamma)=\log (|\tr \tilde\gamma|_k^2/|\det(\tilde\gamma)|_k), \]
    where $\tilde\gamma$ is a representative of $\gamma$ in $\GL_2(k)$.

    This definition also coincides with the logarithm of the attraction rate at the attracting fixed point of the loxordromic element. Hence a suitable analogue of length of periodic geodesics in the non-Archimedean case.
\end{defi}

The following lemma connects the length of loops in the graph with non-Archimedean length of elements in $\Gamma$.

\begin{lem}\label{lem:na length}
    Let $\gamma$ be a non-trivial loxodromic element in $\Gamma$ and $\alpha,\alpha'\in\P^1_k$ be its two fixed points. Then for any $x\in [\alpha,\alpha']-\{\alpha,\alpha' \}$, we have
    \[ d_a(x,\gamma x)=\ell^{na}(\gamma), \]
    where $[x,\gamma x]\subset[\alpha,\alpha']-\{\alpha,\alpha' \}\subset \Sigma$.

    For any $y=\eta_{\beta,s}\notin [\alpha,\alpha']$ with $0<s<\infty$, we have 
    \[ d_a(y,\gamma y)>\ell^{na}(\gamma). \]
\end{lem}
\begin{proof}
    Due to invariance under conjugation,
    up to conjugation, it is sufficient to consider $\gamma=\diag\{\lambda,1 \}$ with $|\lambda|_k>1$. Then the line $[\alpha,\alpha']=\{\eta_{0,r}, 0\leq r\leq \infty \} $. We have $d_a(x,\gamma x)=d_a(\eta_{0,r},\eta_{0,|\lambda|_kr})=\log|\lambda|_k=\ell^{na}(\gamma)$.

    The second statement follows from the fact that $\P^{1,an}_k$ is a tree and $\gamma$ preserves the distance $d_a$.
\end{proof}

\paragraph{Ihara zeta function}

The non-Archimedean case of Selberg zeta function is well studied. The analogue zeta function is called the Ihara zeta function for finite graphs. See \cite{iharaDiscreteSubgroupsTwo1966} and \cite{hortonWhatAreZeta2006}. 

Suppose $G=(V,E)$ is a finite graph. Let $E=\{e_1,\cdots, e_{2J}\}$ be the set of oriented edges of $G$ with $e_j$ and $e_{j+J}$ in opposite direction. Each edge has a positive length. A loop $P$ in $G$ is a finite sequence $P=(e_{i_1},\cdots ,e_{i_n})$ with the end point of $e_{i_j}$ equals the starting point of $e_{i_{j+1}}$ (we use the convention $e_{i_{n+1}}=e_{i_1}$). The length $\ell(P)$ of a path $P$ is the sum of the lengths of $e_{i_j}$. A path is non-backtracking if there are no consecutive edges that are inverse to each other, that is there is no $j$ such that $e_{i_{j+1}}=e_{i_j+J}$. A loop is primitive if it cannot be written as a multiple of another loop. Two loops only differing by starting point are defined to be in the same equivalence class of loops, denoted by $[P]$. Let $\calP$ be set of equivalence classes of primitive non-backtracking loops.

\begin{thm}[Ihara, Hashimoto, Bass]\label{thm-ihara}
        Let $G$ be a finite graph. Define the Ihara zeta function by
\[Z_I(G,s)=\prod_{[P]\in \calP }(1-e^{-s\ell(P)}). \]
Then the Ihara zeta function has an analytic extension to the whole complex plan.
\end{thm}
See \cref{prop:ZI} for a proof and a way to compute the Ihara zeta function.

Let $\Gamma$ be Schottky group in $\PGL_2(k)$. Recall the Mumford curve $\Sigma_X=\Gamma\backslash \Sigma$ of $\Gamma$. Since $\Sigma$ is a tree, which is simply connected, each non-backtracking loop $P$ of $\Sigma_X$ corresponds to a conjugacy class $[\gamma_P]$ in $\Gamma$. The length of each edge in $\Sigma_X$ is given by the length of its lift in the $\Sigma$. Moreover, from \cref{lem:na length}, we obtain that for non-backtracking loop $P$, we have $\ell(P)=\ell^{na}(\gamma_P)$. Therefore, we obtain
\begin{prop}\label{prop:na-Ihara}
Let $Z_I(\Gamma, s)$ be the Ihara zeta function for $\Gamma< \PGL_2(k)$\footnote{This function is also called the Ruelle zeta function.}, that is
\[Z_I(\Gamma, s)=\prod_{[\gamma] \in \calP}(1-e^{-s\ell^{na}(\gamma)}). \]
Then we have 
\begin{equation}\label{equ-na zeta}
    Z_I(\Gamma, s)=Z_I({\Sigma_X},s). 
\end{equation}

\end{prop}

\begin{rem}
If all the edges of $G$ are of length $1$, then\[ Z_I(s)=(1-e^{-2s})^{r-1}\det(1-A_Ge^{-s}+Q_Ge^{-2s}) \]where $A$ is the adjacent matrix of the graph $G$, $r$ is the rank of the fundamental group of $G$ and $Q_G$ is a diagonal matrix with $i$-th element equal to the degree minus 1 of the $i$-th vertex.

Therefore, if $|k^*|=r^{\Q}$ for some $r\in (0,1)$ (in particular, the case of the algebraic closure of $\Ct$), then we can use this formula to compute the Ihara zeta function.
\end{rem}

\subsection{Hybrid space} 
After Poineau--Turchetti \cite{poineauSchottkySpacesUniversal2022}, one can build a Schottky space $S_{g,A_r}^{an}$ over a Banach ring $A_r$.  

Set $r=1/e$. The Banach ring $A_r$ can be described as follows. We say that $f=\sum a_n t^n\in\Ct$  
belongs to $A_r$ if its \emph{hybrid norm} 
\begin{equation}\label{equ-hybrid-norm}
\|f\|_{\mathrm{hyb}} := \sum\max\{|a_n|_{\infty},|a_n|_0\} (1/e)^n
\end{equation}
is finite. (Here, $|.|_0$ is the trivial norm.)
Due to $\max\{|a_n|_\infty,|a_n|_0 \}\geq \sum |a_n|_\infty$, the elements $f\in A_r$ can be realized as meromorphic functions on $\Di_r$ with possible pole at $0$. 
\begin{defi}
The Berkovich affine space $\mathbb{A}^{3g-3,an}_{A_r}$ consists of elements $x$ which is a bounded multiplicative seminorm $|\cdot|_x$ on the ring $R=A_r[\alpha_3,\dots,\alpha_g,\beta_2,\dots,\beta_g,\lambda_1,\dots,\lambda_g]$, i.e., $|\cdot|_x$ satisfies 
\begin{itemize}
\item $|0|_x=0, |1|_x=1$;  
\item $|P+Q|_x\leq |P|_x+|Q|_x$ and $|P\cdot Q|_x = |P|_x|Q|_x$ for all $P,Q\in R$; 
\item $|f|_x\leq\|f\|_{\hyb}$ for all $f\in A_r$. 
\end{itemize}  
The topology is the weakest topology such that for all the element $P\in R$, the map from $\mathbb{A}^{3g-3,an}_{A_r}$ to $\R$ given by $|P|_x$ is continuous.
\end{defi}

When $n=0$, we denote by $\mathcal{M}(A_r)$ the space $\mathbb{A}^{0,an}_{A_r}$, and call it the \emph{Berkovich spectrum} of $A_r$. It is known that $\mathcal{M}(A_r)$ is naturally homeomorphic to the closed disc $\overline{\mathbb{D}}_r$ and, by restricting the seminorms to the base ring, one obtains a continuous morphism $\Pi:\mathbb{A}^{3g-3,an}_{A_r}\to\mathcal{M}(A_r)$.

 For each $x\in\mathbb{A}^{3g-3,an}_{A_r}$, we denote by $\mathcal{H}(x)$ the completion of the fraction field $R/\textrm{Ker}|\cdot|_x$. We call $x$ \emph{Archimedean} if $\mathcal{H}(x)$ is Archimedean, \emph{non-Archimedean} otherwise.

\subsubsection*{Continuous map $p: \overline{\mathbb{D}}_r\to \mathbb{A}^{3g-3,an}_{A_r}$} 

Let $p:\Di\to \mathbb{C}^{3g-3}$, $z\mapsto p(z)=(\underline{\alpha}(z),\underline{\beta}(z),\underline{\lambda}(z))$, be a continuous function in $\overline{\Di}_r^{*}$, holomorphic in $\Di_r^*$, and meromorphic at $0$. By abusing notation, we construct a map $p: \overline{\Di}_r\to \mathbb{A}^{3g-3,an}_{A_r}$ as follows: for any $F=\sum f_n(t) P_n(\underline{\alpha},\underline{\beta},\underline{\lambda})\in A_r[\underline{\alpha},\underline{\beta},\underline{\lambda}]$, we set  
\begin{equation}
\label{eqn:continuous norm map}
|F|_{p(z)}=\begin{cases}
\left|\sum f_n(z) P_n(\underline{\alpha}(z),\underline{\beta}(z),\underline{\lambda}(z))\right|_{\infty}^{1/\log (1/|z|)},\,\,\,\text{if}\,\,z\neq 0\\

\exp(-{\ord_{z=0}(\sum f_n(z)P_n(\underline{\alpha}(z),\underline{\beta}(z),\underline{\lambda}(z)))}),\,\,\,\text{if}\,\,z=0,
\end{cases}
\end{equation}
where $P_n(\underline{\alpha},\underline{\beta},\underline{\lambda})$ is a monomial of total degree $n$. The key advantage of this construction is that from \cite[Proposition A.4]{BoJo2017} the map $p:\overline{\Di}_r\to \mathbb{A}^{3g-3,an}_{A_r}$ is continuous. In the beginning, we deal with meromorphic functions. Using this hybrid construction, we don't have a singularity at $z=0$ and give a sense of the limiting behavior at $z=0$.

For $z\in \overline{\mathbb{D}}_r^*$, we have $\mathcal{H}(|\cdot|_{p(z)})=(\mathbb{C},|\cdot|_{\infty}^{1/\log(1/|z|)})$, and therefore $p(z)$ is an Archimedean point. For $z=0$, we have $\mathcal{H}(|\cdot|_{p(0)})=(\Ct, |\cdot|_{na})$, and therefore $p(0)$ is a non-Archimedean point.

Now, we build projective spaces where the Schottky groups attached to the points in Berkovich affine space act in a natural way. Given $x\in \mathbb{A}^{3g-3,an}_{A_r}$, let 
$$\mathbb{P}^{1,an}_{\mathcal{H}(x)} \simeq \mathbb{P}^1(\mathbb{C})$$ 
if the residue field $\mathcal{H}(x)$ is Archimedean, and $\mathbb{P}^{1,an}_{\mathcal{H}(x)} \simeq$ a tree otherwise. In this context, a point $x\in \mathbb{A}^{3g-3,an}_{A_r}$ belongs to the Schottky space $S^{an}_{g,A_r}$ if $0<|\lambda_j|_x<1$ for all $1\leq j\leq g$, the points $0,1,2, \alpha_3,\dots,\alpha_g,\beta_2,\dots,\beta_g$ are mutually distinct, and the matrices $\Psi(x)\subset \PGL_2(\mathcal{H}(x))$ associated to $x$ induce a Schottky group over $\mathbb{P}^{1,an}_{\mathcal{H}(x)}$. Interestingly enough, Poineau and Turchetti proved an analog of Hejhal's theorem by establishing that $S^{an}_{g,A_r}$ is an open path connected subset of $\mathbb{A}^{3g-3,an}_{A_r}$.

For the family of Schottky groups satisfying condition $(\bigstar)$, due to \cref{rem:schottky-laurent}, we know $p(0)\in S^{an}_{g,A_r}$. Due to openness of $S^{an}_{g,A_r}$,
we recover \cite[Proposition 6.1]{dangHausdorffDimension2024}
\begin{prop}
\label{prop:continuous map to Schottky moduli space}
    For the family of Schottky groups satisfying condition $(\bigstar)$, there exists $0<\eta\leq r$ such that the continuous map $p$ restricted in $\Di_\eta$ has image in $S_{g,A_r}^{an}$.
\end{prop}

\paragraph{Uniform Schottky basis}

For an element $\gamma\in\PGL_2(\Ct)$, whose entries are meromorphic on $\Di_r^*$ with a possible pole at $0$, we define a family of twisted Ford disks:
\begin{equation}\label{equ:fam-ford}
D_{\gamma,\lambda}^+(z)=\begin{cases}
&\{x\in\C|\ |c(z)x+d(z)|_\infty^{2/\log(1/|z|)}\leq \lambda |(ad-bc)(z)|_\infty^{1/\log(1/|z|)} \}\text{ if } \ z\neq 0, \\
& \{ x\in \A_{\Ct}^{1,an}|\ |(cZ+d)(x)|^2\leq \lambda |(ad-bc)(x)|\}\text{ if } \ z=0,
\end{cases} 
\end{equation}
where $c(z)$ means the complex value of the meromorphic function $c$ at $t$.

We state a result of uniform Schottky figures for all small parameters. This is a translation to this special case of \cite[Theorem 4.3.2 and  Corollary 4.3.3]{poineauBerkovichCurvesSchottky2021}. 
\begin{lem}\label{lem:relative ford}
    Let $\Gamma=\langle M_1,\cdots,M_g\rangle$ be a family of Schottky groups satisfying condition $(\bigstar)$. There exist $s\in (0,r)$ and an automorphism $\tau\in Aut(F_g)$ such that the following holds. Let
    $$
    N_1=\tau M_1,\cdots,N_g=\tau M_g\in \PGL_2(\Ct).$$
    There exist positive real numbers $\lambda_1,\cdots,\lambda_g$, such that the family of twisted Ford discs
    \[ (D_{N_1,\lambda_1}^+(z),\cdots,D_{N_g,\lambda_g}^+(z), D_{N_1^{-1},\lambda_1^{-1}}^+(z),\cdots, D_{N_g^{-1},\lambda_g^{-1}}^+(z)) \]
    is a Schottky figure adapted to the basis $(N_1,\cdots, N_g)$ for each $z\in \C$ with $|z|<s$.
\end{lem}

\begin{proof}

We explain the language used in \cite{poineauSchottkySpacesUniversal2022} and how to obtain the lemma.

Recall $S=S_{g,A_r}^{an}$ is the Schottky space.  Let $\calO( S)$ be the ring of analytic functions, which means locally the function can be uniformly approximated by rational functions.\footnote{We will only use meromorphic functions on $\mathbb D$, so we don't give detailed definition.} Let $\gamma=\begin{pmatrix}
    a & b \\  c & d
\end{pmatrix}\in \PGL_2(\calO(S))$ be a loxodormic element with $c\neq 0$ and $0<\lambda $. Let $\A_S^{1,an}$ be the affine line over $S$ with variable $Z$. In \cite{poineauSchottkySpacesUniversal2022}, relative twisted Ford discs is defined by 
\[ D_{\gamma,\lambda}^+=\{ x\in \A_S^{1,an}|\ |(cZ+d)(x)|^2\leq \lambda |(ad-bc)(x)|\}.  \]
We will prove that this gives the same family of twisted Ford disks as in \cref{equ:fam-ford}.

The projection map $\pi$ from $\A_S^{1,an}$ to $S$ is the restriction of the seminorm to the ring $A_r[\underline{\alpha},\underline{\beta},\underline{\lambda}]$. For any $y\in S$, the fiber $\pi^{-1}(y)$ is the affine line $\A_{\calH(y)}^{1,an}$.\footnote{The fiber $\pi^{-1}(y)$ is the set of seminorms on $R[Z]$ whose restriction to $R=A_r[\underline{\alpha},\underline{\beta},\underline{\lambda}]$ coincides with the seminorm $y$. Recall that $\calH(y)$ is the completion of the fractional field of $R/\ker(y)$. Hence this set of seminorms, $\pi^{-1}(y)$, is isomorphic to the set of multiplicative seminorms on $\calH(y)[Z]$ whose restriction to $\calH(y)$ coincides with $y$, which is exactly the affine line $\A_{\calH(y)}^{1,an}$.} Relative Ford discs are the union of twisted Ford discs on these fibers.

For the continuous map $p$, we know $p(0)$ is a non-Archimedean point in $S$. 
 Then we can apply \cite[Theorem 4.3.2 and  Corollary 4.3.3]{poineauBerkovichCurvesSchottky2021} to find a relative twisted Ford disks on an open neighbourhood $U$ of $p(0)$ in $S$, which is a Schottky figure adapted to a chosen basis. The new base is denoted by $    N_1=\tau M_1,\cdots,N_g=\tau M_g $ for some $\tau\in Aut(F_g)$ and with parameters $\lambda_1,\cdots,\lambda_g$. Moreover, for any $y\in U$, the restriction of relative twisted Ford discs to the fiber $\A_{\calH(y)}^{1,an}$ gives a Schottky figure on that fiber. Then we only need to prove this gives the twisted Ford disks in the statement of the Lemma.

For $\calH(y)$ Archimedean, the norm $(\calH(y),||_y)$ is always given by $(\C,||_\infty^{\epsilon(y)})$ with some $\epsilon(y)\leq 1$. Due to our choice of $p(z)$, we know that $\epsilon(p(z))=1/\log(1/|z|)$ for $z\neq 0$. Over $(\C,||_\infty^\epsilon)$, the analytification gives no new point. Hence a semi-norm $||_x\in\A_{\calH(p(z))}^{1,an}=\A_{\C,||_\infty^\epsilon}^{1,an}$ reads as
\[ |P(Z)(x)|=\left|(\sum a_nZ^n)(x)\right|=\left|\sum a_nZ(x)^n\right|_\infty^{1/\log(1/|z|)}, \]
for some unique $Z(x)\in \C$ determined by $x$ and for a polynomial $P\in \C[Z]$. Due to the openness of $U$ and the continuity of the map $p$, we can find $0<s<r$ such that $p(\mathbb D_s)$ is contained in $U$. Therefore, the twisted Ford discs at the fibre over $p(z)$ with $z\neq 0$ and $|z|<s$ reads as
\begin{align*}
 D_{\gamma,\lambda}^+(z)&=\{ x\in \A_{\calH(p(z))}^{1,an}|\ |(cZ+d)(x)|^2\leq \lambda |(ad-bc)(x)|\}\\
&=\{x\in\C|\ |c(z)x+d(z)|_\infty^{2/\log(1/|z|)}\leq \lambda |(ad-bc)(z)|_\infty^{1/\log(1/|z|)} \}. 
\end{align*}
The proof is complete. 
\end{proof}

We define the discs $D_{\frak{a},z}$ to be the Ford discs 
\begin{equation}
\label{eqn:disc convention}
D_{\frak{a},z}:=D^-_{\frak{a}^{-1},\lambda_{\frak{a}^{-1}}}(z)
\end{equation}
for any generator $\frak{a}$ in the generator set $\{N_1,\cdots, N_g, N_1^{-1},\cdots, N_g^{-1} \}$, where the inverse is because we want that the attracting fixed point $\frak{a}^+$ of $\frak{a}$ satisfies $\frak{a}^+\in D_\frak{a}$. For any element $\gamma=\frak{a}_{i_1}\cdots \frak{a}_{i_n}\in\Gamma$, we define
\[D_{\gamma,z}=\frak{a}_{i_1,z}\cdots \frak{a}_{i_{n-1},z}D_{\frak{a_{i_n}},z}. \]

We can further suppose the Schottky basis is uniformly bounded for $|z|\ll 1$. Indeed, the limit set and the Schottky figure of $\Gamma$ are in the unit ball $D(0,1)$ of the Berkovich affine line $\A_k^{1,an}$, which can be realized by suitably conjugating the Schottky group. Then, for small $|z|$, the Schottky figure are in a uniform compact set in $\C$ due to the centers of Schottky disks being meromorphic functions and their radii being functions of the form $\lambda^{\log(1/|z|)}/|c(z)|$, where $\lambda$ and $c(z)$ comes from the disc $D^-_{\frak{a}^{-1},\lambda_{\frak{a}^{-1}}}(z)$ for example $\frak a_z=\begin{pmatrix}
    a(z) & b(z)\\  c(z) & d(z)
\end{pmatrix} $ and $\lambda_{\frak a^{-1}}=\lambda$ .

\section{Convergence of zeta function}

Let $\Gamma<\SL_2(\M(\D))$ be a family of Schottky groups satisfying $(\bigstar)$.

\subsection{Convergence on half plane: Euler product}

\begin{thm}\label{thm:s large}
For $\Re s\gg 1$, we have that 
\begin{equation*}
\frac{Z(\Gamma_z,s/\log(1/|z|))}{Z_{I}(\Gamma,s)}\to 1\quad \text{as}\,\,z\to 0.
\end{equation*}

\end{thm}

The following lemma gives the convergence of the length of each conjugacy class. 
\begin{lem}\label{lem-conjugacy continuous}
    Let $[\gamma]$ be a non-trivial conjugacy class in $\Gamma$, then
    \[ \lim_{z\rightarrow 0}\frac{\ell(\gamma_z)/\log(1/|z|)}{\ell^{na}(\gamma)}=1. \]
\end{lem}
\begin{proof}

  Using the definition of non-Archimedean length (\cref{defi:na length}) and the assumption that $\Gamma<\SL_2(\mathbb{M}(\D))$, we have
    \[ |\Tr(\gamma)|_{na}=e^{\ell^{na}(\gamma)/2}.\]
Every non-trivial element $\gamma$ in $\Gamma$ is loxodromic in $\SL_2(\Ct)$, hence $\ell^{na}(\gamma)>0$. Since $\tr(\gamma_z)$ is meromorphic, we have 
 $\lim_{z\rightarrow 0}|\Tr(\gamma_z)|^{1/\log(1/|z|)}=|\Tr(\gamma)|_{na}$. The assumption that $\Gamma$ satisfies $(\bigstar)$ yields $\ell(\gamma_z)\to \infty$ as $z\to 0$. We obtain
 \begin{equation*}
    |\Tr(\gamma_z)|/e^{\ell(\gamma_z)/2}=(e^{\ell(\gamma_z)/2}+e^{-\ell(\gamma_z)/2})/e^{\ell(\gamma_z)/2}\rightarrow 1.  
    \end{equation*}
    We obtain
    \begin{equation*} \lim_{z\rightarrow 0}\frac{\ell(\gamma_z)/\log(1/|z|)}{\ell^{na}(\gamma)}=\lim_{z\rightarrow 0}\frac{\ell(\gamma_z)/\log(1/|z|)}{2\log|\Tr(\gamma_z)|/\log(1/|z|)}\frac{\log|\Tr(\gamma_z)|/\log(1/|z|)}{\log |\Tr(\gamma)|_{na}}\frac{2\log |\Tr(\gamma)|_{na}}{ \ell^{na}(\gamma)}=1.  \qedhere
    \end{equation*}
\end{proof}

We also need a uniform bound independent of $z$ for the lengths of geodesics. We first recall the uniform distortion estimates from Lemma 5.6 and Corollary 5.9 in \cite{dangHausdorffDimension2024}. Meanwhile, \cref{prop:continuous map to Schottky moduli space} and \cref{lem:relative ford} state that there exists $\eta\in (0,r)$ such that the map 
$p:\mathbb D_{\eta}\rightarrow \mathcal{S}_{g,A_r}^{an}$ is continuous, and $\{D_{\frak{a},z}\}$ is a uniform Schottky basis for any $z\in \D_{\eta}$. Hence, we can apply the results to obtain the following. 

\begin{lem}\label{lem:uniform-dis} There exist $\eta_0\in(0,r)$ and constants $R>0$, $c\in (0,1)$ for any $|z|\leq \eta_0$ and any non-trivial word $\gamma$, the radius $r_{\gamma,z}$ of $D_{\gamma,z}$ satisfies\footnote{Recall that $l(\gamma)$ stands for the word length: cf. Subsection \ref{ss.notations}.}
\begin{equation}\label{equ:r_gammay}
    r_{\gamma,z}\leq Rc^{l(\gamma)-1}.
\end{equation}
Moreover, there exists a constant $N>0$ such that for any $\gamma=\frak{a}\mathbf{b}$ with $\frak{a}\mathbf{b}$ reduced and $l(\frak{a})=1$ and $l(\mathbf{b})>N$,
\begin{equation}\label{equ:sup-gamma}
\frac{\sup_{x\in D_{\mathbf{b},z}}|\frak{a}'(x)|_z}{\inf_{x\in D_{\mathbf{b},z}}|\frak{a}'(x)|_z}\leq\exp\left(\frac{1}{4}r_{\mathbf{b},z}\right) .  
\end{equation}
\end{lem}
Here, the radius and the derivative are computed with respect to the norm $|\cdot|_{z}$. More precisely, for $z\neq 0$, $|\cdot|_z$ is a norm on $\C$ given by $|\cdot|_{z}:=|\cdot|_\infty^{1/\log(1/|z|)}$, and for $z=0$, $|\cdot|_z$ is a norm on $\Ct$ defined by $|\cdot|_{z}=|\cdot|_{na}$.

\begin{lem}\label{lem-unifrom-dis-ele}
Let $\eta_0\in (0,r)$, $R>0$, $c\in (0,1)$ be the constants given in \cref{lem:uniform-dis}. Then 
    there exists $C>0$ such that for any word $\gamma=\mathbf{a}\mathfrak{b}$ with $\mathbf{a}\mathfrak{b}$ reduced and $l(\mathfrak{b})=1$ and $\mathbf{a}$ of arbitrary length, we have 
\begin{equation*}
\frac{\sup_{x\in D_{\frak{b},z}}|\mathbf{a}'(x)|_z}{\inf_{x\in D_{\frak{b},z}}|\mathbf{a}'(x)|_z}\leq C\exp(R/(1-c)) \quad \text{for any }\,\,z\in \D_{\eta_0}.
\end{equation*}
\end{lem}

\begin{proof}
Assume $l(\mathbf{a})=n$ and write    
$\mathbf{a}=\frak{a}_{i_{1}}\cdots\frak{a}_{i_{n}}$.  Recall that $\mathfrak{a}_{i_j}\cdots \mathfrak{a}_{i_n}D_{\frak{b},z}=D_{\mathfrak{a}_{i_j}\cdots \mathfrak{a}_{i_n}\frak{b},z}$ for any $z\in \D_{\eta_0}$.
We estimate the distortion of $\mathbf{a}$: 
\begin{align*}
    \frac{\sup_{x\in D_{\frak{b},z}}|\mathbf{a}'(x)|_z}{\inf_{x\in D_{\frak{b},z}}|\mathbf{a}'(x)|_z}&\leq  \frac{\sup_{x\in D_{\frak{a}_{i_{2}}\cdots \frak{a}_{i_{n}}\frak{b},z}}|\frak{a}_{i_{1}}'(x)|_z}{\inf_{x\in D_{\frak{a}_{i_{2}}\cdots \frak{a}_{i_{n}}\frak{b},z}}|\frak{a}_{i_{1}}'(x)|_z}\cdot\frac{\sup_{x\in D_{\frak{a}_{i_{3}}\cdots \frak{a}_{i_{n}}\frak{b},z}}|\frak{a}_{i_{2}}'(x)|_z}{\inf_{x\in D_{\frak{a}_{i_{3}}\cdots \frak{a}_{i_{n}}\frak{b},z}}|\frak{a}_{i_{2}}'(x)|_z}\cdots \frac{\sup_{x\in D_{\frak{b},z}}|\frak{a}_{i_n}'(x)|_z}{\inf_{x\in D_{\frak{b},z}}|\frak{a}_{i_n}'(x)|_z}
    \\
    & \lesssim \exp\left(\sum_{j< n-N}\frac{1}{4}r_{\frak{a}_{i_{j+1}}\cdots \frak{a}_{i_{n}}\frak{b},z}\right) \lesssim \exp \left(\frac{1}{4}R/(1-c) \right)
\end{align*}
Here, to obtain the second and third inequalities, for $j$ satisfying $n-j>N$, we estimate the distortion of $\frak{a}_{i_j}$ by applying \cref{equ:sup-gamma} and then \cref{equ:r_gammay}, and 
 for $j$ satisfying  $n-j\leq N$, we bound the distortion of $\mathfrak{a}_j$ by a constant only depending on $\Gamma$ and $N$. 
\end{proof}

\begin{lem}\label{lem-uniform ell}
    There exists $c\in (0,1)$ depending on the Schottky family, such that for any conjugacy class $[\gamma]$ with $\ell^{na}(\gamma)>1/c$ and $|z|<c$, we have
    \[ \ell(\gamma_z)\geq c\log(1/|z|)\ell^{na}(\gamma).\]
\end{lem}
\begin{proof}
Let $l(\gamma)$ be the word length of $\gamma$ with the fixed generating set. Choose a $\gamma$ in the conjugacy class that is cyclically reduced, which means that the first and last letters of the word $\gamma$ are not inverse to each other.

Note that  $\gamma D_{\frak{a}_{i_1},z}=D_{\gamma\frak{a}_{i_1},z}$ and the
$\gamma$-attracting fixed point is inside $D_{\frak{a}_{i_{1}},z}$, combined with \cref{lem-unifrom-dis-ele}, we obtain
\begin{equation}\label{equ:gamma-attract}
-\log |\gamma'(\gamma\text{-attracting fixed point})|_z \geq \left|\log\left(\frac{r_{\gamma\frak{a}_{i_{1}},z}}{r_{\frak{a}_{i_{1}},z}}\right)\right|-\frac{R}{(1-c)}.
\end{equation}
As we required $\gamma$ to be cyclically reduced, we have $l(\gamma \frak{a}_{i_{1}})=l(\gamma)+1$. Hence by \cref{equ:r_gammay},
\begin{align*}
 \ell(\gamma_z)&=-\log(1/|z|)\log |\gamma'(\gamma\text{-attracting fixed point})|_z \geq \log(1/|z|)\left|\left|\log\left(\frac{r_{\gamma\frak{a}_{i_{1}},z}}{r_{\frak{a}_{i_{1}},z}}\right)\right|-\frac{R}{(1-c)}\right|\\
 &\geq \log(1/|z|) \left(l(\gamma)|\log c|-|\log r_{\frak{a}_{i_{1}},z}|-\log R-\frac{R}{(1-c)}\right)\gtrsim \log(1/|z|) l(\gamma),    
\end{align*}
where the last inequality is true if $l(\gamma)$ is large enough compared to these constants. For the finitely many $\gamma$ with   $l(\gamma)$ small, we use \cref{lem-conjugacy continuous} to get a uniform bound.
\begin{lem}
\label{lem:word-na0}
Let $\gamma$ be a cyclically reduced word. For the word length $l(\gamma)$ and the non-Archimedean length $\ell^{na}(\gamma)$, they satisfy $\ell^{na}(\gamma)\lesssim l(\gamma)$.
\end{lem}
\begin{proof}
    Since $\ell^{na}(\gamma)$ is the least translation length for the type of points $\eta_{\alpha,r}$ (\cref{lem:na length}), we have for the Gauss point $o=\eta_{0,1}$,
\begin{align*}
\ell^{na}(\gamma)&\leq d_a(o,\gamma o)\leq d_a(o,\frak{a}_{i_{1}}o)+d_a(\frak{a}_{i_{1}}o,\frak{a}_{i_{1}}\frak{a}_{i_{2}}o)+\cdots+d_a(\frak{a}_{i_{1}}\cdots \frak{a}_{i_{n-1}}o,\gamma o)\\
&\leq \sum_j d_a(o,\frak{a}_{i_j} o)\lesssim l(\gamma).  \qedhere 
\end{align*}
\end{proof}
The proof is complete by combining the above lemma.
\end{proof}
Since \cref{equ:gamma-attract} also works for non-Archimedean point $p(0)$, the same proof implies
\begin{lem}\label{lem:word-na}
Let $\Gamma$ be a non-Archimedean Schottky group.
For any cyclically reduced word $\gamma$, we have
    \[ \ell^{na}(\gamma)\gtrsim l(\gamma). \]
\end{lem}

\begin{proof}[Proof of \cref{thm:s large}]
Suppose $s\in \C$ with $\Re s\gg 1$. We have
    \begin{align}
    \label{eqn:expression of quotient}
        \frac{Z(\Gamma_z,s/\log(1/|z|))}{Z_I(\Gamma,s)}=\prod_{[\gamma]\in\calP}\frac{1-e^{-s/\log(1/|z|)\ell(\gamma_z)}}{ 1-e^{-s\ell^{na}(\gamma)}}\cdot \mathcal{R}(\gamma)
    \end{align}
    where
    \begin{equation*}
    \mathcal{R}(\gamma)=\begin{cases}
                \prod_{k\geq 1}(1-e^{-(s/\log(1/|z|)+k)\ell(\gamma_z)}) \quad \text{if}\,\, \Gamma_z< \SL_2(\R),\\
                \prod_{k_1+k_2\geq 1}(1-e^{-(s/\log(1/|z|)+k_1+k_2)\ell(\gamma_z)}e^{-i\theta_{\gamma}(k_1+k_2)})\quad \text{if}\,\,\Gamma_z< \SL_2(\C).
        \end{cases}
    \end{equation*}

We will use the inequality 
\begin{equation}
\label{elementary ineq}
1/(1-x)\leq \exp(2x)\,\,\, \text{for}\,\, \,0\leq x\leq 1/2.
\end{equation}
We estimate $|\mathcal{R}(\gamma)|^{-1}$. Using \cref{lem-uniform ell}, we have for $|z|$ small, if $\Gamma_z\subset\SL_2(\R)$, then
\begin{align*}
     |\mathcal{R}(\gamma)|^{-1}=\prod_{k\geq 1}|1-e^{-(s/\log(1/|z|)+k)\ell(\gamma_z)}|^{-1}\leq \exp\left(2\sum_{k\geq 1}e^{-kc\log(1/|z|)\ell^{na}(\gamma)} \right)=\exp\left(\frac{2e^{-c\log(1/|z|)\ell^{na}(\gamma)}}{1-e^{-c\log(1/|z|)\ell^{na}(\gamma)}}\right).
\end{align*}
If $\Gamma_z< \SL_2(\C)$, then we have
\begin{align*}
     &|\mathcal{R}(\gamma)|^{-1}\leq \prod_{k\geq 1}|1-e^{-(\Re s/\log(1/|z|)+k)\ell(\gamma_z)}|^{-(k+1)}\leq \exp\left(2\sum_{k\geq 1}(k+1)e^{-kc\log(1/|z|)\ell^{na}(\gamma)} \right)\\
     =&\exp\left(\frac{2}{(1-e^{-c\log(1/|z|)\ell^{na}(\gamma)})^2}-2\right)\leq \exp\left(\frac{4e^{-c\log(1/|z|)\ell^{na}(\gamma)}}{(1-e^{-c\log(1/|z|)\ell^{na}(\gamma)})^2}\right).
\end{align*}
Combing these two cases together, we have for $|z|$ small,  
\begin{align}
\label{control k part}
    &\prod_{[\gamma]\in\calP}|\mathcal{R}(\gamma)|^{-1}
    \leq\exp\left(8\sum_{[\gamma]\in \calP}e^{-c\log(1/|z|)\ell^{na}(\gamma)} \right).
\end{align}
 We first sum over $\{[\gamma] \in \calP: \,\ell^{na}(\gamma)=n\}$ and then sum over $n\in\N$. Due to \cref{lem:word-na}, there exists a constant $C$ only depending on the family such that
 \begin{equation}\label{equ:gamma_n}
     \#\{[\gamma]\in\calP:\, \ell^{na}(\gamma)=n\}\leq \#\{\gamma\,\text{cyclically reduced}:\, l(\gamma)\leq C'n\} \leq e^{Cn}.
    \end{equation}
Hence 
\begin{align*}
  \cref{control k part}  \leq \exp\left(\frac{8e^{-(c\log(1/|z|)-C)}}{1-e^{-(c\log(1/|z|)-C)}}\right),
\end{align*}
which converges to $1$ as $|z|$ tends to zero.

Now, we consider the other factor in \cref{eqn:expression of quotient}. Take $N\gg 1$. We separate into two parts: $\ell^{na}(\gamma)\leq N$ and $\ell^{na}(\gamma)>N$. For the first part, which is a product of finite terms, due to \cref{lem-conjugacy continuous},
\begin{align*}
    \prod_{\substack{[\gamma]\in\calP\\ \ell^{na}(\gamma)\leq N}}\frac{1-e^{-s/\log(1/|z|)\ell(\gamma_z)}}{ 1-e^{-s\ell^{na}(\gamma)}}\rightarrow 1.
\end{align*}
For the second part, due to \cref{lem-uniform ell}, we have 
\begin{align*}
    &\left|\log \prod_{\substack{[\gamma]\in\calP\\ \ell^{na}(\gamma)> N}}\left|\frac{1-e^{-s/\log(1/|z|)\ell(\gamma_z)}}{ 1-e^{-s\ell^{na}(\gamma)}}\right|\right|\leq   2\sum_{\substack{[\gamma]\in \calP\\\ell^{na}(\gamma)> N}}e^{-\Re s/\log(1/|z|)\ell(\gamma_z)}+e^{-\Re s\ell^{na}(\gamma)}\\
    \leq & 2\sum_{\substack{[\gamma]\in \calP\\ \ell^{na}(\gamma)> N}}e^{-\Re sc\ell^{na}(\gamma)}+e^{-\Re s\ell^{na}(\gamma)}
    \leq  2\sum_{n>N} e^{Cn}(e^{-n\Re sc}+e^{-\Re sn}) =\frac{4e^{-N(\Re sc-C)}}{1-e^{-N(\Re sc-C)}} ,
\end{align*}
where the first inequality is due to \cref{elementary ineq} and the last inequality is due to \cref{equ:gamma_n}.
Therefore, given $\Re s\gg 1$, we have that for $N$ sufficiently large, the product 
\[\prod_{\substack{[\gamma]\in\calP \\ \ell^{na}(\gamma)> N}}\frac{1-e^{-s/\log(1/|z|)\ell(\gamma_z)}}{ 1-e^{-s\ell^{na}(\gamma)}}\] is close to 1.
The proof is complete.    
\end{proof}

\subsection{Boundedness of zeta function through transfer operator}\label{sec:con transfer}
\subsubsection{Preliminaries on trace class operators}
We recall some preliminaries on trace class operators. See \cite[Appendix B]{DyZw2019} for more details.

Let $\cal{A}:\cal{H}_1\to\cal{H}_2$ be a compact operator between two Hilbert spaces. The singular values $\mu_{\ell}$ of $\cal{A}$ are defined to be the eigenvalues $\lambda_{\ell}$ of $(\cal{A}^*\cal{A})^{1/2}$ in decreasing order, i.e.
\begin{equation*}
    \mu_{\ell}(\cal{A}):=\lambda_{\ell}((\cal{A}^*\cal{A})^{1/2}),\quad \ell=0,1,2,\cdots.
\end{equation*}
We recall the minimax property of singular values:
\begin{equation}\label{e:sing-minimax}
    \mu_{\ell}(\cal{A})=\min\limits_{\rm{codim}(V, \cH_1)=\ell}\max\limits_{v\in V\setminus \{0\}}\frac{\|\cal{A}v\|_{\cal{H}_2}}{\|v\|_{\cal{H}_1}}.
\end{equation}
This follows directly from the minimax property for self-adjoint operators since $\|\cal{A}v\|^2_{\mathcal{H}_2}=(\cal{A}^*\cal{A}v,v)_{\mathcal{H}_2}$. As a corollary, for any compact operators $\mathcal{A}:\mathcal{H}_1\to \mathcal{H}_2$ and $\mathcal{B}:\mathcal{H}_1\to \mathcal{H}_2$, we have the inequality
\begin{equation}\label{e:sing-sum}
    \mu_{\ell}(\cal{A}+\cal{B})\leq \mu_{\ell_1}(\cal{A})+\mu_{\ell_2}(\cal{B}),\quad \forall\ \ell=\ell_1+\ell_2.
\end{equation}
\begin{defi}
    A compact operator $\cA:\cH_1\to \cH_2$ between two Hilbert spaces is a trace class operator if
    \begin{equation*}
        \|\cA\|_1:=\sum\limits_{\ell}\mu_{\ell}(\cA)<\infty.
    \end{equation*}
    $\|\cdot\|_1$ is called the trace norm.
\end{defi}
For a trace class operator $\cA:\cH\to \cH$ on a single space, one can define the trace by
\begin{equation}\label{eq:def-tr}
    \Tr(\cA):=\sum\limits_{j}\langle \cA e_j, e_j\rangle
\end{equation}
where $\{e_j\}$ is an orthonormal basis of $\cH$. One can verify that the definition does not depend on the choice of $\{e_j\}$. By Lidskii's Theorem \cite[Proposition B.31]{DyZw2019}), we have
\begin{equation*}
\Tr(\cA)=\sum_{\ell}\lambda_{\ell}(\cA).
\end{equation*}
The Fredholm determinant is defined by
\begin{equation*}
    \det(1-\cA):=\prod\limits_{\ell}(1-\lambda_{\ell}(\cA)).
\end{equation*}
We have the estimates
\begin{equation}\label{equ:trace-l1}
    |\Tr(\cA)|\leq \sum\limits_{\ell}|\lambda_{\ell}(\cA)|\leq \|\cA\|_1,\quad |\det(1-\cA)|\leq \prod\limits_{\ell}(1+\mu_{\ell}(\cA))\leq e^{\|\cA \|_1}.
\end{equation}

For a trace class operator $\mathcal{A}:\mathcal{H}_1
\to \mathcal{H}_2$ and a bounded operator $\mathcal{B}:\mathcal{H}_2\to \mathcal{H}_3$, it follows from the minimax property that $\mu_{\ell}(\cB\cA)\leq \|\cB\|_{\mathcal{H}_2\to \mathcal{H}_3} \mu_{\ell}(\cA)$ and thus
\begin{equation}\label{e:trace-prod}
    \|\cB\cA\|_{1}\leq \|\cB\|_{\mathcal{H}_2\to \mathcal{H}_3}  \|\cA\|_1.
\end{equation}
Moreover, it is shown in \cite[Eq (B.4.9)]{DyZw2019} that for a trace class operator $\cA:\cH_1\to \cH_2$ and and a bounded operator $\cB:\cH_2\to \cH_1$,
\begin{equation}\label{eq:trace-commute}
    \Tr(\cA\cB)=\Tr(\cB\cA).
\end{equation}
\subsubsection{Transfer operator: $\SL_2(\RR)$ case} 
To estimate the Selberg zeta function in the region where the infinite product expression does not converge, we employ the transfer operator $\cal{L}_s$. The observation is that the transfer operator $\cL_s$ is a trace class operator acting on the space of holomorphic functions with bounded $L^2$ norm, and its Fredholm determinant gives the Selberg zeta function. 

We first consider the case for a Schottky group $\Gamma < \SL_2(\RR)$ of rank $g\geq 2$.  
 Write $\Gamma=\langle \mathfrak{a}_{1}, \mathfrak{a}_{2},\cdots, \mathfrak{a}_{{2g}} \rangle$  with $\mathfrak{a}_{i}\mathfrak{a}_{i+g}=id$ for $i=1,2,\cdots, g$.  
By conjugating $\Gamma$ if necessary, there exist open discs $D_{\mathfrak{a}_1},D_{\mathfrak{a}_2},\cdots, D_{\mathfrak{a}_{2g}}$ {in $\mathbb{C}$} centered in $\RR$ such that their closures are disjoint, and for each $\mathfrak{a}\in \{\mathfrak{a}_1,\ldots, \mathfrak{a}_{2g}\}$,  we have $\mathfrak{a}(\hat{\CC}-\overline{D_{{\mathfrak{a}}^{-1}}})=D_\mathfrak{a}$ and $\mathfrak{a}(\hat{\RR}- I_{\mathfrak{a}^{-1}}) = I_\mathfrak{a}^\circ$, where $I_\mathfrak{a}:=\overline{D_\mathfrak{a}}\cap \RR$.

We define the Hilbert space $\mathcal{H}:=\oplus_{i=1}^{2g}\mathcal{H}(D_{\mathfrak{a}_i})$, where $\mathcal{H}(D_\mathfrak{a})=\{u\in L^2(D_{\mathfrak{a}}): u \text{ is holomorphic}\}$ is equipped with the norm $\|u\|_{\mathcal{H}(D_\mathfrak{a})}^2:=|D_\mathfrak{a}|^{-1}\|u\|_{L^2(D_\mathfrak{a})}^2$. The transfer operator $\cL_s:\mathcal{H}\to \mathcal{H}$ is defined as
\begin{equation}
\label{equ:transfer operator}
    \cL_s u(w):=\sum\limits_{\mathfrak{a}_i\neq \mathfrak{b}^{-1}} \mathfrak{a}_i'(w)^s u(\mathfrak{a}_i(w))  \quad \text{for}\quad w\in D_\mathfrak{b}\quad\text{with}\quad\mathfrak{b}\in \{\mathfrak{a}_1,\cdots, \mathfrak{a}_{2g}\}.
\end{equation}
Here note that the derivative satisfies $\mathfrak{a}_i'(w)>0$ for $w\in D_\mathfrak{b}\cap \RR$, and hence we define $\mathfrak{a}_i'(w)^s:=e^{s\log \mathfrak{a}_i'(w)}$ where we 
choose the branch of complex logarithm so that if $s\in \mathbb{R}$, then $\mathfrak{a}_i'(w)^s$ is the unique holomorphic function on $D_{\mathfrak{b}}$ satisfying $\mathfrak{a}_i'(w)^s>0$ for $w \in D_\mathfrak{b}\cap \RR$.

We introduce a few more notations. We use $\cW^n$ to denote words of length $n$ with respect to the generators  $\mathfrak{a}_1,\cdots,\mathfrak{a}_{2g}$, i.e. words of the form $\mathfrak{a}_{i_1}\mathfrak{a}_{i_2}\cdots \mathfrak{a}_{i_n}$ with  $\mathfrak{a}_{i_{j+1}}\neq \mathfrak{a}_{i_j}^{-1}$. For $\mathbf{a}=\mathfrak{a}_{i_1}\cdots \mathfrak{a}_{i_n}\in \cW^n$, we denote
\begin{equation*}
    {\mathbf{a}}(w)={\mathfrak{a}_{i_1}}({\mathfrak{a}_{i_2}}(\cdots {\mathfrak{a}_{i_n}}(w)))
\end{equation*}
and $\mathbf{a}'=\mathfrak{a}_{i_1}\cdots \mathfrak{a}_{i_{n-1}}$.

For a disc $D$ of radius $R$, and a real number $\rho>0$, we use $\rho D$ to denote the disc with the same center and of radius $\rho R$. 

The following lemma shows that $\cL_s$ is a trace class operator (see also \cite[Lemma 15.7]{borthwickSpectralTheoryInfiniteArea2016}).
\begin{lem}\label{lem:trace-compute-SL2(R)}
Suppose $\gamma:D_\mathfrak{b} \to \rho D_\mathfrak{a}$ is a holomorphic function between two discs with $0<\rho<1$ and $\mathfrak{a},\mathfrak{b}\in \{\mathfrak{a}_1,\ldots, \mathfrak{a}_{2g}\}$, then the pullback operator
\begin{equation*}
    \gamma^*u(w)=u(\gamma(w)): \cH(D_\mathfrak{a})\to \cH(D_\mathfrak{b})
\end{equation*}
is a trace class operator satisfying
\begin{equation*}
    \mu_{\ell}(\gamma^*)\leq C(1-\rho)^{-2}\rho^{\ell}\sqrt{\ell+1}.
\end{equation*}
Consequently, for any $s\in \mathbb{C}$, $\cL_s:\cH\to \cH$ is a trace class operator. 
\end{lem}
\begin{proof}
We may assume $D_\mathfrak{a}$ is centered at $0$ and take an orthonormal basis $\phi_n(w):=\sqrt{n+1}\left(\frac{w}{R}\right)^n$ of $\cH(D_\mathfrak{a})$ where $R$ is the radius of $D_\mathfrak{a}$. Then
    \begin{equation*}
        \|\gamma^*\phi_n\|_{\cH(D_\mathfrak{b})}^2=|D_\mathfrak{b}|^{-1}\int_{D_\mathfrak{b}}(n+1)\frac{|\gamma(w)|^{2n}}{R^{2n}} dm(w)\leq (n+1)\rho^{2n}.
    \end{equation*}
    By the minimax property \cref{e:sing-minimax}, we conclude
    \begin{equation*}
        \mu_{\ell}(\gamma^*)\leq\sum\limits_{n\geq \ell}\|\gamma^*\phi_n\|_{\cH(D_\mathfrak{b})}\leq \sum\limits_{n\geq \ell}\sqrt{n+1}\rho^{n}\leq C(1-\rho)^{-2}\rho^{\ell}\sqrt{\ell+1}.
    \end{equation*}
    Thus
    \begin{equation*}
    \sum\limits_{\ell}\mu_{\ell}(\gamma^*)\leq C(1-\rho)^{-2}\sum\limits_{\ell} \rho^{\ell}\sqrt{\ell+1}<\infty
    \end{equation*}
    which implies that $\gamma^*$ is a trace class operator. Since $\cL_s$ is a finite direct sum of such pullback operators composed with multiplication by holomorphic functions, we conclude $\cL_s$ is also a trace class operator using \cref{e:sing-sum} and \cref{e:trace-prod}.
\end{proof}
Since $\cL_s$ is a trace class operator, we can define its Fredholm determinant. We claim $\det(1-\cL_s)$ is exactly the Selberg zeta function, see \cite[Lemma 15.10]{borthwickSpectralTheoryInfiniteArea2016}. Here we include a proof for completeness but we will omit the proof later for similar cases. Note since $\cL_s$ is in trace class and depends holomorphically on $s\in\CC$, $\det(1-\cL_s)$ is an entire function. So this gives an alternative proof of holomorphic extension of the Selberg zeta function.
\begin{lem}
\label{lem:determinant formula SLR}
For $s\in \CC$, we have
    \begin{equation}
        \det(1-\cL_s)=Z(\Gamma,s).
    \end{equation}
\end{lem}
\begin{proof}
Since both sides are entire functions on $\CC$, it suffices to check the inequality for $\Re s\gg 1$.
   For $\Re s\gg 1$, we have 
    \begin{equation*}
        \det(1-\cL_s)=\exp\left(-\sum\limits_{n=1}^{\infty}\frac{1}{n}\Tr(\cL_s^n)\right).
    \end{equation*}
    On the other hand, 
    \begin{equation*}
        Z(\Gamma,s)=\exp\left(-\sum\limits_{n=1}^{\infty}\frac{1}{n}\sum\limits_{[\gamma],k}e^{-n(s+k)\ell(\gamma)}\right).
    \end{equation*}
    We again consider the orthonormal basis $\phi_k^\mathfrak{a}(w)=\sqrt{k+1}\left(\frac{w-w_\mathfrak{a}}{R_\mathfrak{a}}\right)^k$ where $w_\mathfrak{a}$ is the center of $D_\mathfrak{a}$ and $R_\mathfrak{a}$ is the radius of $D_\mathfrak{a}$. 
    We write
    \begin{equation*}
        \cL_s=\sum\limits_{\mathfrak{a}\neq \mathfrak{b}^{-1}\in \mathcal{W}}\cL_{s,\mathfrak{ab}}\quad\text{with}\quad\cL_{s, \mathfrak{ab}}:\mathcal{H}(D_{\mathfrak{a}})\to \mathcal{H}({D_{\mathfrak{b}}})\quad\text{given by}\quad\cL_{s,\mathfrak{ab}}u(w)={\mathfrak{a}}'(w)^su({\mathfrak{a}}(w))\quad \text{for}\,\,w\in {D_{\mathfrak{b}}}.
    \end{equation*}
    We have
    \begin{equation*}
        \cL_s^n=\sum\limits_{\mathfrak{a}_{i_1}\mathfrak{a}_{i_2}\cdots \mathfrak{a}_{i_{n+1}}\in \cW^{n+1}}\cL_{s,\mathfrak{a}_{i_1}\mathfrak{a}_{i_2}\cdots \mathfrak{a}_{i_{n+1}}}\quad \text{with}\quad \cL_{s,\mathfrak{a}_{i_1}\mathfrak{a}_{i_2}\cdots \mathfrak{a}_{i_{n+1}}}:=\cL_{s,\mathfrak{a}_{i_n}\mathfrak{a}_{i_{n+1}}}\cdots\cL_{s,\mathfrak{a}_{i_2}\mathfrak{a}_{i_3}}\cL_{s,\mathfrak{a}_{i_1}\mathfrak{a}_{i_2}}.
    \end{equation*}
        The Schwartz kernel of $\cL_{s,\mathfrak{a}_{i_1}\mathfrak{a}_{i_2}\cdots \mathfrak{a}_{i_{n+1}}}$ is given by
    \begin{equation*}
        \cL_{s,\mathfrak{a}_{i_1}\mathfrak{a}_{i_2}\cdots \mathfrak{a}_{i_{n+1}}}(w,w')=\sum\limits_{k}  (\mathfrak{a}_{i_1}\cdots \mathfrak{a}_{i_n})'(w)^s\phi_k^{\mathfrak{a}_{i_1}}(\mathfrak{a}_{i_1}\cdots \mathfrak{a}_{i_n}(w))\overline{\phi_k^{\mathfrak{a}_{i_1}}(w')},\quad w\in D_{\mathfrak{a}_{i_{n+1}}},\, w'\in D_{\mathfrak{a}_{i_1}}.
    \end{equation*}
    Its restriction to the diagonal is nonzero if and only if $\mathfrak{a}_{i_{n+1}}=\mathfrak{a}_{i_1}$. Set $\gamma=\mathfrak{a}_{i_1}\cdots \mathfrak{a}_{i_n}$ and $\cL_{s,\gamma}:=\cL_{s,\mathfrak{a}_{i_1}\cdots \mathfrak{a}_{i_{n+1}}}$. So
    $\cL_{s,\gamma} u(w)=\gamma'(w)^s u(\gamma(w))$.  
    A change of variables $w=gw'$ gives the operator $(g)^*\circ\cL_{s,\gamma}\circ (g^{-1})^*:\mathcal{H}(g^{-1}D_{\mathfrak{a}_{i_1}}) \to \mathcal{H}(g^{-1}D_{\mathfrak{a}_{i_{n+1}}})$  given by
    \begin{equation*}
      (g)^*\circ\cL_{s,\gamma}\circ (g^{-1})^*u(w')= \gamma'(gw')^s u(g^{-1}\gamma gw').
    \end{equation*}
    Let $\gamma_1=g^{-1}\gamma g$, and $\cL_{s,\gamma_1}:\mathcal{H}(g^{-1}D_{\mathfrak{a}_{i_1}}) \to \mathcal{H}(g^{-1}D_{\mathfrak{a}_{i_{n+1}}})$ given by $\cL_{s,\gamma_1}u(w')=\gamma_1'(w)^su(\gamma_1w')$. Then 
    \begin{equation*}
        (g)^*\circ\cL_{s,\gamma}\circ (g^{-1})^*u(w')=  \cB^{-1}\cL_{s,\gamma_1}\cB\quad \text{with}\quad \cB u(w')=(g'(w'))^su(w').
    \end{equation*}
    Therefore, by \cref{eq:trace-commute}, $\Tr\cL_{s,\gamma}=\Tr \cL_{s,\gamma_1}$. So we may choose the coordinate  
    so that $\mathfrak{a}_{i_1}\cdots \mathfrak{a}_{i_n}(w)=e^{-\ell(\mathfrak{a}_{i_1}\cdots \mathfrak{a}_{i_n})}w$, then 
    \begin{equation*}
        (\mathfrak{a}_{i_1}\cdots \mathfrak{a}_{i_n})'(w)^s=e^{-s\ell(\mathfrak{a}_{i_1}\cdots \mathfrak{a}_{i_n})}.
\end{equation*}
    We may assume the center of $D_{\mathfrak{a}_{i_1}}$ is $0$. Applying the definition of trace \cref{eq:def-tr} to the basis $\phi_k^{\mathfrak{a}_{i_1}}(w)$, we obtain
    \begin{equation*}
        \Tr \cL_{s,\mathfrak{a}_{i_1}\cdots \mathfrak{a}_{i_{n+1}}} = \frac{e^{-s\ell(\mathfrak{a}_{i_1}\cdots \mathfrak{a}_{i_n})}}{\pi R_{\mathfrak{a}_{i_1}}^2} \sum\limits_{k}  (k+1) \int_{|w|\leq R_{\mathfrak{a}_{i_1}}} \left(\frac{e^{-\ell(\mathfrak{a}_{i_1}\cdots \mathfrak{a}_{i_n})}w}{R_{\mathfrak{a}_{i_1}}}\right)^k\left(\frac{\bar{w}}{R_{\mathfrak{a}_{i_1}}}\right)^k dm(w)= \sum\limits_{k}e^{-(s+k)\ell(\mathfrak{a}_{i_1}\cdots \mathfrak{a}_{i_n})}.
    \end{equation*}
    Suppose the closed loop $\mathfrak{a}_{i_1}\cdots \mathfrak{a}_{i_n}=\mathbf{a}^m$ where $\mathbf{a}$ is primitive. Then
    \begin{equation*}
       \sum\limits_{n}\frac{1}{n} \Tr(\cL_s^n)=\sum\limits_{m}\frac{1}{m}\sum\limits_{[\mathbf{a}]\in\calP , k}  e^{-m(s+k)\ell({\mathbf{a}})}.
    \end{equation*}
    The proof is complete.
\end{proof}
For further applications, we also introduce modified transfer operators where we only use discs with word length $\geq N$. For $\mathbf{a}=\mathfrak{a}_{i_1}\cdots \mathfrak{a}_{i_n}\in\cW^n$, let
\begin{equation*}
    D_{\mathbf{a}}:=\mathfrak{a}_{i_1}\cdots \mathfrak{a}_{i_{n-1}}(D_{\mathfrak{a}_{i_n}}).
\end{equation*}
We define the modified transfer operator on $\cH=\cH_N:=\oplus_{\mathbf{b}\in\cW^N} \cH(D_{\mathbf{b}})$:
\begin{equation}
    \cL_{s,N} u(w) :=\sum\limits_{\mathfrak{a}\neq \mathfrak{a}_{i_1}^{-1}} {\mathfrak{a}}'(w)^s u({\mathfrak{a}}(w)) \quad \text{for}\quad w\in D_{\mathbf{b}}\quad \text{with}\quad\mathbf{b}=\mathfrak{a}_{i_1}\cdots\mathfrak{a}_{i_N}.
\end{equation}
When $N=1$, it agrees with the previous definition \cref{equ:transfer operator}. When $N\geq 2$, the modified transfer operator acts on a space of functions with a larger domain. We can still use the proof of \cref{lem:determinant formula SLR} to show that 
for any $N\geq 1$,
\begin{equation*}
    \det(1-\cL_{s,N})=Z(\Gamma,s).
\end{equation*}
In the computation, we need to use fixed points. By taking larger $N$, the fixed point of one loxodromic element is still contained in only one disc. The computations are the same.

Now we can state a much more precise estimate for the transfer operator.

\begin{prop}\label{prop-mul ls}
    Suppose $\Gamma < \SL_2(\RR)$ is a Schottky group with the notations above. Suppose for some $A,B>0$, $N\in\mathbb{Z}_{>0}$ and $\varphi\in [0,\pi]$, 
    we have for any $\mathfrak{a}\in \mathcal{W}$ and $\mathbf{b}=\mathfrak{a}_{i_1}\cdots \mathfrak{a}_{i_N}\in \mathcal{W}^N$ with $\mathfrak{a}_{i_1}\neq \mathfrak{a}^{-1}$, 
    \begin{itemize}
        \item $e^{-A}\leq |\mathfrak{a}'(w)|\leq e^{A}$ and $|\arg \mathfrak{a}'(w)|\leq \varphi\leq \pi$  for any  
        $w\in D_{\mathbf{b}}$;
        \item $\mathfrak{a}(D_{\mathbf{b}})\subset e^{-B}D_{\mathbf{a}}$ with $\mathbf{a}=\mathfrak{a}\mathbf{b}'$.
    \end{itemize}
    Then the singular values of $\cal{L}_{s,N}$ satisfy
    \begin{equation}
        \mu_{\ell}(\cal{L}_{s,N})\leq C(g,N,B)e^{(A|\Re s|+\varphi |\Im s|-B\ell)/C(g,N)}\sqrt{\ell+1}.
    \end{equation}
\end{prop}
\begin{proof}
    The transfer operator $\cal{L}_{s,N}$ is the direct sum of operators
    $$\cal{L}_{s,\mathbf{ab}}:\cH(D_{\mathbf{a}})\to \cH(D_\mathbf{b})\quad \text{given by}\quad\cal{L}_{s,\mathbf{ab}}u(w)={\mathfrak{a}}'(w)^{s}u({\mathfrak{a}}(w)),$$
    where $\mathfrak{a}\in \mathcal{W}$ and $\mathbf{b}=\mathfrak{a}_{i_1}\cdots \mathfrak{a}_{i_N}\in \mathcal{W}^N$ with $\mathfrak{a}_{i_1}\neq \mathfrak{a}^{-1}$ and $\mathbf{a}=\mathfrak{a}\mathbf{b}'$.
    We have
    \begin{itemize}
        \item $|{\mathfrak{a}}'(w)^s|\leq \max(|{\mathfrak{a}}'(w)|^{\Re s},|{\mathfrak{a}}'(w)|^{-\Re s})|\exp({|\arg {\mathfrak{a}}'(w)||\Im s|})\leq e^{A|\Re s|+\varphi|\Im s|}$;
        \item By \cref{lem:trace-compute-SL2(R)}, the singular values of the pullback operator $\mathfrak{a}^*:\mathcal{H}(D_{\mathbf{a}})\to \mathcal{H}(D_{\mathbf{b}})$ satisfy
        \begin{equation*}
        \mu_{\ell}({\mathfrak{a}}^*)\leq C(1-e^{-B})^{-2}e^{-B\ell}\sqrt{\ell+1}.
        \end{equation*}
    \end{itemize}
Then the proposition follows from \cref{e:sing-sum} and \cref{e:trace-prod}. 
\end{proof}

Recall from \cref{lem:relative ford}, for a family of Schottky groups satisfying condition $(\bigstar)$, we have a uniform Schottky basis.
\begin{prop}\label{prop:uniform-bdd-SL2(R)}
    Suppose we have a degenerating family of Schottky groups $\Gamma_z<\SL_2(\RR)$ satisfying $(\bigstar)$ with their Schottky discs uniformly bounded away from $\infty$. For any $N\geq 1$, $\mathfrak{a}\in \mathcal{W}$,
    $\mathbf{b}=\mathfrak{a}_{i_1}\cdots \mathfrak{a}_{i_N}\in \cW^N$ with $\mathfrak{a}\neq \mathfrak{a}_{i_1}^{-1}$, and sufficiently small $0<|z| \ll_N 1$, we have for any
    $w\in D_{\mathbf{b},z}$,  and the disc $D_{\mathbf{a},z}$ with
    $\mathbf{a}=\mathfrak{a}\mathbf{b}'$,
    \begin{align}
    \label{eqn:uniforma seperation SL2R}
     &   |\log|{\mathfrak{a}}'(w)||\lesssim_{N} \log(1/|z|),\nonumber\\ 
     & \mathfrak{a}(D_{\mathbf{b},z})\subset \frac{1}{10}D_{\mathbf{a},z}
    \end{align}
    in terms of the absolute norm. Moreover, 
    for any $M>0$, there exists $N(M)\geq 1$ such that if we require $\mathbf{b}\in \mathcal{W}^{N(M)}$ additionally, then 
    \begin{equation*}
        \quad |\arg \mathfrak{a}'(w)|\lesssim_M |z|^M.
    \end{equation*}
    Consequently, given any $C,M>0$, the family of zeta functions
    $Z(\Gamma_z,s/\log(1/|z|))$ with $0<|z|<1/e$
    is uniformly bounded (depending on $C$ and $M$) in the region
    \begin{equation*}
        |\Re s|\leq C, \quad |\Im s|\leq C|z|^{-M}\log(1/|z|).
    \end{equation*}
\end{prop}

\begin{proof}
By assumption, $\mathfrak{a}\in \SL_2(\mathbb{M}(\mathbb{D}))$, so we can write $\mathfrak{a}=\begin{pmatrix} a& b\\ c&d\end{pmatrix}$ with $a,b,c,d\in \mathbb{M}(\mathbb{D})$. Fix any $0<|z|\ll 1$.  The uniform separation $\mathfrak{a}(D_{\mathbf{b},z})\subset \frac{1}{10}D_{\mathbf{
    a},z}$ follows from \cref{cor:separtion}.

For any $w\in D_{\mathfrak{b},z}$, we have
   ${\mathfrak{a}}(w)=\frac{a_zw+b_z}{c_zw+d_z}$, where the lower subscript means to evaluate the functions at $z$.
   Then ${\mathfrak{a}}'(w)=1/(c_zw+d_z)^2$. The estimate $|z|^A\leq|{\mathfrak{a}}'(w)|$ follows from the fact that the Schottky discs are uniformly bounded away from $\infty$. The other estimate $|{\mathfrak{a}}'(w)|\leq |z|^{-A}$ follows from that the definition of uniform Schottky discs \cref{equ:fam-ford} and the convention \cref{eqn:disc convention} give
    \[\inf_{w\in D_{\bf{b},z}}|w+d_z/c_z|\geq d(\partial D_{\frak{a}^{-1},z}, \frak{a}^{-1}\infty)\geq \mathrm{diam}(D_{\frak{a}^{-1},z})/2.\]

    Finally, it follow from \cref{lem:uniform-dis} that by taking $N$ sufficiently large, we can make sure $\rm{diam}(D_{\mathbf{b},z})=O(|z|^{cN})$. The estimate $|\arg \mathfrak{a}'(w)|\leq |z|^M$ for $w\in D_{\mathbf{b},z}$ follows from the observation that
    \begin{equation*}
    \mathfrak{a}'(w)=(c_zw+d_z)^{-2}=c_z^{-2}((\Re w+d_z/c_z)+i\Im w)^{-2}
    \end{equation*}
    and $D_{b,z}\cap \mathbb{R}\neq \emptyset$.

    For the boundedness of $Z(\Gamma_z,s/\log(1/|z|))$, it suffices to assume $0<|z|\ll_N 1$ since otherwise, the statement follows from the continuity of the zeta function. Then the hypothesis of \cref{prop-mul ls} is satisfied with $A=A_0\log(1/|z|)$ where $A_0$ is fixed and $\varphi=O(|z|^M)$. We can apply the conclusion of \cref{prop-mul ls} and \cref{equ:trace-l1} to obtain the boundedness of $Z(\Gamma_z,s/\log(1/|z|)$.
\end{proof}

\subsubsection{Transfer operator: $\SL_2(\CC)$ case}

Let $\Gamma$ be a classical Schottky group in $\SL_2(\C)$ of rank $g\geq 2$. The discussion is similar but has several complications, which we detail below.

Write $\Gamma=\langle \mathfrak{a}_{1}, \mathfrak{a}_{2},\cdots, \mathfrak{a}_{{2g}} \rangle$  with $\mathfrak{a}_{i}\mathfrak{a}_{i+g}=id$ for $i=1,2,\cdots, g$.  
By conjugating $\Gamma$ if necessary, 
 there exist open discs $D_{\mathfrak{a}_1},D_{\mathfrak{a}_2},\cdots, D_{\mathfrak{a}_{2g}}\subset \CC$ such that their closures are disjoint, and  for each $\mathfrak{a}\in \{\mathfrak{a}_1,\ldots, \mathfrak{a}_{2g}\}$, we have $\mathfrak{a}(\hat{\CC}-\overline{D_{\mathfrak{a}^{-1}}}) = D_\mathfrak{a}$.

To study the Selberg zeta function, we identify $\CC$ with $\RR^2$, and given $\gamma=\begin{pmatrix} a& b\\ c& d\end{pmatrix}\in \SL_2(\mathbb{C})$,  
complexify the linear fractional transformation $\gamma:\mathbb{R}^2\to \mathbb{R}^2$:
\begin{equation}\label{eq:gamma-real}
\gamma(x,y)=\left(\Re\frac{a(x+iy)+b}{c(x+iy)+d}, \Im \frac{a(x+iy)+b}{c(x+iy)+d}\right),\quad (x,y)\in \RR^2
\end{equation} to 
\begin{equation}\label{eq:wt-gamma}
\wt{\gamma}(w_1,w_2)=\frac{1}{2}\left(\frac{a(w_1+iw_2)+b}{c(w_1+iw_2)+d}+\frac{\bar{a}(w_1-iw_2)+\bar{b}}{\bar{c}(w_1-iw_2)+\bar{d}},-i\frac{a(w_1+iw_2)+b}{c(w_1+iw_2)+d}+i\frac{\bar{a}(w_1-iw_2)+\bar{b}}{\bar{c}(w_1-iw_2)+\bar{d}}\right),\quad (w_1,w_2)\in \CC^2
\end{equation}
For each generator, suppose $D_\mathfrak{a}=\{|w-w_\mathfrak{a}|\leq R_\mathfrak{a}\}$. Let $\wt{D}_\mathfrak{a}$ be the polydisc defined by 
\begin{equation*}
    \wt{D}_\mathfrak{a}=\{(w_1,w_2)\in \CC^2 : \max(|w_1-\Re w_\mathfrak{a}|,|w_2-\Im w_\mathfrak{a}|)\leq R_\mathfrak{a}/3\},
\end{equation*}
and we also view $D_{\mathfrak{a}}\cap \wt{D}_{\mathfrak{a}}$ as a subset of $\wt{D}_{\mathfrak{a}}$.

We introduce the Hilbert space $\wt{\cH}=\oplus_{i=1}^{2g}\mathcal{H}(\wt{D}_{\mathfrak{a}_i})$ where $\mathcal{H}(\wt{D}_{\mathfrak{a}_i})=\{u\in L^2(\wt{D}_{\mathfrak{a}_i}): u \text{ is holomorphic}\}$ with the norm $\|u\|_{\mathcal{H}(\wt{D}_{\mathfrak{a}_i})}^2:=|\wt{D}_{\mathfrak{a}_i}|^{-1}\|u\|_{L^2(\wt{D}_{\mathfrak{a}_i})}^2$. Suppose $\mathfrak{a}_i$ maps $\wt{D}_{\mathfrak{b}}$ into a compact subset of $\wt{D}_{\mathfrak{a}_i}$ for any $\mathfrak{b}\neq \mathfrak{a}_i^{-1}$. The transfer operator $\cL_s:\wt{\cH}\to \wt{\cH}$ is defined as
\begin{equation}
\label{equ:transfer operator SL2C}
    \cL_s u(w_1,w_2):=\sum\limits_{\mathfrak{a}_i\neq \mathfrak{b}^{-1}} [\mathfrak{a}_i'(w_1,w_2)]^s u(\wt{\mathfrak{a}}_i(w_1,w_2))  \quad \text{for}\quad (w_1,w_2)\in \wt{D}_\mathfrak{b}\quad \text{with}\quad \mathfrak{b}\in \{\mathfrak{a}_1,\ldots,\mathfrak{a}_{2g}\}.
\end{equation}
Here $[\mathfrak{a}'_i(w_1,w_2)]$ is a holomorphic function in $\wt{D}_\mathfrak{b}$ such that $[\mathfrak{a}'_i(w_1,w_2)]= |\mathfrak{a}'_i(x+iy)|$ for $(w_1,w_2)=(x,y)\in D_\mathfrak{b}$. More precisely,
if $\mathfrak{a}_i=\begin{pmatrix} a& b\\ c& d\end{pmatrix}\in \SL_2(\mathbb{C})$, then 
\begin{equation}
\label{eqn:extension derivative}
    [{\mathfrak{a}}'_i(w_1,w_2)]=\frac{1}{(c(w_1+iw_2)+d)(\bar{c}(w_1-iw_2)+\bar{d})}.
\end{equation}
\cref{equ:transfer operator SL2C} requires choosing a branch of complex logarithm such that $\log [\mathfrak{a}'_i(w_1,w_2)]$ is a holomorphic function on $\wt{D}_{\mathfrak{b}}$ satisfying $\log [\mathfrak{a}'_i(w_1,w_2)]=\log |\mathfrak{a}'_i(x+iy)|$ for $(w_1,w_2)=(x,y)\in D_{\mathfrak{b}}\cap \wt{D}_{\mathfrak{b}}$. This is achievable because the inequality 
$|w_{\mathfrak{b}}+d/c|>R_{\mathfrak{b}}$ gives 
$|w_1+iw_2+d/c|>R_{\mathfrak{b}}/3$ and hence $[\mathfrak{a}'_i(w_1,w_2)]\neq 0$ for $(w_1,w_2)\in \wt{D}_{\mathfrak{b}}$. Define $[\mathfrak{a}'_i(w_1,w_2)]^s:=\exp ( s \log [\mathfrak{a}'_i(w_1,w_2)])$ on $\wt{D}_{\mathfrak{b}}$.

We show that the pullback operator $\wt{\gamma}^*$ is again a trace class operator.
\begin{lem}\label{lem:trace-compute-SL2(C)}
Suppose $\wt{\gamma}:\wt{D}_\mathfrak{b} \to \rho \wt{D}_\mathfrak{a}$ is a holomorphic function between polydiscs in $\CC^2$ with $0<\rho<1$ and $\frak{a}, \frak{b}\in \{\frak{a}_1,\ldots, \frak{a}_{2g}\}$. Then
the pullback operator
\begin{equation*}
    \wt{\gamma}^*u(w_1,w_2)=u(\wt{\gamma}(w_1,w_2)): \cH(\wt{D}_\mathfrak{a})\to \cH(\wt{D}_\mathfrak{b})
\end{equation*}
is a trace class operator satisfying
\begin{equation*}
    \mu_{\ell}(\wt{\gamma}^*)\leq C_{\rho}\rho^{\sqrt{\ell}}(\ell+1).
\end{equation*} 
\end{lem}

In the following, we will use the multi-index notation: 
\begin{equation*}
\alpha=(\alpha_1,\alpha_2)\in \mathbb{Z}^2_{\geq 0}, \quad |\alpha|=\alpha_1+\alpha_2, \quad \text{and}\quad \alpha!=\alpha_1!\alpha_2!.
\end{equation*}

\begin{proof}
We may assume $\wt{D}_\mathfrak{a}:=\{(w_1,w_2)\in \CC^2:|w_1|<R_1, |w|_2<R_2\}$.  
Consider the orthonormal basis $\phi_{\alpha}(w_1,w_2):=\sqrt{(\alpha_1+1)(\alpha_2+1))}(w_1/R_1)^{\alpha_1}(w_2/R_2)^{\alpha_2}$ of $\cH(\wt{D}_\mathfrak{a})$. Then
    \begin{equation*}
        \|\wt{\gamma}^*\phi_\alpha\|_{\cH(\wt{D}_\mathfrak{b})}^2=|\wt{D}_\mathfrak{b}|^{-1}\int_{\wt{D}_\mathfrak{b}}(\alpha_1+1)(\alpha_2+1) \frac{|\pi_1(\gamma (w_1,w_2))|^{2\alpha_1}}{R_1^{2\alpha_1}} \frac{|\pi_2(\gamma(w_1,w_2))|^{2\alpha_2}}{R_2^{2\alpha_2}} dm(w_1,w_2)\leq (|\alpha|+1)^2\rho^{2|\alpha|},
    \end{equation*}
    where $\pi_i:\CC^2\to \CC$ is to project on the $i$-th coordinate.
    By the minimax property \cref{e:sing-minimax}, we conclude
    \begin{equation*}
        \mu_{\ell}(\wt{\gamma}^*)\leq\sum\limits_{|\alpha|^2\geq \ell-10}\|\wt{\gamma}^*\phi_\alpha\|_{\cH(\wt{D}_\mathfrak{b})}\leq \sum\limits_{k^2\geq \ell-10}(k+1)^{2}\rho^{k}\leq C_{\rho}\rho^{\sqrt{\ell}}(\ell+1).
    \end{equation*}
      Thus
    \begin{equation*}
    \sum\limits_{\ell}\mu_{\ell}(\wt{\gamma}^*)\leq C_{\rho}\sum\limits_{\ell} \rho^{\sqrt{\ell}}(\ell+1)<\infty,
    \end{equation*}
    which implies $\wt{\gamma}^*$ is a trace class operator. 
\end{proof}
Similar to the $\SL_2(\RR)$ case, we have 
\begin{equation*}
\det(1-\cL_s)=Z(\Gamma,s)\quad \text{for any} \quad s\in \CC.
\end{equation*}
We refer to \cite{guillopeSelbergZetaFunction2004} for more details.

As in the $\SL_2(\RR)$ case, we also consider the modified transfer operators to obtain further applications. The precise construction is as follows. For $\mathbf{a}=\frak{a}_{i_1}\cdots \frak{a}_{i_n}\in \mathcal{W}^n$, define the polydisc $\wt{D}_{\mathbf{a}}$ in $\CC^2$ associated to $D_{\mathbf{a}}=D(w_{\mathbf{a}},R_{\mathbf{a}})\subset \CC$ by
\begin{equation*}
\wt{D}_{\mathbf{a}}=\{(w_1,w_2): \max(|w_1-\Re w_{\mathbf{a}}|, |w_2-\Im w_{\mathbf{a}}|)<R_{\mathbf{a}}/3\}.
\end{equation*}
For any $N\in \mathbb{N}$, the modified transfer operator $\cL_{s,N}$ on $\mathcal{H}=\mathcal{H}_N:=\oplus_{b\in \mathcal{W}^N}\mathcal{H}(\wt{D}_{\mathbf{b}})$ is defined by
\begin{equation*}
        \cL_{s,N} u(w_1,w_2) :=\sum\limits_{\mathfrak{a}\neq \bar{\mathfrak{a}}_{i_1}} [{\mathfrak{a}}'(w_1,w_2)]^s u(\wt{{\mathfrak{a}}}(w_1,w_2)) \quad \text{for}\quad (w_1,w_2)\in \wt{D}_{\mathbf{b}}\quad\text{with}\quad \mathbf{b}=\mathfrak{a}_{i_1}\cdots\mathfrak{a}_{i_N}.
\end{equation*}
Again we have for any $N\in \mathbb{N}$, 
\begin{equation*}
\det(1-\cL_{s,N})=Z(\Gamma,s) \quad \text{for any}\quad s\in \mathbb{C}.
\end{equation*}

We state a proposition similar to \cref{prop-mul ls} for $\SL_2(\CC)$.
\begin{prop}\label{p:sing-est-3d}
    Suppose $\Gamma<\SL_2(\CC)$ is a Schottky group with the notations above. Suppose for some $A,B>0$, $N\in \mathbb{Z}_{>0}$ and $\varphi\in [0,\pi]$, 
    we have for any $\frak{a}\in \mathcal{W}$ and $\mathbf{b}=\frak{a}_{i_1}\cdots \frak{a}_{i_N}\in \mathcal{W}^N$ with $\frak{a}_{i_1}\neq \frak{a}^{-1}$, 
    \begin{itemize}
        \item $e^{-A}\leq |[\mathfrak{a}'(w_1,w_2)]|\leq e^{A}$ and $|\arg[\mathfrak{a}'(w_1,w_2)]|\leq \varphi\leq \pi$ for any $(w_1,w_2)\in \wt{D}_{\mathbf{b}}$; 
        \item $\wt{\mathfrak{a}}(\wt{D}_{\mathbf{b}})\subset e^{-B}\wt{D}_{\mathbf{a}}$ where $\mathbf{a}=\mathfrak{a}\mathbf{b}'$.
    \end{itemize}
    Then the singular values of $\mathcal{L}_s$ satisfy
    \begin{equation*}
        \mu_{\ell}(\mathcal{L}_{s,N})\leq C(g,N,B)e^{(A|\Re s|+\varphi |\Im s|-B\ell^{1/2})/C(g,N)}(\ell+1).
    \end{equation*}
    Consequently, the operator $\cL_{s,N}$ is a trace class operator.
\end{prop}
\begin{proof}
    The proof is similar to \cref{prop-mul ls}. The operator $\cL_{s,N}$ is a direct sum of the following operators:
    Consider each component of $\mathcal{L}_{s,N}$:
    \begin{equation*}
        \mathcal{L}_{s,\mathbf{ab}}:\mathcal{H}(\wt{D}_\mathbf{a})\to \mathcal{H}(\wt{D}_\mathbf{b}),\quad \mathcal{L}_{s, \mathbf{ab}}u(w_1,w_2):=[\mathfrak{a}'(w_1,w_2)]^s u(\wt{\mathfrak{a}}(w_1,w_2)),\quad \mathbf{a}=\mathfrak{a}\mathbf{b}',
    \end{equation*}
    where $\frak{a}\in \mathcal{W}$, $\mathbf{b}=\frak{a}_{i_1}\cdots \frak{a}_{i_N}\in \mathcal{W}^N$ with $\frak{a}_{i_1}\neq \frak{a}^{-1}$, and $\mathbf{a}=\frak{a}\mathbf{b}'$.
    
    We have
    \begin{itemize}
        \item $|[\mathfrak{a}'(w_1,w_2)]^s|\leq \max\left(|[\mathfrak{a}'(w_1,w_2)]|^{\Re s},|[\mathfrak{a}'(w_1,w_2)]|^{-\Re s}\right)\exp({|\arg [{\mathfrak{a}}'(w_1,w_2)]||\Im s|})\leq e^{A|\Re s|+\varphi |\Im s|}$;
        \item By \cref{lem:trace-compute-SL2(C)}, we know
        \begin{equation*}
            \mu_{\ell}(\wt{\mathfrak{a}}^*)\leq C(B)e^{-B\ell^{1/2}}(\ell+1).
        \end{equation*}
    \end{itemize}
Then the theorem follows from \cref{e:sing-sum} and \cref{e:trace-prod}.
\end{proof}

Similar to \cref{prop:uniform-bdd-SL2(R)}, we have the following.
\begin{prop}\label{prop:uniform-bdd-SL2(C)}
    Suppose we have a degenerating family of Schottky groups $\Gamma_z<\SL_2(\CC)$ satisfying $(\bigstar)$ with their Schottky discs uniformly bounded away from $\infty$. For $N\geq 1$, 
    $\frak{a}\in \mathcal{W}$,
    $\mathbf{b}=\frak{a}_{i_1}\cdots \frak{a}_{i_N}\in \cW^N$ with $\frak{a}\neq \frak{a}_{i_1}^{-1}$, and sufficiently small $0<|z|\ll_N 1$, we have for any $(w_1,w_2)\in \wt{D}_{\mathbf{b},z}$, and the disc $\wt{D}_{\mathbf{a},z}$ with $\mathbf{a}=\frak{a}\mathbf{b}'$,
        \begin{align}
    \label{eqn:uniform separation SL2C}
    &\left|\log|[\mathfrak{a}'(w_1,w_2)]|\right|\lesssim \log(1/|z|)\nonumber\\
    &\wt{\mathfrak{a}}(\wt{D}_{\mathbf{b},z})\subset \frac{1}{10}\wt{D}_{\mathbf{a},z} 
    \end{align}
    in terms of the absolute norm. Moreover, for any $M>0$, there exists $N(M)\geq 1$ such that if we require $\mathbf{b}\in \mathcal{W}^{N(M)}$ additionally, then
    \begin{equation*}
        |\arg [\frak{a}'(w_1,w_2)]|\lesssim |z|^M.
    \end{equation*}
    Consequently, given any $C,M>0$,
     the family of zeta functions $Z(\Gamma_z,s/\log(1/|z|))$ with $0<|z|<1/e$ is uniformly bounded (depending on $C$ and $M$) in the region
    \begin{equation*}
        |\Re s|\leq C,\quad |\Im s|\leq C|z|^{-M}\log(1/|z|).
    \end{equation*}
\end{prop}

\begin{proof}
    By \cref{cor:separtion}, there exists $v>0$ such that for any 
    $N\in \mathbb{N}$, $\frak{a}\in \mathcal{W}$, $\mathbf{b}=\frak{a}_{i_1}\cdots \frak{a}_{i_N}\in \mathcal{W}^N$ with $\frak{a}\frak{a}_{i_1}\neq id$, and sufficiently small 
    $0<|z|\ll 1$, we have 
    \begin{equation}
    \label{eqn: uniform separation equation}
        \mathfrak{a}(D_{\mathbf{b},z})\subset |z|^{v} D_{\mathbf{a},z}.
    \end{equation}
    where $\mathbf{a}=\frak{a}\mathbf{b}'$.
    
    Fix $0<|z|\ll 1$. We first show \cref{eqn:uniform separation SL2C}.
    Suppose $D_{\mathbf{a},z}=D(w_{\mathbf{a},z},R_{\mathbf{a},z})$ and $D_{\mathbf{b},z}=D(w_{\mathbf{b},z},R_{\mathbf{b},z})$. Then $|\mathfrak{a}(w)-w_{\mathfrak{a},z}|\leq |z|^{v}R_{\mathbf{a},z}$ for $w\in D_{\mathbf{b},z}$.
     Using the Cauchy integral formula, we estimate the derivatives of $\frak a$ on $D_{\mathbf{b},z}$ at $w_{\mathbf{b},z}$: for any $k\in \mathbb{N}$,
    \begin{equation*}
        |\frac{d^k}{dw^k} \mathfrak{a}(w_{\mathbf{b},z})|=k!\left|\int_{\partial D_{\mathbf{b},z}}\frac{\frak{a}(w)}{(w-w_{\mathbf{b},z})^{k+1}} dw\right|=k!\left|\int_{\partial D_{\mathbf{b},z}}\frac{\frak{a}(w)-w_{\mathbf{a},z}}{(w-w_{\mathbf{b},z})^{k+1}} dw \right| \leq 2\pi\frac{k!}{R_{\mathbf{b},z}^k}|z|^{v}R_{\mathbf{a},z}.
    \end{equation*}
    We can then form the analytic continuation
    \begin{equation*}
    \begin{aligned}
        &\wt{\mathfrak{a}}(w_1,w_2)=\frak{a}(w_{\mathbf{b},z})+\\
        &\sum\limits_{\substack{\alpha\in\mathbb{Z}^2_{\geq 0}\\\text{with}\,\,|\alpha|>0}}\frac{1}{\alpha!}(w_1-\Re w_{\mathbf{b},z})^{\alpha_1}(w_2-\Im w_{\mathbf{b},z})^{\alpha_2}\left(\partial_x^{\alpha_1}\partial_y^{\alpha_2}\Re \mathfrak{a}(x+iy), \partial_x^{\alpha_1}\partial_y^{\alpha_2}\Im{\mathfrak{a}}(x+iy)\right)|_{x+iy=w_{\mathbf{b},z}}\in \CC^2.
    \end{aligned}
    \end{equation*}
    Combining the above two equations, we obtain that for any $(w_1,w_2)\in \wt{D}_{\mathbf{b},z}$,
    \begin{align*}
        &|\wt{\mathfrak{a}}(w_1,w_2)-w_{\mathbf{a},z}|\leq |z|^vR_{\mathbf{a},z}+4\pi\sum\limits_{\alpha}\frac{|\alpha|!}{\alpha!R_\mathbf{b}^{|\alpha|}} |z|^{v}R_{\mathbf{a},z}(R_{\mathbf{b},z}/3)^{|\alpha|}\\
        =&|z|^vR_{\mathbf{a},z}+4\pi\sum\limits_{k} \left(\frac{2}{3}\right)^k |z|^{v}R_{\mathbf{a},z}\leq (8\pi+1)|z|^{v}R_{\mathbf{a},z}. 
    \end{align*}
    Thus, we have \cref{eqn:uniform separation SL2C}. 

    Next, we estimate $\left|\log |[\frak{a}'(w_1,w_2)]| \right|$.
    We write $\mathfrak{a}=\begin{pmatrix}a &b\\c&d\end{pmatrix}$ with $a,b,c,d\in \mathbb{M}(\D)$. For any $w\in D_{\frak{b},z}$, we have
     $\mathfrak{a}(w)=\frac{a_zw+b_z}{c_zw+d_z}$ and $\mathfrak{a}'(w)=1/(c_zw+d_z)^2$, where the lower subscript means to evaluate the functions at $z$. By \cref{eqn:extension derivative}, we have for any $(w_1,w_2)\in \wt{D}_{\frak{b},z}$, 
    \begin{align}
    \label{eqn:derivative polydisc}
        [\mathfrak{a}'(w_1,w_2)]&=|c_z|^{-2} (w_1+iw_2-w_{\frak{b},z}+w_{\frak{b},z}+d_z/c_z)^{-1} (w_1-iw_2-\bar{w}_{\mathbf{b},z}+\bar{w}_{\mathbf{b},z}+\bar{d}_z/\bar{c}_z)^{-1}\nonumber\\
        &=|\frak{a}'(w_{\mathbf{b},z})|^2 \left(1+\frac{w_1+iw_2-w_{\mathbf{b},z}}{w_{\mathbf{b},z}+d_z/c_z}\right)^{-1}
        \left(1+\frac{w_1-iw_2-\bar{w}_{\mathbf{b},z}}{\bar{w}_{\mathbf{b},z}+\bar{d}_z/\bar{c}_z}\right)^{-1} 
    \end{align}
    Using $|w_{\mathfrak{b},z}+d_z/c_z|>R_{\mathfrak{b},z}$, and $\max(|w_1-\Re w_{\mathfrak{b}_z}|,|w_2-\Im w_{\mathfrak{b},z}|)\leq R_{\mathfrak{b},z}/3$, we have that both the absolute values of the second and third factors are bounded above by $5/3$ and below by $1/3$. Using the argument of $\SL_2(\mathbb{R})$ case (proof of \cref{prop:uniform-bdd-SL2(R)}), we have that there exists $A>0$, such that $|z|^{A}\leq |\frak{a}'(w_{\mathbf{b},z})|\leq |z|^{-A}$. Hence, we have established the estimate for $\left|\log|[\frak{a}'(w_1,w_2)]|\right|$. 

     We estimate $|\arg [\frak{a}'(w_1,w_2)]|$.
    It follows from \cref{cor:separtion} that $D_{\mathbf{b},z}\subset |z|^vD_{\frak{a}_{i_1},z}$ and  $R_{\mathbf{b},z}\leq |z|^{vN}R_{\frak{a}_{i_1},z}$. Then, we have
    \begin{equation*}
    \frac{|w_1+iw_2-w_{\mathbf{b},z}|}{|w_{\mathbf{b},z}+d_z/c_z|}\leq \frac{2}{3}|z|^v(1-|z|^v)^{-1},
    \end{equation*}
    which yields the estimate $|\arg [\frak{a}'(w_1,w_2)]|\lesssim |z|^{M}$ using \cref{eqn:derivative polydisc}.

    Finally, we use \cref{p:sing-est-3d} to obtain the boundedness of the zeta functions.
\end{proof}

\subsection{Convergence on the whole plane}

\begin{lem}\label{lem-hol conv}
    Let $\Omega$ be a connected open set in $\CC$ and $f(z)$, $f_n(z)$ be a family of holomorphic functions on $\Omega$. Suppose 
    \begin{itemize}
        \item $|f_n(z)|\leq C$ for any $z\in\Omega$ for some constant $C>0$;
        \item $f_n(z)\to f(z)$ pointwisely in some open set $\Omega'\subset \Omega$, as $n\to\infty$.
    \end{itemize}
    Then $f_n(z)\to f(z)$ uniformly on any compact subset of $\Omega$ as $n\to\infty$.
\end{lem}
\begin{proof}
    Since $f_n(z)$ is uniformly bounded, by Montel's theorem, any subsequence of $\{f_n\}$ has a subsequence that converges uniformly on any compact subset. Suppose the lemma is false, then there exists a subsequence of $f_n$ that converges locally uniformly but not to $f(z)$. This contradicts the second condition.
\end{proof}

    Therefore, the convergence part of \cref{thm:zeta-conv} follows from \cref{lem-hol conv} by combining \cref{thm:s large}, \cref{prop:uniform-bdd-SL2(R)} and \cref{prop:uniform-bdd-SL2(C)}.

\subsection{Uniform separation}

In this subsection, we verify the uniform separation property in \cref{eqn:uniforma seperation SL2R} of \cref{prop:uniform-bdd-SL2(R)} and \cref{eqn: uniform separation equation} of \cref{prop:uniform-bdd-SL2(C)} using the continuity of the hybrid model. The idea is to relate discs in Archimedean and non-Archimedean cases and then compute the discs under the action of $\Gamma$ in both cases. The key input is \cref{lem:relative ford}.

We need \cite[Lemma 3.2.2]{poineauBerkovichCurvesSchottky2021}\footnote{In \cite[Lemma 3.2.2]{poineauBerkovichCurvesSchottky2021}, they forget to take into account the $\epsilon$ in the Archimedean case. }
\begin{lem}\label{lem-gammaD}
    Let $k$ be a complete valued field. If $k$ is Archimedean, we suppose $(k,|\cdot|)$ is isometrically embedded into $(\C,|\cdot|_\infty^\epsilon)$ for some $\epsilon\in (0,1]$, where $|\cdot|_\infty$ is the usual absolute value on $\mathbb{C}$. Let $\rho>0$ and $\gamma=\begin{pmatrix}
        a & b\\ c & d
    \end{pmatrix}\in\PGL_2(k)$. If $\gamma D^+(\alpha,\rho)\subset \A_k^{1,an}$, then $|\alpha+(d/c)|>\rho$ and 
    \begin{equation}\label{equ:gamma ball}
       \gamma D^+(\alpha,\rho)=\begin{cases}
           & D^+\left((\frac{a}{c}-\frac{ad-bc}{c^2}\frac{\overline{\alpha+d/c}}{|\alpha+d/c|^{2/\epsilon}-\rho^{2/\epsilon}},\frac{|ad-bc|\rho}{|c|^2(|\alpha+d/c|^{2/\epsilon}-\rho^{2/\epsilon})^\epsilon}\right), 
           \text{ if }k \text{ is Archimedean;} \\
           & D^+\left(\frac{b+a\alpha}{c(\alpha+d/c)},\frac{|ad-bc|\rho}{|c|^2|\alpha+d/c|^2}\right), \text{ otherwise.} 
       \end{cases}
   \end{equation}
\end{lem}

Recall that for the non-Archimedean case, an open disc $D^-(\beta,\lambda)$ is contained in a closed disc $D^+(\alpha,\rho)$ with $\alpha,\beta\in k$ and $\lambda,\rho >0$ iff
\begin{equation*}
\lambda<\rho \quad \text{and}\quad |\alpha-\beta|<\rho. 
\end{equation*}

We recall the continuous map $p:\overline{\D}_r\to \mathbb{A}^{3g-3,an}_{A_r}$ in
the hybrid model: for $0<|z|<r$, 
the Archimedean norm $|\cdot|_z$ evaluates a function $f\in \Ct$ which is holomorphic on $\Di_r^*$ and meromorphic at $0$ 
by the formula
\[ |f|_z=|f(z)|_{\infty}^{1/\log(1/|z|)} ,\]
and we have as $|z|\rightarrow 0$
\[|f|_z\rightarrow |f|_{na}.  \]
A ball with center $\alpha$ and radius $\rho$ satisfies
\[ D_z^\pm(\alpha,\rho):=D_\infty^\pm(\alpha,\rho^{\log(1/|z|) })=D_\infty^\pm(\alpha,|z|^{\log(1/\rho) }).\]
Here, we use the lower subscript $z$ (resp. $na$ ) to emphasize that the ball is measured using the norm $|\cdot|_z$ (resp. $|\cdot|_{na}$ ).

\begin{lem}\label{lem:separation}

Let $\gamma\in\PGL_2(\Ct)$, and $\alpha,\beta\in\Ct$ which are holomorphic on $\Di_r^*$ and meromorphic at $0$. Suppose for some $\rho,\lambda>0$, we have  $\gamma D_{na}^-(\alpha,\rho)\subset D_{na}^+(\beta,\lambda)$. Then there exists $v>0$ such that for any $0<|z|\ll 1$, we have
\begin{equation*}
\gamma_zD_z^-(\alpha(z),\rho)\subset D_z^+(\beta(z),\lambda e^{-v}) 
\end{equation*}
and in terms of the absolute norm, we have
\begin{equation*}
\gamma_zD_\infty^-(\alpha(z),|z|^{\log(1/\rho)})\subset |z|^v D_\infty^+(\beta(z),|z|^{\log(1/\lambda)}). 
\end{equation*}
\end{lem}
\begin{proof}
Assume $\gamma=\begin{pmatrix} a& b\\ c&d\end{pmatrix}$ with $a,b,c,d\in \Ct$. 
From \cref{lem-gammaD} and $\gamma D_{na}^-(\alpha,r)\subset \A_k^{1,an} $ with $k=\Ct$, we have
\begin{equation}\label{equ:hypo}
    |\alpha+d/c|_{na}>\rho \quad \text{and}\quad \frac{|ad-bc|_{na}\rho}{|c|^2_{na}|\alpha+d/c|_{na}^2}<\lambda.
\end{equation}

For the Archimedean radius, let $\epsilon=1/\log(1/|z|)$, then by \cref{lem-gammaD} the radius of $\gamma_zD_z^-(\alpha(z),\rho)$ is equal to
\begin{equation*}
\frac{|ad-bc|_z\rho}{|c|_z^2(|\alpha+d/c|_z^{2/\epsilon}-\rho^{2/\epsilon})^\epsilon}=\frac{|(ad-bc)|_z\rho}{|c|_z^2|\alpha+d/c|_z^2}\frac{1}{(1-(\rho/|\alpha+d/c|_z)^{2/\epsilon})^\epsilon}. 
\end{equation*}
Due to $\rho/|\alpha+d/c|_z\rightarrow \rho/|\alpha+d/c|_{na}<1$ and $\epsilon\rightarrow 0$ as $|z|\rightarrow 0$, we obtain that 
\begin{equation}\label{equ:1-r}
    (1-(\rho/|\alpha+d/c|_z)^{2/\epsilon})^\epsilon\rightarrow 1, 
\end{equation}
as $|z|\rightarrow 0$.
Hence the Archimedean radius converges to the non-Archimedean radius, and by continuity, we obtain that for $|z|$ small, the Archimedean radius also satisfies
\begin{equation*}
 \frac{|ad-bc|_z\rho}{|c|_z^2(|\alpha+d/c|_z^{2/\epsilon}-\rho^{2/\epsilon})^\epsilon}<\lambda.  
\end{equation*}
In particular, there exists $v_0>0$ such that for $|z|$ small, we have a stronger inequality
\begin{equation*}
\frac{|ad-bc|_z\rho}{|c|_z^2(|\alpha+d/c|_z^{2/\epsilon}-\rho^{2/\epsilon})^\epsilon}<\lambda e^{-v_0}.  
\end{equation*}
In terms of the absolute norm, for $|z|$ sufficiently small, the Archimedean radius (the left-hand side of the following inequality) satisfies
\begin{equation*}
\left(\frac{|ad-bc|_zr}{|c|_z^2(|\alpha+d/c|_z^{2/\epsilon}-(r)^{2/\epsilon})^\epsilon}\right)^{1/\epsilon}\leq |z|^{v_0}\lambda^{1/\epsilon}. 
\end{equation*}

Next, we consider 
the center of the ball $\gamma_zD_z^+(\alpha,r)$ for $0\leq |z|\ll1$. For the non-Archimedean case (i.e., $z=0$), from the hypothesis, we have
$|\frac{b+a\alpha}{c(\alpha+d/c)}-\beta|_{na}<\lambda$. Using the convergence $|\cdot|_z\to |\cdot|_{na}$ as $|z|\to 0$, we obtain a constant $v_1>0$ such that for $0<|z|\ll 1$, we have
\begin{equation*}
\left|\frac{b+a\alpha}{c(\alpha+d/c)}-\beta\right|_{z}<e^{-v_1}\lambda,\quad \text{i.e.,}\quad \left|\frac{b+a\alpha}{c(\alpha+d/c)}-\beta\right|_{\infty}<|z|^{v_1}\lambda^{1/\epsilon}.
\end{equation*}
Using \cref{lem-gammaD},  the distance between the Archimedean centre $\delta_a$ and the non-Archimedean center $\delta_{na}$ (view them as functions of $z$) is given by
\begin{align*}
\delta_{na}-\delta_a=\frac{b+a\alpha}{c(\alpha+d/c)}-\left(\frac{a}{c}-\frac{ad-bc}{c^2}\frac{\overline{\alpha+d/c}}{|\alpha+(d/c)|_z^{2/\epsilon}-\rho^{2/\epsilon}}\right) &=-\frac{ad-bc}{c^2(\alpha+d/c)}\left(1-\frac{1}{1-(\rho/|\alpha+d/c|_z)^{2/\epsilon}}\right)\\
&=\frac{ad-bc}{c^2(\alpha+d/c)}\frac{(\rho/|\alpha+d/c|_z)^{2/\epsilon}}{1-(\rho/|\alpha+d/c|_z)^{2/\epsilon}}. 
\end{align*}
Its $|\cdot|_z$ norm equals (notice that for a real number $|x|_z=|x|^{\epsilon}$) 
\begin{equation}\label{equ:diff-cen}
\begin{split}
|\delta_{na}-\delta_a|_z&=\left|\frac{ad-bc}{c^2(\alpha+d/c)}\right|_z \left(\frac{\rho}{|(\alpha+d/c)|_z}\right)^2\frac{1}{(1-(\rho/|\alpha+d/c|_z)^{2/\epsilon})^\epsilon}\\
 &=\rho\left|\frac{ad-bc}{c^2(\alpha+d/c)^2}\right|_z \frac{\rho}{|(\alpha+d/c)|_z}\frac{1}{(1-(\rho/|\alpha+d/c|_z)^{2/\epsilon})^\epsilon}<\lambda, 
\end{split}
\end{equation}
where the inequality is due to \cref{equ:hypo} and \cref{equ:1-r}, and holds for $|z|$ small. As the estimate for the radius, we can replace $\lambda$ by $e^{-v_2}\lambda$ for some $v_2>0$ in the previous inequality, which holds
for $|z|$ small.
Combining the estimate of radius and distance between centers, we obtain for $|z|$ sufficiently small and any $x\in \gamma_zD^-_z(\alpha(z),\rho)$,
\begin{align*}
&|x-\beta|_z=|x-\beta|_{\infty}^{\epsilon}\leq (|x-\delta_{a}|_{\infty}+|\delta_{a}-\delta_{na}|_{\infty}+|\delta_{na}-\beta|_{\infty})^{\epsilon}\\
\leq &(|z|^{v_0}\lambda^{1/\epsilon}+|z|^{v_2}\lambda^{1/\epsilon}+|z|^{v_1}\lambda^{1/\epsilon})^{\epsilon}\leq 3^{\epsilon} \lambda e^{-v}<\lambda e^{-v/2},
\end{align*}
where $v=\min\{v_0,v_1,v_2\}$. This completes the proof.
\end{proof}
\begin{rem}
    From \cref{equ:diff-cen}, as $z\rightarrow 0$, the $|\cdot|_z$-norm of the difference of two centers does not go to zero. 
\end{rem}

We are ready to prove the main result of this part:
\begin{prop}\label{prop:separtion}
Suppose $\Gamma<\SL_2(\mathbb{M}(\mathbb{D}))$ is a family of Schottky groups satisfying $(\bigstar)$. There exists $v>0$ such that for any $0<|z|\ll 1$, the corresponding Schottky group and the Schottky basis given by \cref{lem:relative ford} satisfy
    \[ \mathfrak{a}_z(D_{\mathfrak{b},z})\subset |z|^v D_{\mathfrak{a},z},\]
   in terms of the absolute norm, for any generators $\mathfrak{a},\mathfrak{b}$ with $\mathfrak{ab}\neq id$. 
\end{prop}

\begin{proof}
Recall that
    \cref{lem:relative ford} gives a uniform Schottky basis for all $|z|$ small.
        In particular, for $z=0$, the Schottky basis for $\Gamma<\PGL_2(k)$ with $k=\Ct$ equipped with the norm $|\cdot|_{na}$ is given by
    \[\mathcal{B}=\{D^+_{\delta_i,\lambda_i},D^+_{\delta_i^{-1},\lambda_i^{-1}},\ i=1,\cdots, g \}, \]
   Here the closed discs are given as in \cref{equ:fam-ford}. Writing $\delta_i=\begin{pmatrix} a_i & b_i\\c_i&d_i\end{pmatrix}$, we have 
   \begin{align}
   \label{eqn:inclusion}
   &D^+_{\delta_i,\lambda_i}:=D_{na}^+\left(\delta_i^{-1}\infty, \lambda_i^{1/2}\left|\frac{(a_id_i-b_ic_i)}{c_i^2}\right|^{1/2}_{na}\right),\nonumber\\
  & \delta_i^{-1} D^{-}_{\delta_j,\lambda_j}\subset D^+_{\delta_i,\lambda_i} \quad \text{for any} \quad i\neq j.
   \end{align}

Then we apply \cref{lem:separation} to obtain the uniform separation. More precisely, take any $\delta_i=\frak{a}^{-1}$ and 
$\delta_j=\frak{b}^{-1}$ with $i\neq j$.  
For the non-Archimedean case $z=0$, the Schottky discs are given by 
\begin{equation}
\label{eqn:define a b}
D_{\frak{a},0}=D^-_{\delta_i,\lambda_i}\quad \text{and} \quad D_{\frak{b},0}=D^-_{\delta_j,\lambda_j}.
\end{equation}
The Schottky discs for the Archimedean case are given by balls (see \cref{equ:fam-ford} and \cref{eqn:disc convention})
\begin{equation*}
D_{\mathfrak{a},z}=D_z^-\left((\delta_i^{-1}\infty)(z), \lambda_i^{1/2}\left|\frac{(a_id_i-b_ic_i)}{c_i^2}\right|^{1/2}_z\right) \quad \text{and}\quad D_{\mathfrak{b},z}=D_z^-\left((\delta_j^{-1}\infty)(z), \lambda_j^{1/2}\left|\frac{(a_jd_j-b_jc_j)}{c_j^2}\right|^{1/2}_z\right) .
\end{equation*}

Set
\begin{equation*}
\lambda_{na}=\lambda_i^{1/2}\left|\frac{(a_id_i-b_ic_i)}{c_i^2}\right|^{1/2}_{na} \quad\text{and}\quad\rho_{na}=\lambda_j^{1/2}\left|\frac{(a_jd_j-b_jc_j)}{c_j^2}\right|^{1/2}_{na}.
\end{equation*}
Similarly, we introduce the notations $\lambda(z)$ and $\rho(z)$ by replacing $|\cdot|_{na}$ with $|\cdot|_z$.

We can apply \cref{lem:separation} to the discs $D_{\frak{b},0}$ and $D_{\frak{a},0}$ due to the property \cref{eqn:inclusion} and the convention \cref{eqn:define a b}, and use the convergence of the radii $\rho(z)\to \rho_{na}$ and $\lambda(z)\to \lambda_{na}$ as $z\to 0$. This yields a constant 
 $v'>0$ such that for any $0<|z|\ll 1$, we have
\begin{equation}
\label{eqn:convergence in z norm}
\frak{a}_z D^-_z\left((\delta^{-1}_j\infty)(z), \rho(z)\right)=(\delta_i^{-1})_z D^-_z\left((\delta^{-1}_j\infty)(z), \rho(z)\right)\subset D^-_z\left((\delta_i^{-1}\infty)(z),e^{-v'}\lambda(z)\right),
\end{equation}
and in terms of the absolute norm, we have
\begin{equation*}
\frak{a}_{z}  
D_{\mathfrak{b},z}\subset |z|^{v'}D_{\mathfrak{a},z}. 
\end{equation*}
%in terms of the absolute norm.

Repeating this argument for all possible pairs, we finish the proof.
\end{proof}
\begin{rem}
    This proposition also implies the radius estimate in \cref{lem:uniform-dis}.
\end{rem}

\begin{rem}\label{rem:real-schottky}
    If the Schottky group $\Gamma$ is indeed inside $\SL_2(\R(\!(t)\!))$, then for $z\in \R$ the discs $D_{\frak{a},z}$ are centered in $\R$ and the group $\Gamma_z$ preserves $\R\cup \infty$. Therefore, the intersections $I_{\frak{a},z}=D_{\frak{a},z}\cap\R$ give a Schottky figure for $\Gamma_z$ as a subgroup of $\SL_2(\R)$.
\end{rem}

\begin{cor}\label{cor:separtion}
Suppose $\Gamma<\SL_2(\mathbb{M}(\mathbb{D}))$ is a family of Schottky groups satisfying $(\bigstar)$. There exists $v_1>0$ such that for any $0<|z|\ll 1$ and for any $n\in \mathbb{N}$, the corresponding Schottky group and the Schottky basis given by \cref{lem:relative ford} satisfy
    \[ \mathfrak{a}_z(D_{\mathbf{b},z})\subset |z|^{v_1} D_{\mathbf{a},z},\]
    in terms of the absolute norm, where $\mathbf{a}=\mathfrak{a}\mathbf{b}'$ with $\frak{a}$
    a generator and $\mathbf{b}=\frak{b}_{i_1}\cdots \frak{b}_{i_n}\in \mathcal{W}^{n}$ satisfying $\mathfrak{ab}_{i_1}\neq id$.
    \end{cor}
\begin{proof}
 Notice that $D_{\mathbf{a},z}=\frak{a}\frak{b}_{i_1}\cdots \frak{b}_{i_{n-2}}D_{\mathbf{b}_{i_{n-1}},z},\ \mathfrak{a}_z(D_{\mathbf{b},z})=\frak{a}\frak{b}_{i_1}\cdots \frak{b}_{i_{n-2}}D_{\mathbf{b}_{i_{n-1}}\mathbf{b}_{i_n},z}$. 
The idea is to apply the distortion estimate of \cref{lem-unifrom-dis-ele} to 
\begin{equation*}
\gamma:=\frak{a}\frak{b}_{i_1}\cdots \frak{b}_{i_{n-2}},\quad D(\alpha(z),\lambda(z)):=D_{\frak{b}_{i_{n-1}},z}, \quad D(\beta(z),\rho(z)):=D_{\frak{b}_{i_{n-1}}\frak{b}_{i_n},z}, 
\end{equation*}
and use the result that $D(\beta(z),\rho(z))\subset D(\alpha(z),e^{-v}\lambda(z))$
for some $v>0$ due to \cref{eqn:convergence in z norm}. More precisely, let $w_0$ be the center of $D(\alpha(z),\lambda(z))$. For any points $w_1\in \partial D(\alpha(z),\lambda(z))$ and $w_2\in \partial D(\beta(z),\rho(z))$, we have
\begin{equation}
\label{eqn:distance boundaries}
(1-e^{-v})\lambda(z) \leq |w_1-w_2|_z\leq \int_{0}^{1} \left|(\gamma^{-1}f(t))'\right|_z dt\leq |\gamma w_1-\gamma w_2|_z \cdot |\gamma'w_0|_z^{-1}\cdot C\exp(R/(1-c)),
\end{equation}
where $f:[0,1]\to \mathbb{C}$ is given by $f(t)=(1-t)\gamma w_2+t\gamma w_1$, and the constants $C, R>0$ and $c\in (0,1)$ are given as in \cref{lem-unifrom-dis-ele}.

Let $w_3\in D(\alpha(z),\lambda(z))$ be such that $\gamma w_3$ is the center of $\gamma D(\alpha(z),\lambda(z))$. By a similar argument, we have
\begin{equation}
\label{eqn:distance center boundary}
|\gamma w_3-\gamma w_1|_z\leq 2\lambda(z) \cdot |\gamma'(w_0)|_z\cdot C\exp(R/(1-c)).
\end{equation}
Combining \cref{eqn:distance boundaries} and \cref{eqn:distance center boundary}, we obtain
\begin{equation*}
|\gamma w_1-\gamma w_2|_z\geq \frac{1}{2} C^{-2} \exp(-2R/(1-c)) \cdot (1-e^{-v}) \cdot|\gamma w_3-\gamma w_1|_z,
\end{equation*}
which implies there exists some $v_1>0$ such that for $0<|z|\ll 1$, 
\begin{equation*}
\frak{a}_z(D_{\mathbf{b},z}) \subset |z|^{v_1} D_{\mathbf{a},z}
\end{equation*}
in terms of the absolute norm. 
\end{proof}

\section{Speed of convergence}
Let $\Gamma<\SL_2(\mathbb{M}(\mathbb{D}))$ be a family of Schottky groups satisfying $(\bigstar)$.
In this section, we establish a logarithmic rate for the convergence  $Z(\Gamma, s/\log(1/|z|))\to Z_{I}(\Gamma, s)$.

\begin{thm}
\label{prop:logarithmic convergence}
Given any bounded region $K\subset \mathbb{C}$, we have for any $0<|z|<1/e$, 
\begin{equation*}
\left|Z(\Gamma_z,s/\log(1/|z|))-{Z_{I}(\Gamma,s)}\right|\lesssim_{K} \frac{1}{\log(1/|z|)} \quad{for} \quad s\in K.
\end{equation*}
\end{thm}

\subsection{Expansion of lengths}

The key to obtaining an effective convergence result such as \cref{prop:logarithmic convergence} lies in establishing the following expansion of $\ell(\gamma_z)$ with the leading term $\ell^{na}(\gamma)\log(1/|z|)$ and exponential bounds for the coefficients $a_j(\gamma)$.

\begin{prop}
\label{prop:expansion of length function}
 There exists a constant $C>0$ that depends only on $\Gamma$ such that for any $\gamma\in \Gamma\backslash \{id\}$, we have for any $0<|z|<e^{-C\ell^{na}(\gamma)}$, the function $\ell(\gamma_z)$ in $z$ has the expansion
\begin{equation}
\label{equ:expansion of length function}
\ell(\gamma_z) =\ell^{na}(\gamma) \log(1/|z|)+\Re \left(\sum_{j\geq 0} a_j(\gamma)z^j\right),
\end{equation}
where $a_j(\gamma)$'s are complex numbers satisfying
\begin{equation}\label{eq:expansion-a_j}
        |a_0(\gamma)|\leq C\ell^{na}(\gamma)\quad \text{and}\quad|a_j(\gamma)|\leq Ce^{C\ell^{na}(\gamma)(j+1)} \quad \text{for}\quad j\in \mathbb{N}.
    \end{equation}
\end{prop}

The proof of \cref{prop:expansion of length function} is based on the following Laurent expansion of the trace.
\begin{prop}\label{prop:laurent-est}
 There exists a constant $C>0$ that depends only on $\Gamma$ such that for any $\gamma\in \Gamma\backslash \{id\}$, the function $\tr(\gamma_z)$ in $z$ has the Laurent expansion
\begin{equation}
\label{Laurent trace}
 \tr(\gamma_z)=z^{-\ell^{na}(\gamma)/2}\left(\sum_{j\geq 0}A_j(\gamma)z^j\right)
\end{equation}
where $A_j(\gamma)$'s are complex numbers satisfying
    \begin{equation}
        |A_0(\gamma)|\geq e^{-C\ell^{na}(\gamma)}/C\quad \text{and}\quad|A_j(\gamma)|\leq Ce^{j+C\ell^{na}(\gamma)}\quad \text{for}\quad j\in\mathbb{Z}_{\geq 0}.
    \end{equation}

\end{prop}

\begin{proof}
Since each entry of $\gamma$ is an element of $\Ct$,  $\tr(\gamma)$ is also an element of $\Ct$, and the first term on the right-hand side of \cref{Laurent trace} is obtained using the definition of $\ell^{na}(\gamma)$. 

For the upper bound of $A_j(\gamma)$, recall $\{{\mathfrak{a}_k},\ 1\leq k\leq 2g  \}$ be a finite symmetric generating set of $\Gamma$.  Write $\gamma={\mathfrak{a}_{i_1}}\cdots{\mathfrak{a}_{i_m}}$ which we may assume the word on the right-hand side is cyclically reduced as $\Tr(\gamma_z)$ is invariant under conjugation. It follows from \cref{lem:word-na} that  
$m\lesssim \ell^{na}(\gamma)$.  
Each entry $\gamma (i,l)$ of $\gamma$ is a sum of at most $2^m$ terms, with each term a product of $m$ entries of generators. 
By subadditivity and submultiplicativity of the hybrid norm, the entries of $\gamma$ satisfy\footnote{The finiteness of hybrid norm follows from that the entries are meromorphic function on the unit disc with possible pole at $0$.}
\[ \sup_{\substack{1 \leq i,l\leq 2 }}\|\gamma(i,l)\|_{\mathrm{hyb}}\leq 2^m \sup_{\substack{1\leq k\leq 2g\\1 \leq i,l\leq 2 }}\{ \|\mathfrak{a}_{k}(i,l) \|_{\mathrm{hyb}} \}^m\leq e^{C\ell^{na}(\gamma)}. \]
Each entry of $\gamma$ has the form $\sum_{j\geq j_0}c_jz^j$. By the definition of hybrid norm, we have $|c_j|\leq \|\sum_{j\geq j_0}c_jz^j\|_{\mathrm{hyb}}e^{j}$. Combining these two estimates, we obtain the upper bound of $A_j(\gamma)$.

Now we prove the lower bound. By \cref{lem-uniform ell}, there exists $c>0$ such that for any conjugacy class $[\gamma]$ with $\ell^{na}(\gamma)>1/c$ and $0<|z|<c$, we have
    \[\ell(\gamma_z)\geq c\log(1/|z|)\ell^{na}(\gamma). \]
    Let $\lambda_1(\gamma_z),\lambda_2(\gamma_z)$ be the two eigenvalues of $\gamma_z$ with $|\lambda_1(\gamma_z)|>1$. Notice that $2\log|\lambda_1(\gamma_z)|=\ell(\gamma_z)=-2\log|\lambda_2(\gamma_z)|$. For $0<|z|<c$
    \begin{equation}\label{equ:trgg}
     |\tr(\gamma_z)|\geq e^{\ell(\gamma_z)/2}-e^{-\ell(\gamma_z)/2}\geq \frac{1}{2}(1/|z|)^{c\ell^{na}(\gamma)/2}\geq 1. 
    \end{equation}
    Recall
    \begin{equation}\label{equ:trgamma}
     \tr(\gamma_z)=z^{-\ell^{na}(\gamma)/2}(A_0(\gamma)+A_1(\gamma)z+\cdots). 
    \end{equation}
    Let $f(z)=A_0(\gamma)+A_1(\gamma)z+\cdots$, which is an analytic function on $|z|<1/e$. 
    Using \cref{equ:trgg,equ:trgamma} we obtain for $0<|z|<c$,
    \[ \log|f(z)|\geq -\frac{1}{2}\ell^{na}(\gamma)\log(1/|z|). \]
    Since $\gamma_z$ is loxodromic for all $z\in\D^*$, the analytic function $f(z)$ has no zeros in $\D$. Applying the maximal principle to the harmonic function $\log |f(z)|$, we have
    \begin{equation*}
        \log |A_0(\gamma)| = \log |f(0)| \geq \min_{|z|=c/2} \log | f(z) | \geq -\frac{1}{2}\ell^{na}(\gamma)\log(2/c).
    \end{equation*}
    Therefore, we obtain an exponential lower bound of $|A_0(\gamma)|$ in terms of $\ell^{na}(\gamma)$.
    \end{proof}
\begin{proof}[Proof of \cref{prop:expansion of length function}] We use the same notations as in the proof of \cref{prop:laurent-est}. We have $\lambda_1(\gamma_z)=\tr(\gamma_z)\left(\frac{1}{2}+\sqrt{\frac{1}{4}-\frac{1}{\tr(\gamma_z)^2}}\right)$ \footnote{On $\C\backslash (-\infty,0]$, we choose the principal of $\log z$: $\log z=\log |z|+i\arg z$ with $|\arg z|<\pi$; set $\sqrt{z}=e^{\log z/2}$.} is a meromorphic function at $0$. This is because 
\Cref{equ:trgg} gives $\lambda_{2}(\gamma_z)/\lambda_1(\gamma_z)\to 0$ as $z\to 0$, which implies $\log(\frac{1}{4}-\frac{1}{\tr(\gamma_z)^2})=\log\left(\frac{1-\lambda_2(\gamma_z)/\lambda_1(\gamma_z)}{2(1+\lambda_2(\gamma_z)/\lambda_1(\gamma_z))}\right)^2=2\log\left(\frac{1-\lambda_2(\gamma_z)/\lambda_1(\gamma_z)}{2(1+\lambda_2(\gamma_z)/\lambda_1(\gamma_z))}\right)$ is holomorphic in a neighborhood of $0$. We continue to use \Cref{Laurent trace} to obtain
\begin{align*}
\lambda_1(\gamma_z)&=z^{-\ell^{na}(\gamma)/2} A_0(\gamma) \left(1+\sum\limits_{j\geq 1} \frac{A_j(\gamma)}{A_0(\gamma)}z^j\right) \left(\frac{1}{2}+\sqrt{\frac{1}{4}-\frac{z^{\ell^{na}(\gamma)}}{A_0(\gamma)^2}\left(1+\sum\limits_{j\geq 1} \frac{A_j(\gamma)}{A_0(\gamma)}z^j\right)^{-2}}\right)\\
&=: z^{-\ell^{na}(\gamma)/2} A_0(\gamma)\cdot f(z)\cdot g(z).
\end{align*}

By the previous estimate, we have $|A_j(\gamma)/A_0(\gamma)|\leq e^{C\ell^{na}(\gamma)+j}$ and $|A_0(\gamma)|^{-1}\leq e^{C\ell^{na}(\gamma)}$ for some $C>0$. Thus for $|z|<e^{-10C\ell^{na}(\gamma)}/20$, we have the bounds
\begin{equation*}
    \left|\sum\limits_{j\geq 1} \frac{A_j(\gamma)}{A_0(\gamma)}z^j\right|\leq 2e^{2C\ell^{na}(\gamma)}|z|\leq \frac{1}{10},\quad \frac{|z|}{|A_0(\gamma)|^2}\leq e^{2C\ell^{na}(\gamma)}|z|\leq \frac{1}{10}.
\end{equation*}
This implies that on the disc $|z|<e^{-10C\ell^{na}(\gamma)}/20$
\begin{equation*}
\log(f(z)g(z))=\log(f(z)) \log (g(z)).
\end{equation*}
Moreover, we apply Cauchy's integral formula to obtain upper bounds of the form \cref{eq:expansion-a_j} for the coefficients of the power series expansion of $\log (f(z))$ and $\log (g(z))$. We complete the proof by using the equation
\begin{align}
\label{equ:obtaining lz using lambdaz}
\ell(\gamma_z)&=2\log|\lambda_1(\gamma_z)|=\ell^{na}(\gamma) \log(1/|z|) +2\log |A_0(\gamma)|+2\log |f(z)|+2\log |g(z)|\nonumber\\
&=\ell^{na}(\gamma) \log(1/|z|) +2\log |A_0(\gamma)|+2\Re\left(\sum_{j\geq 1}b_j z^j\right)
\end{align}
where $\sum_{j\geq 1}b_jz^j$ is the power series expansion of $\log(f(z))+\log(g(z))$ about $0$. \qedhere

\end{proof}

\subsection{Convergence to intermediate zeta functions: 
statement of results}

To obtain a better convergence rate, we introduce the following \textit{intermediate zeta functions}: for each $M\in\mathbb{Z}_{\geq 0}$ and $z\in \mathbb{D}^*$, 
\begin{equation}
\label{equ:intermediate zeta function}
Z_M(\Gamma,z,s)=\prod_{[\gamma]\in \calP}(1-e^{-s\ell_M(\gamma,z)}) 
\end{equation}
where $\ell_{M}(\gamma,z)$ is the {\it{$M$th-expansion of length}} defined by 
\begin{equation}
\label{eqn:Mth expansion of length}
\ell_M(\gamma,z)=\ell^{na}(\gamma)+\Re(\sum\limits_{j=0}^{M}a_j(\gamma)z^j/\log(1/|z|))
\end{equation}
with $a_j(\gamma)$'s the coefficients given in the expansion \cref{equ:expansion of length function}.

A priori, an intermediate zeta function is defined for $\Re s$ large. In \cref{sec:analytic}, we will prove that it admits an analytic extension to $s\in \CC$. 

We establish the following polynomial rate for the convergence of
the $\log(1/|z|)$-rescaled zeta functions to an intermediate zeta function.
\begin{thm}\label{prop:speed-conv}
    For any $C>0$, $M\in \mathbb{Z}_{\geq 0}$ and $\epsilon>0$, we have for any $0<|z|<1/e$, 
    \begin{equation*}
        |Z(\Gamma_z,s/\log(1/|z|))-Z_M(\Gamma,z,s)|\lesssim_{C,M,\epsilon} |z|^{1-\epsilon}\quad \text{for}\quad |\Re s|\leq C, \, |\Im s|\leq C|z|^{-M}.
    \end{equation*}
\end{thm} 

As a corollary, we obtain the convergence of zeros of zeta functions: in particular, the convergence of the first zeros gives the convergence of the Hausdorff dimensions of limit sets.

\begin{cor}\label{cor:res-conv} 
    Let $R>0$, $\epsilon>0$ and $M\in \mathbb{Z}_{\geq 0}$. Suppose that no zero of $Z_{I}(\Gamma,s)$ lies on the boundary of $D_R:=\{s:|s|<R\}$. Let
    \begin{itemize}
        \item $\rho_1,\cdots,\rho_{A_{z}}$ be the zeros of $Z(\Gamma_z, s/\log(1/|z|))$ in $D_R$;
        \item $\rho^M_1,\cdots, \rho^M_{B_z}$ be the zeros of $Z_M(\Gamma,z,s)$ in $D_R$;
        \item $\rho_1^{na},\cdots,\rho_{N}^{na}$ be the zeros of $Z_{I}(\Gamma,s)$ in $D_R$. \footnote{Our convention is to list zeros according to their multiplicities.}
    \end{itemize} 
    
    Then there exists $t_0>0$ such that for $0<|z|<t_0$, we have $A_{z}=B_z=N,$ and we can order the zeros such that for each $j\in \mathbb{N}\cap[1,N]$ 
    \begin{equation*}
        |\rho_j-\rho_j^M|\leq |z|^{\frac{1}{m_j^{na}}-\epsilon},
    \end{equation*}
    where $m_j^{na}$ is the multiplicity of $\rho_j^{na}$ as a zero of $Z_{I}(\Gamma, s)$.
\end{cor}

\begin{cor}\label{cor:zero-ZM}
    Let $C>0$, $\epsilon>0$ and $M\in \mathbb{Z}_{\geq 0}$. Let $D\subset [-C,C]+i[-C|z|^{-M}, C|z|^{-M}]$ be an open connected set with its boundary a simple closed curve and
     \begin{itemize}
        \item $\rho_1,\cdots,\rho_{A_{z}}$ be the zeros of $Z(\Gamma_z,s/\log(1/|z|))$ in $D$;
        \item $\rho^M_1,\cdots, \rho^M_{B_z}$ be the zeros of $Z_M(\Gamma,z,s)$ in $D$.
    \end{itemize} 
    Suppose $|Z_M(\Gamma,z,s)|\geq |z|^{1-\epsilon}$ on $\partial D$ for $0<|z|<t_0$, then there exists $t_1\in (0,t_0)$ such that for $0<|z|<t_1$, we have $A_{z}=B_z$.
\end{cor}
The difference between \cref{cor:res-conv} and \cref{cor:zero-ZM} lies in the regions we compare $Z(\Gamma_z, s/\log(1/|z|))$ and $Z_{M}(\Gamma,z,s)$. \cref{cor:zero-ZM} will be used in \cref{exa.PV19} to recover \cite{pollicottZerosSelbergZeta2019a} on the structure of the zeros with large imaginary parts
of the Selberg zeta function for a three-funnel surface.

\subsection{Convergence to intermediate zeta functions: proofs}

In this section, we prove the effective convergence results \cref{prop:logarithmic convergence},
\cref{prop:speed-conv}, \cref{cor:res-conv} and \cref{cor:zero-ZM}.

First, we recall the Hadamard three-circle theorem.
\begin{prop}\label{prop:hadamard}
Let $r_1< r_2< r_3\in (0,\infty)$ and set $\alpha=\frac{\log(r_3/r_2)}{\log(r_3/r_1)}$.
    Let $f:\{z:r_1 \leq |z-a| \leq r_3\}\to\mathbb{C}$ be a holomorphic function with $a\in \mathbb{C}$. Then
    \begin{equation*}
        \max_{|z-a|=r_2}|f(z)|\leq \left(\max_{|z-a|=r_1}|f(z)|\right)^{\alpha}\left(\max_{|z-a|=r_3} |f(z)|\right)^{1-\alpha}.
    \end{equation*}
\end{prop}
\begin{proof}
    This follows from the fact that $\log|f(z)|$ is a subharmonic function.
\end{proof}
The strategy of the proof of \cref{prop:speed-conv}
is to prove convergence in the region $\Re s\gg 1$ and then use \cref{prop:hadamard} to obtain
convergence in a larger region. 

We need a lower bound for the function $\ell_M(\gamma,z)$ given in \cref{eqn:Mth expansion of length},
whose proof will be provided at the end of \cref{sec:use derivative of Schottky groups}.
\begin{lem}\label{lem:ellM>}
    There exists $C>1$ depending only on $\Gamma$ such that for any $M\in \mathbb{Z}_{\geq 0}$ and
    for any $\gamma\in \Gamma\backslash \{id\}$, we have for any $0<|z|\leq e^{-C(M+1)}$, 
    \begin{equation*}
     \ell_M(\gamma,z)\geq  \ell^{na}(\gamma)/2. 
    \end{equation*}
\end{lem}

\begin{prop}\label{lem:quant-bound-zeta}
There exist $s_0>1$ and $C_{\Gamma}>0$ depending on $\Gamma$, such that for any $M\in \mathbb{Z}_{\geq 0}$, we have for any $\log(1/|z|)>C_{\Gamma}(M+1)$ and any $C_0>0$,
\begin{equation}\label{equ:zeta-conv-quant}
        \left|\frac{Z(\Gamma_z, s/\log(1/|z|))}{Z_M(\Gamma,z,s)}-1\right|\lesssim_{C_0} |z| \quad \text{for}\quad \Re s> (M+1) s_0, \,\,|\Im s|\leq C_0|z|^{-M}.
\end{equation}
\end{prop}

\begin{proof}
We establish the convergence of lengths of geodesics. We use $C$ to refer to a constant depending only on $\Gamma$, which may vary from line to line.

We first deal with the case $\Re s< 10\log(1/|z|)$. Due to \cref{prop:expansion of length function}, for each $\gamma\in\Gamma\backslash \{id\}$ and $0<|z|<e^{-C\ell^{na}(\gamma)}$,

\begin{equation}\label{equ:expansion}
    \ell(\gamma_z)/\log(1/|z|)=\ell^{na}(\gamma)+\Re\left(\sum\limits_{j\geq 0}a_j(\gamma)z^j\right)/\log(1/|z|)
\end{equation}
where $a_j(\gamma)$'s are complex numbers satisfying
\begin{equation}\label{equ:lb-aj}
    |a_0(\gamma)|\leq C \ell^{na}(\gamma), \quad |a_j(\gamma)|\leq Ce^{C\ell^{na}(\gamma)(j+1)}.
\end{equation} 
By \cref{equ:expansion} and \cref{equ:lb-aj}, we have for $\log(1/|z|)\gg C$ and $\ell^{na}(\gamma)\ll C^{-1}\log(1/|z|)$, 
\begin{equation}\label{equ:modif-ell-lb}
    \ell_M(\gamma,z)\geq \ell^{na}(\gamma)/2, \quad \ell(\gamma_z)/\log(1/|z|)\geq \ell^{na}(\gamma)/2,
\end{equation}
and
\begin{equation}\label{equ:modif-ell-difference}
    |\ell_M(\gamma,z)-\ell(\gamma_z)/\log(1/|z|)|\leq C(e^{C\ell^{na}(\gamma)}|z|)^{M+1}/\log(1/|z|).
\end{equation}
Using the following basic inequality: for $|z_1|,|z_2|<1/2$,
\begin{equation}\label{equ:log1-z}
    |\log(1-z_1)-\log(1-z_2)|\leq 2|z_1-z_2|,
\end{equation}
we have for $1\ll\Re s<10\log (1/|z|)$ and $\ell^{na}(\gamma)\leq \epsilon_0 \log (1/|z|)$ with  $\epsilon_0>0$ sufficiently small,
\begin{equation*}
\begin{aligned}
    &\left|\log(1-e^{-s\ell(\gamma_z)/\log(1/|z|)})-\log (1-e^{-s \ell_M(\gamma,z)})\right|\lesssim \left| e^{-s\ell(\gamma_z)/\log(1/|z|)}-e^{-s\ell_M(\gamma,z)} \right|\\
    &= e^{-\Re s \ell_M(\gamma,z)}\left|1-e^{- s\ell(\gamma_z)/\log(1/|z|)+s\ell_M(\gamma,z)}\right|\lesssim e^{-\Re s\ell^{na}(\gamma)/2}\frac{|s(e^{C\ell^{na}(\gamma)}|z|)^{M+1}|}{\log (1/|z|)}\leq e^{(C(M+1)-\Re s)\ell^{na}(\gamma)/2}\frac{|sz^{M+1}|}{\log (1/|z|)},
\end{aligned}
\end{equation*} 
where to obtain the second to last inequality, we use the inequality 
\[ |1-e^{z}|\lesssim |z|,\ \text{ for }|\Re z|\leq 1, \]
because \cref{equ:modif-ell-difference}, and the assumption $\Re s<10 \log(1/|z|)$ and $\ell^{na}(\gamma)\leq \epsilon_0 \log (1/|z|)$
give 
$$|\Re s(e^{C\ell^{na}(\gamma)}|z|)^{M+1}|/\log (1/|z|)\leq 10e^{C\ell^{na}(\gamma)}|z|\leq 1.$$
Meanwhile, for $\log (1/|z|)\geq C(M+1)$ and $\ell^{na}(\gamma)>\epsilon_0 \log (1/|z|)$, by \cref{lem-uniform ell} and \cref{lem:ellM>} we have
\begin{equation}
\label{eq:different t and M}
    |\log(1-e^{-s\ell(\gamma_z)/\log (1/|z|)})|\lesssim e^{-c\Re s\ell^{na}(\gamma)},\quad|\log (1-e^{-s\ell_M(\gamma,z)})|\lesssim e^{-c\Re s\ell^{na}(\gamma)},
\end{equation}
where we use $|\log(1-e^{-z})|\lesssim e^{-\Re z}$ for $\Re z>1$.
Thus for $\log (1/z)\geq C(M+1)$, and for 
$(M+1)s_0\leq \Re s<10 \log (1/|z|)$ and $|\Im s| \leq C_0|z|^{-M}$, we have
\begin{equation}\label{equ:zeta-conv-quant-k=0}
\begin{aligned}
    \sum\limits_{[\gamma] \in \calP}|\log(1-e^{-s\ell(\gamma_z)/\log (1/|z|)})-\log (1-e^{-s\ell_M(\gamma,z)})| 
   & \lesssim  \frac{|sz^{M+1}|}{\log (1/|z|)}\left(\sum\limits_{ \ell^{na}( \gamma)\leq \epsilon_0\log(1/|z|) } e^{(C(M+1)-\Re s) \ell^{na}(\gamma)/2}\right)\\
     +\sum\limits_{ \ell^{na}( \gamma)> \epsilon_0\log(1/|z|)  }  e^{-c\Re s \ell^{na}(\gamma)}&\lesssim \frac{|sz^{M+1}|}{\log (1/|z|)}+|z|\lesssim_{C_0} |z|, 
\end{aligned}
\end{equation}
where we used the counting result \cref{equ:gamma_n} to obtain the last inequality.

Now, we treat the part of $k\geq 1$ terms in the Selberg zeta function. We only provide the details for the $\SL_2(\C)$-case here, and $\SL_2(\R)$-case is similar. For $\Re s\gg 1$,  \cref{lem-uniform ell} and \cref{equ:modif-ell-lb} allow us to use \cref{equ:log1-z} with $z_2=0$ to obtain
\begin{equation}
\label{equ:k geq 1 terms}
\sum_{k\geq 1}(k+1)|\log(1-e^{-(s/\log(1/|z|)+k)\ell(\gamma_z)})|\lesssim \sum_{k\geq 1}(k+1)e^{-(\Re s/\log(1/|z|)+k)\ell(\gamma_z)}\leq 2 e^{-(\Re s/\log(1/|z|)+1)\ell(\gamma_z)}/(1-e^{-\ell(\gamma_z)})^2. 
\end{equation}
Let 
$\log(1/|z|) \gg C$ and $\Re s\gg C$. We have the following two cases.
\begin{itemize}
\item For
$\ell^{na}(\gamma)\leq \epsilon_0 \log (1/|z|)$ with $\epsilon_0>0$ sufficiently small, by \cref{equ:expansion}, we have
\[ \ell(\gamma_z)\geq \log(1/|z|)\left(\ell^{na}(\gamma)-\frac{|a_0(\gamma)|}{\log(1/|z|)}+O\left(\frac{e^{C\ell^{na}(\gamma)}|z|}{\log (1/|z|)}\right)\right)\geq \left(1-\frac{2C}{\log(1/|z|)}\right)\log(1/|z|)\ell^{na}(\gamma).  
\]
This gives \footnote{For $\Re s, \log(1/|z|)\geq 4C$, we have $(\Re s/\log(1/|z|)+1)(1-2C/\log(1/|z|))\geq 1+2C/\log(1/|z|)-8C^2/\log(1/|z|)^2\geq 1$.}
\[ (\Re s/\log(1/|z|)+1)\ell(\gamma_z)\geq (\Re s/\log(1/|z|)+1)\left(1-\frac{2C}{\log(1/|z|)}\right)\log(1/|z|)\ell^{na}(\gamma)
\geq\log(1/|z|)\ell^{na}(\gamma) . \]

\item For $\ell^{na}(\gamma)>\epsilon_0 \log (1/|z|)$, we can continue to apply \cref{lem-uniform ell} to \cref{equ:k geq 1 terms} to obtain an upper bound in terms of $\ell^{na}(\gamma)$.
\end{itemize}
In conclusion, for $\log(1/|z|) \gg C$ and $\Re s\gg C$, we obtain
\begin{equation}\label{equ:zeta-conv-quant-k>0}
\begin{aligned}   
\sum_{[\gamma]\in \calP}\sum_{k\geq 1}(k+1)|\log(1-e^{-(s/\log(1/|z|)+k)\ell(\gamma_z)})|\lesssim \sum\limits_{ \ell^{na}( \gamma)\leq \epsilon_0\log(1/|z|) } e^{-\ell^{na}(\gamma)\log(1/|z|)}
     \\
     +\sum\limits_{ \ell^{na}( \gamma)> \epsilon_0\log(1/|z|)  }  e^{-c(\Re s+\log(1/|z|))\ell^{na}(\gamma)}\lesssim |z|,
\end{aligned}
\end{equation}
where we used again the counting result \cref{equ:gamma_n} to obtain the last inequality.

For $\log(1/|z|)\geq C(M+1)$ and $(M+1)s_0<\Re s<10 \log(1/|z|)$, we obtain \cref{equ:zeta-conv-quant} by combining \cref{equ:zeta-conv-quant-k=0} and \cref{equ:zeta-conv-quant-k>0}. For $\Re s\geq 10\log(1/|z|)$, it remains to bound the right-hand side of the following inequality by $|z|$:
\begin{equation*}
\left|\log(1-e^{-s\ell(\gamma_z)/\log(1/|z|)})-\log (1-e^{-s \ell_M(\gamma,z)})\right| \leq \left|\log(1-e^{-s\ell(\gamma_z)/\log(1/|z|)})|+|\log (1-e^{-s \ell_M(\gamma,z)})\right|.
\end{equation*}
This can be achieved by an argument similar to 
the proof of \cref{equ:zeta-conv-quant-k>0}.
\end{proof}

Now we are ready to prove \cref{prop:speed-conv}.

\begin{proof}[Proof of \cref{prop:speed-conv}] 
Fix any $C>1$,  $M\in \mathbb{Z}_{\geq 0}$, and $\epsilon>0$. Recall the constants $s_0>0$ and $C_{\Gamma}>0$ in \cref{lem:quant-bound-zeta}. 

Set $R=(M+1)s_0+1$, and let $C_{\epsilon}>1$ be a constant satisfying $C_{\epsilon}>\max\{10 C,R\}$ and 
\begin{equation*}
\log(R+10 C)/\log (R+C_{\epsilon})<\epsilon.
\end{equation*}
It suffices to prove the proposition for $\log(1/|z|)>C_{\Gamma}(M+1)$. This 
is because the zeta functions depend continuously on the parameters $z$ and $s$, thus uniformly bounded for $s\in [-C,C]+i[-C|z|^{-M},C|z|^{-M}]$ and $\log(1/|z|) \leq C_{\Gamma}(M+1)$ depending on $C,M,\epsilon$.

For any $K_T:=[-C,C]+i[T,T+1]$ with $T\in [-C|z|^{-M},C|z|^{-M}]$, we choose three concentric discs $D_1\subset D_2\subset D_3$ with radii $r_1<r_2<r_3$ such that $D_1\subset \{s: \,\Re s>(M+1)s_0, |\Im s|\leq (C+2)|z|^{-M}\}$, $D_2\supset K_T$ and $D_3\subset \{s:\, \Re s > - C_{\epsilon}, |\Im s|\leq (C+R+C_{\epsilon}+1)|z|^{-M}\}$ (see Figure~\ref{fig:disc}). In particular, we can take
\begin{equation*}
   D_1=D(a,R-(M+1)s_0),\quad D_2=D(a,R+10 C), \quad D_3=D(a,R+C_{\epsilon})
\end{equation*}
where $a=R+iT$.

\begin{figure}
    \centering

\begin{tikzpicture}[scale=0.8]

\draw[->,ultra thick] (-5,0)--(8,0) node[right]{$x$};
\draw[->,ultra thick] (-1.5,-3)--(-1.5,7.5) node[above]{$y$};

\node at (-4,0.3)  {$-C_{\epsilon}$} ;

\node at (0,0.3)  {$(M+1)s_0$} ;

\node at (6,0.3)  {$2R+C_{\epsilon}$} ;

\node at (-4,6.5)  {$i(C+R+C_{\epsilon}+1)|z|^{-M}$} ;

\node at (2,2.3)  {$D_1$};

\node at (3.5,2.3)  {$D_2$};

\node at (4.5,2.3)  {$D_3$};

\draw [dashed] (-0.6,-3) -- (-0.6,7.5);

\draw (1,2) circle  (1.5);

\draw (1,2) circle (3);

\draw (1,2) circle (4);

\end{tikzpicture}

\caption{Choice of discs}
\label{fig:disc}
    
\end{figure}

By \cref{cor:analytic-middle}, the function
\begin{equation*}
    F_z(s)=Z(\Gamma_z,s/\log(1/|z|))-Z_{M}(\Gamma,z,s)
\end{equation*}
is an entire function. By
\cref{lem:quant-bound-zeta}, $F_z(s)$ satisfies the bound $|F_z(s)|\lesssim_{C} |z|$ for $s\in D_1$ and $s\in D_3\cap \{s:\,\Re s>(M+1)s_0\}$. By \cref{prop:uniform-bdd-SL2(R)} and \cref{prop:uniform-bdd-SL2(C)}, $Z(\Gamma_z,s/\log(1/|z|))$ is uniformly bounded on $D_3\cap \{ s:\,\Re s \leq (M+1)s_0\}$. Meanwhile, \cref{cor:analytic-middle} shows that $Z_M(\Gamma,z,s)$ is uniformly bounded on $D_3\cap \{s:\, \Re s \leq (M+1)s_0\}$. Therefore the function $F_z(s)$ is uniformly bounded on $D_3$, and from \cref{prop:hadamard} we conclude
\begin{equation*}
    |F_z(s)|\lesssim_{C,M,\epsilon} |z|^{\log(r_3/r_2)/\log(r_3/r_1)},\quad s\in \partial D_2.
\end{equation*}
The same bound also holds for $s\in D_2$ by the maximal modulus principle since $F_z(s)$ is an entire function.
The statement follows from
\begin{equation*}
    \log(r_3/r_2)/\log(r_3/r_1) =1-\frac{\log r_2-\log r_1}{\log r_3-\log r_1}=1-\log(R+10 C)/\log (R+C_{\epsilon})
     >1-\epsilon. \qedhere
\end{equation*}
\end{proof}

\begin{proof}[Proof of \cref{prop:logarithmic convergence}]
    Take $M=0$ in \cref{prop:speed-conv}, we have for $0<|z|<1/e$,
    \begin{equation*}
        |Z(\Gamma_z,s/\log(1/|z|) ) -Z_0(\Gamma,z,s)|\lesssim_{C} |z|^{1/2}, \quad \text{ for } |\Re s|\leq C,\, |\Im s|\leq C.
    \end{equation*}
    It suffices to prove that
    \begin{equation*}
        |Z_0(\Gamma,z,s)-Z_I(\Gamma,s)|\lesssim_K  \frac{1}{\log(1/|z|)},\quad s\in K.
    \end{equation*}
    By \cref{cor:analytic-middle},
    \begin{equation*}
        Z_0(\Gamma,z,s)=P_0(e^{s \mu_1(z)}, e^{s\mu_2(z)},\cdots, e^{s\mu_J(z)})
    \end{equation*}
    where each $\mu_j(z)=a_j+a_{j,0}/\log(1/|z|)$. By \cref{lem:M'-M}, 
    \begin{equation*}
        Z_I(\Gamma,s)=P_0(e^{sa_1},\cdots, e^{s a_J}).
    \end{equation*}
    Since $P_0$ is smooth, we conclude that
    \begin{equation*}
        |P_0(e^{s \mu_1(z)}, e^{s\mu_2(z)},\cdots, e^{s\mu_J(z)})-P_0(e^{sa_1},\cdots, e^{s a_J})|\lesssim_K \sup_{j} |\mu_j(z)-a_j| \lesssim_K \frac{1}{\log(1/|z|)},\quad s\in K.
    \end{equation*}
    The proof is complete.
\end{proof}

\begin{proof}[Proof of \cref{cor:res-conv}]
Given a disc $D_R$ that satisfies the assumption, let $\rho^{na}$ be a zero of $Z_I(\Gamma,s)$ in $D_R$ of multiplicity $m$. Then there exists a small disc $D_0$ centered at $\rho^{na}$ such that $\rho^{na}$ is the only zero of $Z_{I}(\Gamma,s)$ in $\bar{D}_0$. We can write
\begin{equation*}
Z_I(\Gamma,s)=(s-\rho^{na})^m f(s)
\end{equation*}
where $f(s)$ is an entire function non-vanishing on $D_0$.

Since $Z(\Gamma_z,s/\log(1/|z|))\to Z_I(\Gamma,s)$ uniformly for $s\in \partial D_0$, there exists $t_1>0$ such that for $0<|z|<t_1$, we have
\begin{equation*}
|Z_I(\Gamma,s)|>|Z_{I}(\Gamma,s)-Z(\Gamma_z,s/\log(1/|z|))|\quad \text{for}\quad s\in \partial D_0.
\end{equation*} 
By Rouch\'{e}'s theorem, $Z_I(\Gamma,s)$ and $Z(\Gamma_z,s/\log(1/|z|))$ have the same number of zeros in $D_0$. We can write 
\begin{equation*}
Z(\Gamma_z,s/\log(1/|z|))=(s-\rho_1)\cdots(s-\rho_{m}) f_z(s)
\end{equation*}
where $f_z(s)$ is an entire function non-vanishing on $D_0$. We claim that $f_z(s)\to f(s)$ uniformly on $\bar{D}_0$. Write $D_0=D(\rho^{na},r)$.
By the uniform convergence $Z(\Gamma_z,s/\log(1/|z|))\to Z_I(\Gamma,s)$ on $\bar{D}_0$,  for any $\epsilon>0$, there exists $t(\epsilon)>0$ such that for $0<|z|<t(\epsilon)$, we have 
\begin{equation*}
    |Z(\Gamma_z,s/\log(1/|z|))- Z_I(\Gamma,s)|<\min\left\{\epsilon,\inf_{\epsilon\leq |s-\rho^{na}|\leq r}|Z_I(\Gamma,s)|\right\}\quad \text{for} \quad s\in \bar{D}_0.
\end{equation*}
Therefore $Z(\Gamma_z,s/\log(1/|z|))$ has no zero in $\{s:\epsilon\leq |s-\rho^{na}|\leq r\}$, which implies $|\rho_j-\rho^{na}|< \epsilon$ for each $\rho_j$. For $s\in \bar{D}_0$, 
\begin{equation*}
    |f_z(s)-f(s)| \leq  \sup_{s\in \partial D_0}|Z(\Gamma_z,s/\log(1/|z|))/((s-\rho_1)\cdots(s-\rho_{m}))-Z_I(\Gamma,s)/(s-\rho^{na})^m|\lesssim \epsilon.
\end{equation*}
As a consequence, there exist $C>1$ and $t_2>0$ such that for $0<|z|<t_2$, we have $C^{-1}<|f_z(s)|<C$ for $s\in D_0$.

Similarly, since $Z_M(\Gamma,z,s)\to Z_I(\Gamma,s)$, there exists $t_3>0$ such that for $0<|z|<t_3$, $Z_M(\Gamma,z,s)$ has $m$ zeros in $D_0$, and we can write 
    \begin{equation*}
        Z_M(\Gamma,z,s)=(s-\rho^M_1)\cdots(s-\rho^M_m)f_{M,z}(s)
    \end{equation*}
    where $f_{M,z}(s)$ is an entire function non-vanishing on $D_0$.

    For $0<|z|<\min\{t_1,t_2,t_3\}$, consider the following functions on $D_0$:
    \begin{equation*}
    g_z(s)=(s-\rho_1)\cdots(s-\rho_{m})\quad \text{and} \quad h_z(s)=g_z(s)-Z_{M}(\Gamma,z,s)/f_z(s)
    \end{equation*}
    Thus, we can complete the proof by applying the following \cref{lem:perturb-poly} to $g_z(s)$ and $h_z(s)$.
\end{proof}

\begin{lem}\label{lem:perturb-poly}
    Let $g(s)=s^d+a_1s^{d-1}+\cdots+a_d=(s-\rho_1)\cdots(s-\rho_d)$ be a polynomial such that all of its zeros $\rho_1,\ldots,\rho_d$ are in a disc $D(\rho,\epsilon_0)$. Let $h(s)$ be a holomorphic function on $D(\rho,3\epsilon_0)$ such that $|h(s)|\leq \epsilon \leq 4^{-d^3-10d}\epsilon_0^{d}$ for $z\in \partial D(\rho,2\epsilon_0)$. Then $g(s)+h(s)$ also has $d$ zeros $\rho_1',\cdots, \rho_d'$ in $D(\rho,2\epsilon_0)$, and we can order the zeros such that for each $j\in [1, d]\cap \mathbb{N}$,
    \begin{equation*}
        |\rho_j-\rho_j'|<4^{d^2+3}\epsilon^{1/d}.
    \end{equation*}
\end{lem}
\begin{proof}
Since $|g(s)|\geq \epsilon_0^d$ on $\partial D(\rho,2\epsilon_0)$, and $|h(s)|\leq \epsilon<\epsilon_0^d$, by Rouch\'e's theorem $g(s)$ and $g(s)+h(s)$ have the same number of zeros in $D(\rho,2\epsilon_0)$. 

We claim there exist finitely many discs $D_\ell$ 
of radius $C\epsilon^{1/d}$ for some $C\in [2, 4^{d^2+2}]\cap \mathbb{N}$ 
 such that they cover all the $\rho_j$'s, and $2D_{\ell}\cap 2D_{\ell'}=\varnothing$ {\footnote{Here, if $D_{\ell}=D(a,r)$, then $2D_{\ell}=D(a,2r)$.}} for $\ell\neq \ell'$. Let $S=\{|\rho_i-\rho_j|:\, i,j=1,2,\cdots,d\}\subset [0, 2\epsilon_0]$. Then $\#S\leq d^2$. By the pigeonhole principle, there exists $C\in [2,4^{d^2+2}]\cap \mathbb{N}$ such that
$[C\epsilon^{1/d}, 4C\epsilon^{1/d}] \cap S=\varnothing$ (note the assumption $\epsilon\leq 4^{-d^3-10d} \epsilon_0^d$ is used here to ensure $4^{d^2+3}\epsilon^{1/d}<\epsilon_0$). Then for any pair $\rho_i,\rho_j$, 
we have either $|\rho_i-\rho_j|< C \epsilon^{1/d}$ or $|\rho_i-\rho_j|>4C\epsilon^{1/d}$. 
Let $D_1=D(\rho_1,C\epsilon^{1/d})$. Choose a zero $\rho_{j_2}\notin D_1$, and let $D_2=D(\rho_{j_2},C\epsilon^{1/d})$. Using the inequality  $|\rho_1-\rho_{j_2}|>4C\epsilon^{1/d}$, we obtain  
 $2D_1\cap 2D_2=\varnothing$. We repeat this process finitely many times to obtain the desired family of discs.

The fact $\epsilon \leq 4^{-d^3-10d} \epsilon_0^{d}$ ensures $2D_{\ell}\subset D(\rho_*,2\epsilon_0)$. Note
\begin{equation*}
    |h(s)|\leq \epsilon< (C\epsilon^{1/d})^d \leq |g(s)|\quad \text{for}\quad s\in \partial(2D_{\ell}).
\end{equation*}
By Rouch\'e's theorem, $g(s)+h(s)$ and $g(s)$ have the same number of zeros in $2D_{\ell}$, and the proof is complete. \qedhere

\begin{proof}[Proof of \cref{cor:zero-ZM}]
By Rouch\'e's theorem, it suffices to show 
\begin{equation*}
    |Z(\Gamma_z,s/\log(1/|z|))-Z_M(\Gamma,z,s)|< |Z_M(\Gamma,z,s)|,\quad z\in\partial D.
\end{equation*}
Since we assume $|Z_M(\Gamma,z,s)|\geq |z|^{1-\epsilon}$, it follows from \cref{prop:speed-conv} that
\begin{equation*}
    |Z(\Gamma_z,s/\log(1/|z|))-Z_M(\Gamma,z,s)|\lesssim_{C,M,\epsilon}|z|^{1-\epsilon/2}<|z|^{1-\epsilon}
\end{equation*}
for $|z|$ sufficiently small.
\end{proof}

\begin{proof}[Proof of \cref{thm:zeta-conv} and \cref{thm:three-funnel}]
    \cref{thm:zeta-conv} is a direct consequence of \cref{prop:logarithmic convergence} and \cref{prop:speed-conv}. The estimate on the Hausdorff dimension of the limit set in \cref{thm:three-funnel} is a consequence of \cref{cor:res-conv} and the computation of the intermediate zeta function $Z_0(\Gamma,z,s)$ for symmetric three-funnel hyperbolic surfaces in \cref{exa.borth}.
\end{proof}

\end{proof}

\section{Analytic extension of intermediate zeta functions}\label{sec:analytic}

Recall the intermediate zeta functions introduced in \cref{equ:intermediate zeta function}:
\begin{equation}
	\label{equ:recall intermediate zeta functions}
	Z_M(\Gamma,z,s)=\prod\limits_{[\gamma]\in \calP}(1-e^{-s\ell_M(\gamma,z)})\quad \text{for} \quad \Re s\gg 1.
\end{equation}
The goal of this section is to show that $Z_{M}(\Gamma,z,s)$ admits an analytic extension to $\mathbb{C}$.

This can be seen as a generalization of the analytic continuation of the Ihara zeta function. We need to find good combinatorial relations between $\ell_M(\gamma,t)$'s and then obtain a determinant formula similar to the Ihara zeta function. 

We will first review the proof of the determinant formula for a weighted version of the Ihara zeta function. Then, we use derivative cocycles from the Schottky group to get a good combinatorial relation and apply the determinant formula. 

\subsection{Weighted Ihara zeta function}
We briefly review the proof of analytic continuation of the weighted Ihara zeta function. For more details, see, for example, \cite[Section 3]{hortonWhatAreZeta2006}.

Let $G=(V,E)$ be a finite graph, where $E=\{e_1,\cdots, e_{2J}\}$ is the set of oriented edges of $G$ with $e_j$ and $e_{j+J}$ in opposite direction. We write $e_{j+J}=e_j^{-1}$ to refer to their relation. For each edge $e_j$, we associate it a complex number $h_j$.
Define $W(s)$ to be a $(2J)\times (2J)$-matrix by
\begin{equation*}
	W(e_j,e_k)(s)=e^{-s(h_j+h_k)/2}
\end{equation*}
if the endpoint of $e_j$ is the beginning point of $e_k$ and $e_j\neq e_k^{-1}$, and $W(e_j,e_k)(s)=0$ otherwise.

For a non-backtracking loop $P=(e_{i_1},\cdots, e_{i_n})$, define
$\ell_h(P)=\sum\limits_{j=1}^{n} h_{i_j}$.
The weighted Ihara zeta function is defined as a product over primitive loops $P$:
\begin{equation*}
	Z_I(G,h,s)=\prod\limits_{[P]\in \calP} (1-e^{-s\ell_h(P)}).
\end{equation*}
Recall that we use $P$ to denote a non-backtracking loop with a beginning point at a vertex, and $[P]$ to denote a loop forgetting the starting point. For $s$ with large real part, the product is absolute convergent. The following \cref{prop:ZI} gives the analytic extension of $Z_I(G,h,s)$ to the entire complex plane. 

\begin{prop}\label{prop:ZI}
	For $\Re s\gg 1$, we have $Z_I(G,h,s)=\det (I-W(s))$.
\end{prop}
\begin{proof}
	We expand $\log Z_I(G,h,s)$:
	\begin{equation*}
		\ag{
			\log Z_I(G,h,s)&=-\sum\limits_{j=1}^{\infty}\sum\limits_{[P]\in \calP}\frac{1}{j} e^{-js\ell_h(P)}\\
			&=-\sum\limits_{j=1}^{\infty}\sum\limits_{P \in \calP}\frac{1}{j\#(P)} e^{-js\ell_h(P)}\\
			&=-\sum\limits_{n=1}^{\infty}\sum\limits_{\substack{P\in \calP\\j:j\#(P)=n}}\frac{1}{n} e^{-js\ell_h(P)}
		},
	\end{equation*}
	where the sum $\sum_{P\in \calP}$ is over primitive non-backtracking loops $P$ and $\#(P)$ is the number of edges in the primitive loop $P$. 
	
	We can also expand $\log \det(I-W(s))$:
	\begin{equation*}
		\ag{
			\log \det (I-W(s))&=-\sum\limits_{n=1}^{\infty}\frac{1}{n}\tr W(s)^n.
		}
	\end{equation*}
	We finish the proof by using the equality $\tr W(s)^n=\sum\limits_{\substack{P\in \calP\\j:j\#(P)=n}} e^{-js\ell_h(P)}$.
\end{proof}

\subsection{$M$th-expansion of lengths and derivative cocycles}
\label{sec:use derivative of Schottky groups}
 In this section we prove the intermediate zeta function $Z_M(\Gamma,z,s)$ has an analytic extension and is computable for concrete examples. We will use the geometry of Schottky basis. Throughout this section, we fix generators $\mathfrak{a}_1,\cdots,\mathfrak{a}_g$ and a Schottky basis $D_{\mathfrak{a}_j}$ satisfying \cref{prop:separtion}.

For computational convenience, we will moreover assume $D_{\mathfrak{a}_j}$ are contained in $D(0,C)$, i.e. uniformly bounded away from $\infty$. This is always possible by choosing an appropriate coordinate.

The key observation is that the leading terms of the derivative is locally constant, from which we can compute the intermediate zeta function. 

Consider pairs of the form $(U,f)$ where $U$ is an open neighborhood of $0$, and $f$ is meromoprhic on $U$ and holomorphic on $U\backslash \{0\}$. Two such pairs $(U,f)$ and $(V,g)$ are equivalent if there exists an open neighborhood $W$ of $0$, contained in $U\cap V$, such that $f|_{W}=g|_{W}$. A germ of meromorphic functions at $0$ is an equivalence class of such a pair, and it has a representative that is a convergent complex Laurent series about $0$, $t^{n}\sum_{j=0}^{\infty}a_j t^j$ with $n\in \mathbb{Z}, a_{0}\neq 0$. Let $\mathcal{M}_0$ be the ring of germs of meromorphic functions at $0$, which is isomorphic to the ring of all convergent complex Laurent series about $0$. The ring $\mathcal{O}_0$ of germs of holomorphic functions at $0$ is defined analogously, which is isomorphic to the ring of all convergent complex power series about $0$.

Fix any $M\in \mathbb{Z}_{\geq 0}$. We introduce several combinatorial operations.
\begin{itemize}
	\item For $a=t^n \sum_{j=0}^{\infty}a_jt^j\in\mathcal{M}_0$ with $n\in\ZZ, a_0\neq 0$, we let $\lt_{M}(a)$ be the first $M$ leading terms of $a$:
	\begin{equation*}
		\lt_M(a):=t^n \sum_{j=0}^{M}a_jt^j.
	\end{equation*}
	
	\item We define the formal expansion of logarithm $\plog: \mathcal{M}_0\to \mathbb{Z}\log(1/t)\oplus (\mathcal{O}_0/{2\pi i\mathbb{Z}})$ by mapping 
	$a=a_0t^n(1+\sum_{j=1}^{\infty}a_jt^j)\in \mathcal{M}_0$ with $n\in \mathbb{Z}$ and $a_0\neq 0$ to
	\begin{equation}\label{equ:log}
		\mathrm{plog}(a):=-n\log (1/t)+\log a_0+\sum_{m=1}^{\infty}\frac{(-1)^{m-1}}{m}\left(\sum_{j=1}^{\infty} a_jt^j\right)^m.
	\end{equation}
	Here $\sum_{m=1}^{\infty}\frac{(-1)^{m-1}}{m}(\sum_{j=1}^{\infty} a_jt^j)^m$ is the power series expansion of the holomorphic function $\log (1+\sum_{j=1}^{\infty} a_jt^j)$ about $0$, which is uniquely determined.
	
	\item We define $\lt'_{M}:\mathbb{Z}\log(1/t)\oplus (\mathcal{O}_0/2\pi i\mathbb{Z})\to \mathbb{Z}\log(1/t)\oplus (\mathcal{O}_0/2\pi i \mathbb{Z})$ by mapping 
	$a=n{\log(1/t)}+\sum\limits_{j=0}^{\infty}a_jt^j$ to its first $M$ leading terms
	\begin{equation*}
		\lt'_{M}(a):=n\log(1/t)+\sum\limits_{j=0}^{M}a_jt^j.
	\end{equation*}
	And we define $\lt'_{-1}:\mathbb{Z}\log(1/t)\oplus (\mathcal{O}_0/2\pi i\mathbb{Z})\to \mathbb{Z}\log(1/t)$ by mapping 
	$a=n{\log(1/t)}+\sum\limits_{j=0}^{\infty}a_jt^j$ to its  $(-1)$-leading term
	\begin{equation*}
		\lt_{-1}'(a):=n\log (1/t).
	\end{equation*}

\end{itemize}

\begin{lem}
	\label{properties of combinatorial operations}
	For any $M\in \mathbb{Z}_{\geq 0}$, and for $x,y\in\mathcal{M}_0$, we have
	\begin{align}\label{equ:ltm'}
		&\lt'_{M}\circ\plog\circ\lt_M(x)=\lt'_M\circ\plog(x),\\
		&\lt'_M\circ\plog\circ\lt_M(xy)=\lt'_M\circ\plog (x) +\lt'_M\circ\plog (y).  \nonumber
	\end{align}
\end{lem}

\begin{proof}
	We recall the definition of $\plog$ \cref{equ:log}. The element $x\in \mathcal{M}_0$ can be written uniquely as $x=x_0(1+x_1)$ with $x_1=\sum_{j\geq 1}b_jt^j,\ x_0=b_0t^n$ for some $b_0\neq 0,\ n\in\Z$. Then
	$\plog x=-n\log(1/t)+\log b_0+(x_1-x_1^2/2+x_1^3/3-\cdots)$. From this expansion, we see that $\lt'_M\circ\plog(x)$ is determined by $x_0$ and the first $M-1$ term of $x_1$, and hence determined by $\lt_M(x)$. This completes the proof of the first equality. 
	
	The second equality uses the observation that for two convergent power series $\sum_{j\geq 1}a_jt^j,\,\sum_{l\geq 1}b_lt^l$, we have $\log((1+\sum_{j\geq 1}a_jt^j)(1+\sum_{l\geq 1}b_lt^l))=\log(1+\sum_{j\geq 1}a_jt^j)+\log(1+\sum_{l\geq l}b_lt^l)$ in an open neighborhood of $0$.
\end{proof}

In this subsection, it suffices to consider the action of $\SL_2(\Ct)$ on $\mathbb{P}^{1}_{\Ct}$, the set of classical points. 
For $\gamma=\begin{pmatrix} a& b\\ c& d\end{pmatrix}\in\SL_2(\Ct) $ and $x\neq \gamma^{-1}\infty\in \Ct$, the derivative of $\gamma$ at $x$ is given by
$$\gamma'x=1/(cx+d)^2.$$ 
\begin{lem}\label{lem:lc-constant}
	Fix $M\in \mathbb{Z}_{\geq 0}$. For 
	any $\gamma=\begin{pmatrix} a& b\\ c& d\end{pmatrix}\in \SL_2(\Ct)$ with $\gamma\infty\neq\infty$ and $x,y\in \Ct\backslash \{\gamma^{-1}\infty\}$ with $|x-y|_{na}<e^{-M} |x-\gamma^{-1}\infty|_{na}$, we have
	\[ \lt_M(\gamma'x)=\lt_M(\gamma'y). \]
\end{lem}

\begin{proof}
	It is sufficient to prove that $|\gamma
	'(x)-\gamma'(y)|_{na}<e^{-M}|\gamma'(x)|_{na}$, which is equivalent to $\lt_M(\gamma'x)=\lt_M(\gamma'y)$.

	We have
	\[ \gamma'x-\gamma'y=\frac{1}{(cx+d)^2}-\frac{1}{(cy+d)^2}=\frac{c(y-x)(c(x+y)+2d)}{(cx+d)^2(cy+d)^2}. \]
	Due to $\gamma\infty\neq\infty$, we know $c\neq 0$. Therefore
	\[\frac{|\gamma
		'x-\gamma'y|_{na}}{|\gamma'x|_{na}}=\left|\frac{c(x-y)(c(x+y)+2d)}{(cy+d)^2}\right|_{na}=\left|\frac{(x-y)(c(x+y)+2d)}{(y+d/c)(cy+d)}\right|_{na}. \]
	Due to $|x-y|_{na}<e^{-M}|x-\gamma^{-1}\infty|_{na}$, we have that
	\[ |y+d/c|_{na}=|y-\gamma^{-1}\infty|_{na}=|x-\gamma^{-1}\infty|_{na}>e^M |x-y|_{na}\]
	and $|cx+d+cy+d|_{na}\leq \max\{|cx+d|_{na},|cy+d|_{na} \}=|cy+d|_{na}$.
	The proof is complete.
\end{proof}

Let $\Gamma$ be a Schottky group in $\SL_2(\Ct)$. For $\gamma\in \Gamma\backslash \{id\}$, let $\lambda_1(\gamma),\ \lambda_2(\gamma)\in \Ct$ be its eigenvalues with $|\lambda_1(\gamma)|_{na}>|\lambda_2(\gamma)|_{na}$. Then there exists $\eta>0$ depending on the group $\Gamma$, such that for any $0<|z|<\eta$, the evaluation of $\lambda_1(\gamma)$ at $z$ satisfies $\lambda_1(\gamma)(z)=\lambda_1(\gamma_z)$. Hence, $\lambda_1(\gamma)$ is a meromorphic function at $0$ (see the proof of \cref{prop:expansion of length function}). Using \cref{equ:obtaining lz using lambdaz} and the definition of $\plog$, we obtain the following lemma.

\begin{lem}\label{lem:formal-expansion}
	There exists a constant $C>0$ that depends only on $\Gamma$ such that for any $\gamma\in \Gamma\backslash \{id\}$, we have for any $0<|z|<e^{-C\ell^{na}(\gamma)}$,
	\begin{equation*}
		2\ \Re\left(\plog(\lambda_1(\gamma))(z)\right)=\ell^{na}(\gamma) \log(1/|z|)+\Re\left(\sum_{j\geq 0} a_j(\gamma) z^j\right),
	\end{equation*}
	where $\plog (\lambda_1(\gamma))(z)$ is to evaluate the series $\plog (\lambda_1(\gamma))$ at $z$, and
	$a_j(\gamma)$'s are the complex coefficients given in \cref{equ:expansion of length function}.
	\end{lem}
For each $\gamma\in \Gamma\backslash \{id\}$, we define 
\begin{equation*}
	L_M(\gamma):=\lt_M'(2\ \plog(\lambda_1(\gamma))).
\end{equation*}
Then we have that for any $z\in \mathbb{D}^*$, the evaluation of $L_M(\gamma)$ at $z$ satisfies 
\begin{equation}
	\label{equ:LM and length expansion}
	\Re\left(L_M(\gamma)(z)\right)=\ell_M(\gamma,z)\log(1/|z|)
\end{equation}
where $\ell_M(\gamma,z)$ is defined in \cref{equ:intermediate zeta function}, and hence we obtain $\lt'_{-1}(L_M(\gamma))=\ell^{na}(\gamma)\log(1/t)$.

	\begin{prop}\label{prop:lc-cocycle}
		Fix any $M\in \mathbb{Z}_{\geq 0}$. There exists $C>0$ depending only on $\Gamma$ 
		such that for $N>C(M+1)$, we have for any reduced word $\gamma=\mathfrak{a}_{i_1}\cdots \mathfrak{a}_{i_N}$,
		\begin{equation*}
			\lt_M(\mathfrak{a}_{i_1}'(x))=\lt_M(\mathfrak{a}_{i_1}'(y))\quad \text{for any}\quad x,y\in \Ct\cap D_{\mathfrak{a}_{i_2}\cdots \mathfrak{a}_{i_N}}.
		\end{equation*}
		Denote the common value by $\lt_M(\mathfrak{a}_{i_1}'(\mathfrak{a}_{i_2}\cdots \mathfrak{a}_{i_N}))$. Moreover, for any cyclically reduced word $\gamma=\mathfrak{a}_{i_1}\cdots \mathfrak{a}_{i_n}$ with $n>N$, we have
		\begin{equation}
			\label{equ:LM}
			-L_M(\gamma)=l_M(\mathfrak{a}_{i_1},(\mathfrak{a}_{i_2}\cdots \mathfrak{a}_{i_N}))+l_M(\mathfrak{a}_{i_2},(\mathfrak{a}_{i_3}\cdots \mathfrak{a}_{i_{N+1}}))+\cdots+l_M(\mathfrak{a}_{i_n},(\mathfrak{a}_{i_{n+1}}\cdots \mathfrak{a}_{i_{n+N-1}})), 
		\end{equation}
		where $i_{k+j}:=i_{k+j-n}$ when $k+j>n$, and 
		\begin{equation*}
			l_M(\mathfrak{a}_{i_k},(\mathfrak{a}_{i_{k+1}}\cdots \mathfrak{a}_{i_{k+N-1}})):=\lt'_M\circ\plog\circ\lt_M(\mathfrak{a}_{i_k}'(\mathfrak{a}_{i_{k+1}}\cdots \mathfrak{a}_{i_{k+N-1}})).
		\end{equation*}
			\end{prop}
	
	\begin{rem}
		\label{rem:analytic Ihara}
		An easy consequence of \cref{equ:LM} is that
		\begin{equation*}
			\lt'_{-1}(-L_M(\gamma))=\lt'_{-1}\left(\sum_{k=1}^{n}l_{M}(\mathfrak{a}_{i_k},(\mathfrak{a}_{i_{k+1}}\cdots \mathfrak{a}_{i_{k+N-1}}))\right)=-\ell^{na}(\gamma)\log(1/t),
		\end{equation*}
		which is useful for the holomorphic extension of the Ihara zeta function.
	\end{rem}
	
	\begin{rem}
		\cref{prop:lc-cocycle} gives another proof of the lower of the leading coefficients. In general, it is hard to get a lower bound estimate of the leading coefficients. Similar questions in Archimedean case are important and hard. In the case of Schottky groups, the geometry and the hyperbolicity enable us to obtain such lower bound.
	\end{rem}

	\begin{proof}[Proof of \cref{prop:lc-cocycle}]

		By assumption, $\infty$ is not contained in any Schottky discs. Due to the properties of the Schottky group, there exists $c_0$ such that we have a lower bound \[d_{na}(\mathfrak{a}_i^{-1}\infty, (D_{\overline{\mathfrak{a}}_i})^c)\geq c_0\] 
		for all generators $\mathfrak{a}_i$. Moreover, the size of Schottky discs $D_{\mathfrak{a}_{i_1}\cdots \mathfrak{a}_{i_N}}$ tends uniformly to zero as the word length $N$ goes to infinity (\cref{lem:uniform-dis} for non-Archimedean norm). Therefore, we can find $N$ such that the radius of any disc $D_{\mathfrak{a}_{i_1}\cdots \mathfrak{a}_{i_{N-1}}}$ is less than $e^{-M}c_0$. By \cref{lem:uniform-dis}, we can take $N>{C(M+1)}$ for a uniform constant $C$ depending only on $\Gamma$.  

For any cyclically reduced word $\gamma=\mathfrak{a}_{i_1}\cdots \mathfrak{a}_{i_n}$ with $n>N$, recall $\lambda_1(\gamma), \lambda_2(\gamma)\in \Ct$ are its eigenvalues with $|\lambda_1(\gamma)|_{na}>|\lambda_2(\gamma)|_{na}$, and let
$\gamma_+$ be its attracting fixed point, which belongs to $\Ct\subset \mathbb{P}^1_{\Ct}$. Then we have that
$\gamma'(\gamma_+)=\lambda_2(\gamma)^2$. It follows from \cref{lem:formal-expansion} that
\begin{align*}
	-{L_M(\gamma)}= \lt_M'\circ\plog(\gamma'(\gamma_+))&=\lt_M'\circ\plog\circ\lt_M(\gamma'(\gamma_+))\\
	&=\lt_M'\circ\plog\circ\lt_M(\mathfrak{a}_{i_1}'(\mathfrak{a}_{i_2}\cdots \mathfrak{a}_{i_n}\gamma_+)\cdots \mathfrak{a}_{i_n}'(\gamma_+))\\
	&=\lt_M'\circ\plog\circ\lt_M\left(\lt_M(\mathfrak{a}_{i_1}'(\mathfrak{a}_{i_2}\cdots \mathfrak{a}_{i_n}\gamma_+))\cdots \lt_M(\mathfrak{a}_{i_n}'(\gamma_+))\right),
\end{align*}
where for the last equality we use the fact that for any $x,y\in\Ct$,
\begin{equation*}
	\lt_M(xy)=\lt_M(\lt_M(x)\lt_M(y)).
\end{equation*}
Since $\gamma$ is cyclically reduced, the attracting fixed point $\gamma_+$ is in the disc $D_{\gamma}\subset D_{\mathfrak{a}_{i_1}\cdots \mathfrak{a}_{i_{N-1}}}$, and $\gamma^{-1}\infty\in D_{\mathfrak{a}_{i_n}^{-1}}$. Due to the choice of $N$ and \cref{lem:lc-constant}, we obtain for any $y\in D_{\mathfrak{a}_{i_1}\cdots \mathfrak{a}_{i_{N-1}}}\cap \Ct$, 
\begin{equation*}
	\lt_M(\mathfrak{a}_{i_n}'(\gamma_+))=\lt_M(\mathfrak{a}_{i_n}'(y))
\end{equation*}
For any $\mathfrak{a}_{i_j}$, we have $\mathfrak{a}_{i_{j+1}}\cdots \mathfrak{a}_{i_n}\gamma_+\in\mathfrak{a}_{i_{j+1}}\cdots \mathfrak{a}_{i_n}D_{\gamma}\subset D_{\mathfrak{a}_{i_{j+1}}\cdots \mathfrak{a}_{i_n}\gamma}\subset D_{\mathfrak{a}_{i_{j+1}}\cdots \mathfrak{a}_{i_{j+N-1}}}$, and $a_{i_j}^{-1}\infty \in D_{a_{i_j}}$. Therefore, for any $y\in D_{\mathfrak{a}_{i_{j+1}}\cdots \mathfrak{a}_{i_{j+N-1}}}\cap\Ct$,
\begin{equation*}
	\lt_M(\mathfrak{a}_{i_j}'(\mathfrak{a}_{i_{j+1}}\cdots \mathfrak{a}_{i_n}\gamma_+))=\lt_M(\mathfrak{a}_{i_j}'(y)). 
\end{equation*}
In conclusion, we obtain that
\begin{equation*}
	-L_M(\gamma)=\lt_M'\circ\plog\circ\lt_M\left(\lt_M(\mathfrak{a}_{i_1}'(\mathfrak{a}_{i_2}\cdots \mathfrak{a}_{i_{N}})\cdots \lt_M(\mathfrak{a}_{i_n}'(\mathfrak{a}_{i_{n}}\cdots \mathfrak{a}_{n+N-1}))\right)
\end{equation*}
The proof is complete by applying \cref{equ:ltm'} to the right hand side of the above equation. 
\end{proof}

As a corollary, we obtain the following. 

\begin{cor}\label{cor:analytic-middle}
For any $M\in \mathbb{Z}_{\geq 0}$, the intermediate zeta function $Z_M(\Gamma,z,s)$ has a holomorphic extension to the entire complex plane $s\in\C$. More precisely, there exists a polynomial $P_{M}(x_1,\cdots,x_J)$ such that for any $s\in \mathbb{C}$ and $0<|z|<1/e$,
\begin{equation}\label{eq:Z_M-poly}
	Z_M(\Gamma,z,s) = P_M(e^{s\mu_1(z)},\, e^{s\mu_2(z)},\cdots, e^{s\mu_J(z)})
\end{equation}
where for $j=1,2,\cdots, J$, $\mu_j(z)$ is of the form 
\begin{equation}\label{eq:mu_j-expansion}
	\mu_j(z)=a_{j}+\frac{1}{\log(1/|z|)}\Re\left(\sum\limits_{k=0}^{M}a_{j,k} z^k \right) \quad \text{with}\,\,a_j\in\mathbb{R},\,\, a_{j,k}\in \mathbb{C}\,\,\text{for}\,\,k=0,\ldots,M.
\end{equation}

\end{cor}

\begin{proof} 

For $\gamma\in \Gamma\backslash \{id\}$, recall from \cref{equ:LM and length expansion} that for $z\in \mathbb{D}^*$, $\Re\left(L_M(\gamma)(z)\right)=\ell_M(\gamma,z)\log(1/|z|)$.

We introduce an edge matrix $W(s)$ with coordinates given by reduced words of length $N$, that is, words of form $\mathfrak{a}_{i_1}\cdots \mathfrak{a}_{i_N}$, where $N>0$ is the constant given in \cref{prop:lc-cocycle}. The entries are given by 
\begin{equation*}
	W(\mathfrak{a}_{i_1}\cdots \mathfrak{a}_{i_N}, \mathfrak{a}_{i_2}\cdots \mathfrak{a}_{i_{N+1}})(s)=\exp\left(\frac{s}{2\log(1/|z|)}\Re\left( l_M(\mathfrak{a}_{i_1},(\mathfrak{a}_{i_2}\cdots \mathfrak{a}_{i_N}))(z)+l_M(\mathfrak{a}_{i_2},(\mathfrak{a}_{i_3}\cdots \mathfrak{a}_{i_{N+1}}))(z)\right)\right), 
\end{equation*}
and the other entries are all equal to zero. Using the definition of $Z_{M}(\Gamma,z,s)$, \Cref{prop:ZI}, and \Cref{prop:lc-cocycle}, we can show that $\mathrm{det}(I-W(s))$ is the holomorphic extension of
$Z_{M}(\Gamma,z,s)$ to entire complex plane. 

It follows from \cref{prop:lc-cocycle} that for any reduced word $\mathfrak{a}_{i_1}\cdots \mathfrak{a}_{i_N}$, 
\begin{equation*}
	l_M(\mathfrak{a}_{i_1},(\mathfrak{a}_{i_2}\cdots \mathfrak{a}_{i_N}))=\lt'_M\circ \plog \circ \lt_M(\mathfrak{a}_{i_1}'(x)) \quad \text{for any}\,\,x\in D_{\mathfrak{a}_{i_2}\cdots \mathfrak{a}_{i_N}}\cap \Ct.
\end{equation*}
Note $\mathfrak{a}_{i_1}'(x)$ is of the form $t^{m}\sum_{j\geq 0}b_jt^j$ with $m\in \mathbb{Z}, \,b_0\neq 0$. Using the definitions of $\lt'_M, \,\plog, \,\lt_M$, we can show that $\mathrm{det}(I-W(s))$ is a polynomial described in the statement of the corollary.
\end{proof}
\begin{rem*}
We define a graph $G_{N,g}$ with vertices given by 
\begin{equation*}
	\{\mathfrak{a}_{i_1}\cdots\mathfrak{a}_{i_{N-1}}:\mathfrak{a}_{i_j}\mathfrak{a}_{i_{j+1}}\neq 1\}
\end{equation*}
and there is a (directed) edge between two vertices of the form $\mathfrak{a}_{i_1}\cdots \mathfrak{a}_{i_{N-1}}$ and $\mathfrak{a}_{i_2}\cdots \mathfrak{a}_{i_{N}}$. We consider the weighted Ihara zeta function of this graph equipped  the edge between $\mathfrak{a}_{i_1}\cdots \mathfrak{a}_{i_{N-1}}$ and $\mathfrak{a}_{i_2}\cdots \mathfrak{a}_{i_{N}}$ with weight
given by $-\frac{1}{\log(1/|z|)}\Re l_M(\mathfrak{a}_{i_1},(\mathfrak{a}_{i_2}\cdots \mathfrak{a}_{i_N}))$.
By a similar proof as \cref{prop:ZI}, we have $Z_M(\Gamma,z,s)=\det(I-V(s))$ where 
\[ V(\mathfrak{a}_{i_1}\cdots \mathfrak{a}_{i_{N-1}}, \mathfrak{a}_{i_2}\cdots \mathfrak{a}_{i_{N}})(s)=\exp\left(\frac{s}{\log(1/|z|)}\Re l_M(\mathfrak{a}_{i_1},(\mathfrak{a}_{i_2}\cdots \mathfrak{a}_{i_N}))\right), \]
which is smaller and has the advantage of being easier to compute.
\end{rem*}

Continuing with \cref{cor:analytic-middle}, we prove that any two
intermediate zeta functions are related in the following way.
\begin{cor}\label{lem:M'-M}
Fix any $M\in \mathbb{Z}_{\geq 0}$. Let $\mu_j(z)$ be given as in \cref{eq:mu_j-expansion} for $j=1,\ldots,J$. Then we have for any  $M'\in \ZZ_{\geq 0}$ with $M'\leq M$, the holomorphic extension of $Z_{M'}(\Gamma,z,s)$ to $\mathbb{C}$ is given by 
\begin{equation}
	\label{equ:ZM'}
	Z_{M'}(\Gamma,z,s)=P_M(e^{s\mu_1^{M'}(z)},\cdots, e^{s\mu_J^{M'}(z)}).
\end{equation}
where each $\mu_j^{M'}(z)$ is the first $M'$ leading terms of $\mu_j(z)$: 
\begin{equation*}
	\mu_j^{M'}(z)=a_{j}+\frac{1}{\log(1/|z|)}\Re\left(\sum\limits_{k=0}^{M'}a_{j,k} z^k \right)
\end{equation*}
with $a_{j},\,a_{j,k}\in \mathbb{C}$ defined in \cref{eq:mu_j-expansion}. 
Similarly, the holomorphic extension of the
Ihara zeta function to $\mathbb{C}$
is given by
\begin{equation}
	\label{equ:analytic extension of Ihara}
	Z_I(\Gamma,s) = P_M(e^{sa_1},\cdots, e^{sa_J}).
\end{equation}
\end{cor}
\begin{proof}
Recall that in the proof of \cref{prop:lc-cocycle}, we showed that for any reduced word $\gamma = \mathfrak{a}_{i_1}\cdots \mathfrak{a}_{i_N}$, $l_M(\mathfrak{a}_{i_1},(\mathfrak{a}_{i_2}\cdots \mathfrak{a}_{i_N}))$ is of the following form:
\begin{equation*}
	l_M (\mathfrak{a}_{i_1},(\mathfrak{a}_{i_2}\cdots \mathfrak{a}_{i_N}))= m \log(1/t) +\sum\limits_{j=0}^{M}b_j t^j.
\end{equation*}
with $m\in \mathbb{Z}$ and $b_0\neq 0$.
We have the following general formula for combinatorial operations: for $x\in \Ct$, 
\begin{equation*}
	\lt'_{M'}\circ\plog \circ\lt_{M'} (x)=\lt'_{M'}\circ\plog \circ\lt_{M} (x)= \lt'_{M'}\circ\lt'_M\circ\plog \circ\lt_{M} (x) , 
\end{equation*}
where the first equality is due to \cref{equ:ltm'}. By definition of $l_M$, this implies $l_{M'}(\mathfrak{a}_{i_1}, (\mathfrak{a}_{i_2}\cdots \mathfrak{a}_{i_N}))=\lt'_{M'}l_{M}(\mathfrak{a}_{i_1}, (\mathfrak{a}_{i_2}\cdots \mathfrak{a}_{i_N}))$, and hence
\begin{equation*}
	l_{M'}(\mathfrak{a}_{i_1}, (\mathfrak{a}_{i_2}\cdots \mathfrak{a}_{i_N})) = m \log(1/t) +\sum\limits_{j=0}^{M'}b_j t^j.
\end{equation*}
To obtain \cref{equ:ZM'}, 
when applying the proof of \cref{cor:analytic-middle} to $Z_{M'}(\Gamma,z,s)$, we consider an edge matrix $W(s)$ similar to the one for $Z_M(\Gamma,z,s)$ with $l_M$ replaced by $l_{M'}$.

For the holomorphic extension of the Ihara zeta function, recall that for any $\gamma\in \Gamma\backslash \{id\}$, $\ell^{na}(\gamma)\log(1/t)=\lt_{-1}'(L_{M}(\gamma))$. \cref{rem:analytic Ihara} allows us to apply the proof of \cref{cor:analytic-middle} to $Z_I(\Gamma,s)$ with $l_{M}(\mathfrak{a}_{i_1},(\mathfrak{a}_{i_2}\cdots\mathfrak{a}_{i_N}))$ replaced by $\lt'_{-1}\left(l_{M}(\mathfrak{a}_{i_1},(\mathfrak{a}_{i_2}\cdots\mathfrak{a}_{i_N}))\right)$, which yields \cref{equ:analytic extension of Ihara}.
\end{proof}

\begin{prop}\label{prop:expansion-der-cycl}
There exists $C>0$ depending on $\Gamma$ such that for any reduced word $\gamma = \mathfrak{a}_{i_1}\cdots \mathfrak{a}_{i_N}$, we have
\begin{equation}\label{eq:expansion-derivative-cocycle}
	\mathfrak{a}_{i_1}'(\mathfrak{a}_{i_2}\cdots\mathfrak{a}_{i_N}(\infty))=t^m \sum\limits_{j=0}^{\infty} a_j t^j,
\end{equation}
where $a_j$'s are complex numbers satisfying
\begin{equation*}
	|a_0|\geq e^{-C} \quad \text{and}\quad |a_j|\leq e^{C(N+j)}\quad\text{for}\quad j\in \mathbb{Z}_{\geq 0}.
\end{equation*}
Moreover, we have for any $M\in \ZZ_{\geq 0}$ and $N\geq C_1(M+1)$ (from \cref{prop:lc-cocycle}),
\begin{equation}\label{eq:expansion-lM}
	l_M(\mathfrak{a}_{i_1},(\mathfrak{a}_{i_2}\cdots \mathfrak{a}_{i_N})) =m \log (1/t)+\sum_{j=0}^{M}b_j t^j,
\end{equation}
where $b_j$'s are complex numbers satisfying
\begin{equation*}
	|b_j|\leq e^{C(N+1)j}\quad\text{for}\quad j=0,\ldots, M.
\end{equation*}
\end{prop}
For the proof, please see \cref{sec.proof}.
\begin{proof}[Proof of \cref{lem:ellM>}]
By \cref{equ:LM and length expansion} and \cref{prop:lc-cocycle}, there exists $C_1>0$ depending only on $\Gamma$ such that for $N=\lceil C_1(M+2) \rceil$ and for any cyclically reduced word $\gamma=\mathfrak{a}_{i_1}\cdots \mathfrak{a}_{i_n}$ with $n>N$, we have
\begin{equation*}
	\log(1/|z|) \ell_M(\gamma,z)=-\Re \left(\sum_{k=1}^{n}l_M(\mathfrak{a}_{i_k},(\mathfrak{a}_{i_{k+1}}\cdots \mathfrak{a}_{i_{k+N-1}}))(z)\right).
\end{equation*}
By \cref{prop:expansion-der-cycl}, there exists $C_2>0$ depending only on $\Gamma$, such that for $k=1,\ldots,n$
\begin{equation*}
	l_M(\mathfrak{a}_{i_k},(\mathfrak{a}_{i_{k+1}}\cdots \mathfrak{a}_{i_{k+N-1}}))-\lt'_{-1}\left(l_M(\mathfrak{a}_{i_k},(\mathfrak{a}_{i_{k+1}}\cdots \mathfrak{a}_{i_{k+N-1}}))\right)
	=\sum_{j=0}^{M} b_jt^j
\end{equation*}
where each $b_j$ satisfies $|b_j|\leq e^{C_2(N+1)j}$. This implies there exists $C_3>0$ depending only on $\Gamma$, such that for any $|z|< e^{-2C_2(N+1)}\leq e^{-2C_2(C_1(M+2)+1)}$, we have 
\begin{equation*}
	\left|\sum_{j=0}^Mb_j z^j \right|\leq C_3.
\end{equation*}
Recall \cref{rem:analytic Ihara}. Combining these together, we obtain $C>0$ depending only on $\Gamma$ such that for
any $0<|z|<e^{-C(M+1)}$, 
\begin{align*}
	\log (1/|z|) \ell_M(\gamma,z)\geq \log (1/|z|)\ell^{na}(\gamma)-nC_3\geq \frac{1}{2} \log(1/|z|) \ell^{na}(\gamma),
\end{align*}
where the second inequality is due to \cref{lem:word-na}.

For any reduced word $\gamma=\mathfrak{a}_{i_1}\cdots \mathfrak{a}_{i_n}$ with 
$n\leq N=\lceil C_1(M+1)\rceil$, we can prove the statement for $\gamma$ by using \cref{lem:word-na} and \cref{prop:expansion of length function}.
\end{proof}

\section{Examples}

Recall that $\Gamma<\SL_2(\mathbb{M}(\mathbb{D}))$ is a family of Schottky groups satisfying $(\bigstar)$.

\subsection{Two-generator case}
In this section, we discuss the case when  $\Gamma$ is generated by two generators to illustrate the $0$-th intermediate zeta function. 
 In this case, we can obtain explicit expressions for the leading coefficient $A_0(\gamma)$ using the Fricke relation of $\SL_2(\Ct)$:
\begin{equation}\label{equ:fricke}
\tr(gh)+\tr(gh^{-1})=\tr(g)\tr(h), \quad g,h\in\SL_2(\Ct).
\end{equation}

For convenience, we denote the leading term of the trace of $\gamma\in \Gamma\backslash \{id\}$ as
\begin{equation*}
\lt(\gamma):=\lt_0(\tr(\gamma))=A_0(\gamma) t^{-\ell^{na}(\gamma)/2}
\end{equation*}
(see \cref{Laurent trace} for the Laurent expansion of $\tr(\gamma)$). Recall the $0$-th expansion of length, $\ell_0(\gamma,t)$ defined in \cref{eqn:Mth expansion of length}.

\begin{lem}For any $\gamma\in \Gamma\backslash \{id\}$, 
we have
\begin{equation*}
	\ell_0(\gamma,z)=2\log|\lt(\gamma)|/\log(1/|z|),
\end{equation*}
and hence
\begin{equation*}
	Z_0(\Gamma,z,s)=\prod_{[\gamma]\in \calP}(1-e^{-2s\log|\lt(\gamma)|/\log(1/|z|)}). 
\end{equation*}
\end{lem}
\begin{proof}
Given any $\gamma\in \Gamma\backslash \{id\}$, recall \cref{prop:laurent-est} gives
\begin{equation*}
	\tr(\gamma) =t^{-\ell^{na}(\gamma)/2}(A_0(\gamma)+A_1(\gamma) t+\cdots),
\end{equation*}
and \cref{prop:expansion of length function} and \cref{equ:obtaining lz using lambdaz} give
\begin{equation*}
	\ell(\gamma_z)/\log(1/|z|)=\ell^{na}(\gamma)+2\log|A_0(\gamma)|/\log(1/|z|)+\Re\left(\sum\limits_{j\geq 1}a_j(\gamma)z^j/\log(1/|z|)\right).
\end{equation*}
The proof is complete.
\end{proof}

When the Schottky group $\Gamma$ is generated by two generators, i.e. $\Gamma=\langle \mathfrak{a}_1, \mathfrak{a}_2, \mathfrak{a}_3, \mathfrak{a}_4\rangle$ where $\mathfrak{a}_1\mathfrak{a}_3=\mathfrak{a}_2\mathfrak{a}_4=id$,
we have three different possibilities (\cref{fig:dumbbell,fig:eight,fig:theta}) for the corresponding finite graph $\Sigma_X$ that we will discuss below. (In the \cref{fig:dumbbell,fig:eight,fig:theta}, we use $\frak{a}$ and $\frak{b}$ to denote two generators and $\frak{a}_{\pm}$ for its attracting and repelling fixed points in $\P_{\Ct}^1$. The tree on the left is the subtree connecting these points in $\P_{\Ct}^{1,an}$. The graph on the right is the corresponding Mumford curve of the Schottky group.)

\begin{figure}
\begin{center}
	\def\svgwidth{450bp}   %% Creator: Inkscape 1.1.1 (c3084ef, 2021-09-22), www.inkscape.org
%% Accompanies image file 'dumbbell.pdf' (pdf, eps, ps)
\begingroup%
  \makeatletter%
  \providecommand\color[2][]{%
    \errmessage{(Inkscape) Color is used for the text in Inkscape, but the package 'color.sty' is not loaded}%
    \renewcommand\color[2][]{}%
  }%
  \providecommand\transparent[1]{%
    \errmessage{(Inkscape) Transparency is used (non-zero) for the text in Inkscape, but the package 'transparent.sty' is not loaded}%
    \renewcommand\transparent[1]{}%
  }%
  \providecommand\rotatebox[2]{#2}%
  \newcommand*\fsize{\dimexpr\f@size pt\relax}%
  \newcommand*\lineheight[1]{\fontsize{\fsize}{#1\fsize}\selectfont}%
  \ifx\svgwidth\undefined%
    \setlength{\unitlength}{533.4129053bp}%
    \ifx\svgscale\undefined%
      \relax%
    \else%
      \setlength{\unitlength}{\unitlength * \real{\svgscale}}%
    \fi%
  \else%
    \setlength{\unitlength}{\svgwidth}%
  \fi%
  \global\let\svgwidth\undefined%
  \global\let\svgscale\undefined%
  \makeatother%
  \begin{picture}(1,0.38525153)%
    \lineheight{1}%
    \setlength\tabcolsep{0pt}%
    \put(0,0){\includegraphics[width=\unitlength,page=1]{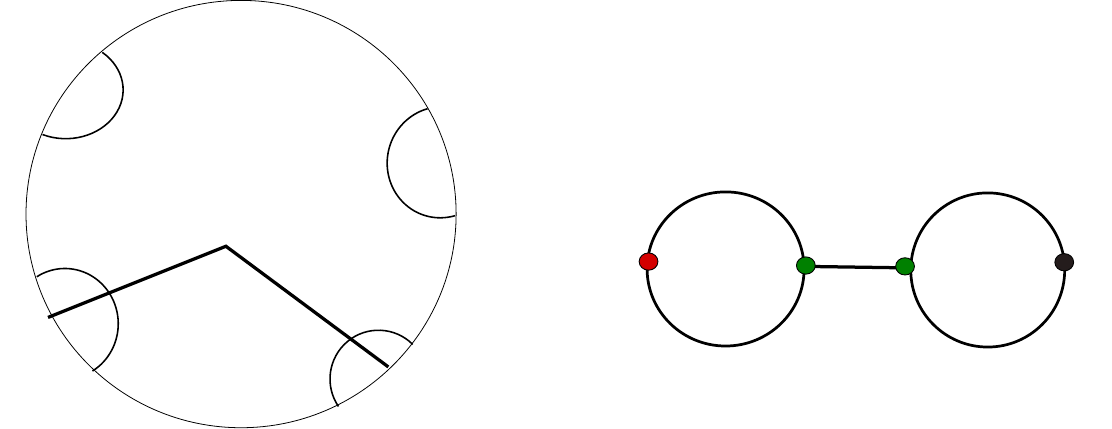}}%
    \put(0.01243825,0.32134679){\makebox(0,0)[lt]{\lineheight{1.25}\smash{\begin{tabular}[t]{l}$\mathfrak{a}_+$\end{tabular}}}}%
    \put(0.42006881,0.25877893){\makebox(0,0)[lt]{\lineheight{1.25}\smash{\begin{tabular}[t]{l}$\mathfrak{a}_-$\end{tabular}}}}%
    \put(-0.00215436,0.1173474){\makebox(0,0)[lt]{\lineheight{1.25}\smash{\begin{tabular}[t]{l}$\mathfrak{b}_-$\end{tabular}}}}%
    \put(0.34870909,0.02877856){\makebox(0,0)[lt]{\lineheight{1.25}\smash{\begin{tabular}[t]{l}$\mathfrak{b}_+$\end{tabular}}}}%
    \put(0,0){\includegraphics[width=\unitlength,page=2]{dumbbell.pdf}}%
    \put(0.60607326,0.1537499){\makebox(0,0)[lt]{\lineheight{1.25}\smash{\begin{tabular}[t]{l}$h_1$\end{tabular}}}}%
    \put(0.91256182,0.15099352){\makebox(0,0)[lt]{\lineheight{1.25}\smash{\begin{tabular}[t]{l}$h_2$\\\end{tabular}}}}%
    \put(0.7558392,0.16226967){\makebox(0,0)[lt]{\lineheight{1.25}\smash{\begin{tabular}[t]{l}$h_3$\\\end{tabular}}}}%
  \end{picture}%
\endgroup%
	
	\caption{Dumbbell graph}
	\label{fig:dumbbell}
\end{center}
\end{figure}

\begin{figure}
\begin{center}
	\def\svgwidth{450bp}   %% Creator: Inkscape 1.1.1 (c3084ef, 2021-09-22), www.inkscape.org
%% Accompanies image file 'eight.pdf' (pdf, eps, ps)
\begingroup%
  \makeatletter%
  \providecommand\color[2][]{%
    \errmessage{(Inkscape) Color is used for the text in Inkscape, but the package 'color.sty' is not loaded}%
    \renewcommand\color[2][]{}%
  }%
  \providecommand\transparent[1]{%
    \errmessage{(Inkscape) Transparency is used (non-zero) for the text in Inkscape, but the package 'transparent.sty' is not loaded}%
    \renewcommand\transparent[1]{}%
  }%
  \providecommand\rotatebox[2]{#2}%
  \newcommand*\fsize{\dimexpr\f@size pt\relax}%
  \newcommand*\lineheight[1]{\fontsize{\fsize}{#1\fsize}\selectfont}%
  \ifx\svgwidth\undefined%
    \setlength{\unitlength}{516.38697934bp}%
    \ifx\svgscale\undefined%
      \relax%
    \else%
      \setlength{\unitlength}{\unitlength * \real{\svgscale}}%
    \fi%
  \else%
    \setlength{\unitlength}{\svgwidth}%
  \fi%
  \global\let\svgwidth\undefined%
  \global\let\svgscale\undefined%
  \makeatother%
  \begin{picture}(1,0.39795376)%
    \lineheight{1}%
    \setlength\tabcolsep{0pt}%
    \put(0,0){\includegraphics[width=\unitlength,page=1]{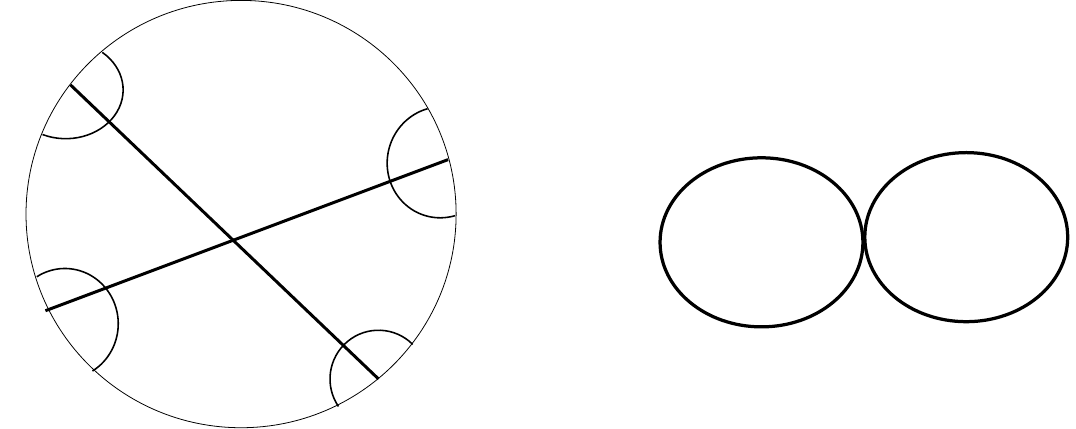}}%
    \put(0.00703877,0.32903723){\makebox(0,0)[lt]{\lineheight{1.25}\smash{\begin{tabular}[t]{l}$\mathfrak{a}_+$\end{tabular}}}}%
    \put(0.35928579,0.02207185){\makebox(0,0)[lt]{\lineheight{1.25}\smash{\begin{tabular}[t]{l}$\mathfrak{a}_-$\end{tabular}}}}%
    \put(-0.0022254,0.11250154){\makebox(0,0)[lt]{\lineheight{1.25}\smash{\begin{tabular}[t]{l}$\mathfrak{b}_-$\end{tabular}}}}%
    \put(0.4302378,0.25814212){\makebox(0,0)[lt]{\lineheight{1.25}\smash{\begin{tabular}[t]{l}$\mathfrak{b}_+$\end{tabular}}}}%
    \put(0,0){\includegraphics[width=\unitlength,page=2]{eight.pdf}}%
  \end{picture}%
\endgroup%
	
	\caption{Figure eight graph}
	\label{fig:eight}
\end{center}
\end{figure}

\begin{figure}
\begin{center}
	
	\def\svgwidth{450bp}   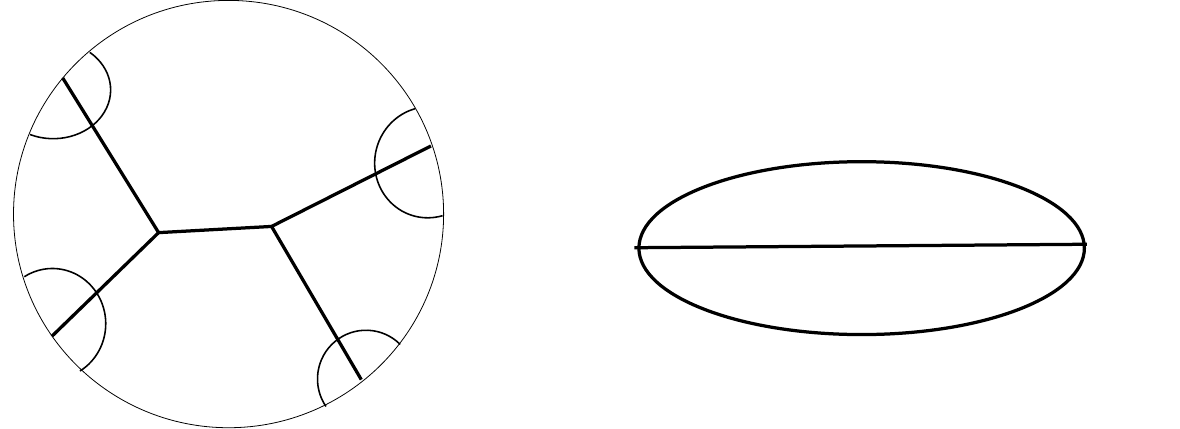	
	\caption{Figure $\Theta$ graph}
	\label{fig:theta}
\end{center}
\end{figure}

We will need
\begin{lem}
Assume $\Gamma=\langle \frak{a}_1,\ldots,\frak{a}_4 \rangle<\SL_2(\mathbb{M}(\mathbb{D}))$ is a family of Schottky groups satisfying $(\bigstar)$.\footnote{This lemma also works for a general Schottky group in $\SL_2(\Ct)$.}
The following identity holds for any cyclically reduced word:
\begin{equation}\label{equ:identity-2gen}
	\lt(\mathfrak{a}_{i_1}\mathfrak{a}_{i_2}\cdots \mathfrak{a}_{i_n})^2=\prod\limits_{j=1}^{n}\lt(\mathfrak{a}_{i_{j}}\mathfrak{a}_{i_{j+1}})
\end{equation}
with the convention that $i_{n+1}=i_1$.

\end{lem}

\begin{proof}

Suppose we have a cyclically reduced word $\mathfrak{a}_1\gamma_1\mathfrak{a}_1\gamma_2$. By \cref{equ:fricke}, we have
\[ \tr(\mathfrak{a}_1\gamma_1\mathfrak{a}_1\gamma_2)=\tr(\mathfrak{a}_1\gamma_1)\tr(\mathfrak{a}_1\gamma_2)-\tr(\gamma_2^{-1}\gamma_1). \]
Claim: We have
\begin{equation}
    \label{equ:claim-2gen}
|\tr(\gamma_2^{-1}\gamma_1)|_{na}<|\tr(\mathfrak{a}_1\gamma_1)\tr(\gamma_2^{-1}\mathfrak{a}_1^{-1})|_{na}=|\tr(\mathfrak{a}_1\gamma_1)\tr(\mathfrak{a}_1\gamma_2)|_{na}.
\end{equation}
Proof of the claim: The claim only works for the two-generator case, and we give an ad-hoc proof.

For the figure eight graph, by using the graph and $2\log |\tr(\gamma)|_{na}=\ell^{na}(\gamma)$, the length of the non-backtracting loop of $\gamma$, we have
\[|\tr( \mathfrak{a}_1\gamma_1)|_{na}>|\tr(\gamma_1)|_{na}. \]
Therefore
\begin{align*}
	|\tr(\gamma_2^{-1}\gamma_1)|_{na}&=\exp(\ell^{na}(\gamma_2^{-1}\gamma_1)/2)\leq \exp((\ell^{na}(\gamma_2^{-1})+\ell^{na}(\gamma_1))/2)\\
	&= |\tr(\gamma_1)|_{na}|\tr(\gamma_2^{-1})|_{na}<|\tr(\mathfrak{a}_1\gamma_1)|_{na}|\tr(\gamma_2^{-1}\mathfrak{a}_1^{-1})|_{na}.  
\end{align*}

For the figure $\Theta$ graph (\cref{fig:theta}): let $h_1,h_2,h_3$ be the lengths of the three edges $e_1,e_2,e_3$ connecting two vertices $v_1,v_2$. For each generator, we use a loop with the starting point $v_1$ in the graph. For example the loops of $\frak a_1,\frak a_2$ are represented by $e_1e_5=e_1e_2^{-1}$ and $e_2e_6=e_2e_3^{-1}$, respectively. For a cyclically reduced word $\gamma$, we connect the loops corresponding to each letters through $v_1$ to get a loop of $\gamma$. The only possible backtracking part in the loop is the middle edge in the graph $\Theta$, which corresponds to the pairs $\frak{a}_1\frak{a}_2$ and $\frak{a}_4\frak{a}_3$ in $\gamma$. We have that
\begin{equation}\label{equ:ell theta}
	\ell^{na}(\gamma)=n_1(\gamma)(h_1+h_2)+n_2(\gamma)(h_2+h_3)-n_3(\gamma)(2h_2),     
\end{equation}
with $n_1(\gamma)$ the number of $\frak{a}_1,\frak{a}_3$ in $\gamma$, $n_2(\gamma)$ the number of $\frak{a}_2,\frak{a}_4$, and $n_3(\gamma)$ the number of pairs $\frak{a}_1\frak{a}_2$ and $\frak{a}_4\frak{a}_3$ in $\gamma$\footnote{For example $\gamma=\frak a_{i_1}\cdots \frak a_{i_N}$, the pair $\frak a_{i_N}\frak a_{i_1}$ should also be considered.}. 
\begin{lem}\label{equ:two-generator}
	For cyclically reduced $\gamma=\frak{a}_{i_1}\cdots \frak{a}_{i_N}$, we have 
\footnote{It is possible that the word is simply $\frak{a}_{i_1}$, then we use that $2\ell^{na}(\frak{a}_{i_1})=\ell^{na}(\frak{a}_{i_1}\frak{a}_{i_1})$.}
	\begin{equation}\label{equ:cyc}
		2\ell^{na}(\gamma)=\sum_{1\leq j\leq N}\ell^{na}(\frak{a}_{i_j}\frak{a}_{i_{j+1}}) . 
	\end{equation}
\end{lem}
\begin{proof}
It suffices to check the equality for $n_1,n_2,n_3$ in \cref{equ:ell theta}. The corresponding equality for $n_1,n_2$ is trivial. For $n_3$, let $n_4$ be the number of pairs of $\frak{a}_2\frak{a}_1$ and $\frak{a}_3\frak{a}_4$ in $\gamma$. Let $I=\{\frak{a}_1,\frak{a}_4 \}$ and $II=\{\frak{a}_2,\frak{a}_3 \}$. A cyclically reduced word can be represented as loop on the graph with vertices $\{\frak{a}_i , 1\leq i\leq 4\}$ and all the edges are allowed except the edges $\frak{a}_i\frak{a}_{i+2}$. Then $n_3$ is the number of edges in the loop from vertices in $I$ to $II$. $n_4$ is the number of edges in the loop from vertices in $II$ to $I$. Since $\gamma$ represents a loop, we have $n_3(\gamma)=n_4(\gamma)$. Notice that $n_3(\frak{a}_2\frak{a}_1)=n_3(\frak{a}_1\frak{a}_2)=1$. We have
	\[ 2n_3(\gamma)=n_3(\gamma)+n_4(\gamma)=\sum_{1\leq  j\leq N} n_3(\frak{a}_{i_j}\frak{a}_{i_{j+1}}), \]
	where the second equality is because both sides represent the number of $\frak{a}_1\frak{a}_2$, $\frak{a}_4\frak{a}_3$, $\frak{a}_2\frak{a}_1$, $\frak{a}_3\frak{a}_4$ in $\gamma$. The proof is complete.
\end{proof}
For a non-cyclically reduced word $\gamma$, we also have
\begin{equation}\label{equ:non-cyc}
	2\ell^{na}(\gamma)\leq \sum_{1\leq j\leq N}\ell^{na}(\frak{a}_{i_j}\frak{a}_{i_{j+1}}), 
\end{equation}
since in order to obtain a cyclically reduced representation of $\gamma$, we are doing induction and we only need to treat the case such as $\cdots \mathfrak{b}_1\frak{a}_1\frak{a}_1^{-1}\mathfrak{b}_2\cdots$. 

We go back to the claim \cref{equ:claim-2gen}. For any $\mathfrak{b}_1,\mathfrak{b}_2$ in the set of generators, we have
\begin{equation}\label{equ:reduce}
	\ell^{na}(\mathfrak{b}_1\frak{a}_1)+\ell^{na}(\frak{a}_1^{-1}\mathfrak{b}_2)\geq \ell^{na}(\mathfrak{b}_1\mathfrak{b}_2),
\end{equation}
which is true by considering all $16$ possibilities of $\mathfrak{b}_1,\mathfrak{b}_2$, and the equality holds if and only if 
\begin{equation}\label{equ:b1b2}
	\mathfrak{b}_1=\frak{a}_3 \text{ or }\mathfrak{b}_2=\frak{a}_1.
\end{equation}
Therefore suppose $\frak{a}_1\gamma_1=\frak{a}_1\mathfrak{b}_1\cdots \mathfrak{b}_2$ and $\gamma_2^{-1}\frak{a}_1^{-1}=\mathfrak{b}_3\cdots \mathfrak{b}_4\frak{a}_1 ^{-1} $ with $\mathfrak{b}_i$ from the set of generators, due to \cref{equ:cyc} and \cref{equ:non-cyc}
\begin{align*}
	&2\ell^{na}(\frak{a}_1\gamma_1)+2\ell^{na}(\gamma_2^{-1}\frak{a}_1^{-1})-2\ell^{na}(\gamma_2^{-1}\gamma_1)\\
	\geq &\ \ell^{na}(\frak{a}_1\mathfrak{b}_1)+\ell^{na}(\mathfrak{b}_2\frak{a}_1)+\ell^{na}(\frak{a}_1^{-1}\mathfrak{b}_3)+\ell^{na}(\mathfrak{b}_4\frak{a}_1^{-1})-\ell^{na}(\mathfrak{b}_4\mathfrak{b}_1)-\ell^{na}(\mathfrak{b}_2\mathfrak{b}_3)\\
	= &  \ell^{na}(\mathfrak{b}_1\frak{a}_1)+\ell^{na}(\frak{a}_1^{-1}\mathfrak{b}_4)-\ell^{na}(\mathfrak{b}_1\mathfrak{b}_4)+\ell^{na}(\mathfrak{b}_2\frak{a}_1)+\ell^{na}(\frak{a}_1^{-1}\mathfrak{b}_3)-\ell^{na}(\mathfrak{b}_2\mathfrak{b}_3)\geq 0, 
\end{align*} 
where the last inequality is due to \cref{equ:reduce}. Since $\frak{a}_1\gamma_1\frak{a}_1\gamma_2$ is cyclically reduced, we obtain $\mathfrak{b}_1,\mathfrak{b}_2\neq \frak{a}_3$ and $\mathfrak{b}_3,\mathfrak{b}_4\neq \frak{a}_1$. Therefore neither $(\mathfrak{b}_1,\mathfrak{b}_4)$ nor $(\mathfrak{b}_2,\mathfrak{b}_3)$ satisfies \cref{equ:b1b2}. Hence the equality of the last equation cannot hold and the proof is complete.

For the dumbbell graph (\cref{fig:dumbbell}), we can prove by the same method as the figure $\Theta$ graph: first we show \cref{equ:cyc} holds, then \cref{equ:reduce} holds, finally we can conclude \cref{equ:claim-2gen}.

Back to the proof: Hence \footnote{This is not a general equation, for three generators with graph $(||)$ consisting of two vertices and four edges between them, the sum of length of $\mathfrak{a}_1\mathfrak{a}_2$ and $\mathfrak{a}_2^{-1}\mathfrak{a}_3^{-1}$ equals the length of $\mathfrak{a}_1\mathfrak{a}_3^{-1}$.}
\begin{equation}\label{equ:gamma a gamma a}
	\lt(\mathfrak{a}_1\gamma_1\mathfrak{a}_1\gamma_2)=\lt(\mathfrak{a}_1\gamma_1)\lt(\mathfrak{a}_1\gamma_2). 
\end{equation}
Moreover, if $\mathfrak{a}_1\gamma_1$ and $\mathfrak{a}_1\gamma_2$ satisfy \cref{equ:identity-2gen}, then by \cref{equ:gamma a gamma a}, the longer word $\mathfrak{a}_1\gamma_1\mathfrak{a}_1\gamma_2$ also satisfies \cref{equ:identity-2gen}.

For any cyclically reduced word of length greater than $5$, we can find some generator that appears at least twice. Therefore we can use \cref{equ:gamma a gamma a} to reduce the length of the word and do induction. 

The only cyclically reduced words without any generator appearing twice are the commutator $\mathfrak{a}_1\mathfrak{a}_2\mathfrak{a}_3\mathfrak{a}_4$ and $\mathfrak{a}_i\mathfrak{a}_j$, $\mathfrak{a}_i$. For the commutator, we have the well-known formula from the Fricke relation
\[ \tr(\mathfrak{a}_1\mathfrak{a}_2\mathfrak{a}_1^{-1}\mathfrak{a}_2^{-1})=\tr(\mathfrak{a}_1)^2+\tr(\mathfrak{a}_2)^2+\tr(\mathfrak{a}_1\mathfrak{a}_2)^2-\tr(\mathfrak{a}_1)\tr(\mathfrak{a}_2)\tr(\mathfrak{a}_1\mathfrak{a}_2)-2=\tr(\mathfrak{a}_1)^2+\tr(\mathfrak{a}_2)^2-\tr(\mathfrak{a}_1\mathfrak{a}_2)\tr(\mathfrak{a}_1\mathfrak{a}_2^{-1})-2 \]
By checking the three graphs of two generator cases, we have $|\tr(\mathfrak{a}_1)^2|_{na}=|\tr(\mathfrak{a}_1^2)|_{na}<|\tr(\mathfrak{a}_1\mathfrak{a}_2)\tr(\mathfrak{a}_2^{-1}\mathfrak{a}_1)|_{na}$.
Hence
\[\lt(\mathfrak{a}_1\mathfrak{a}_2\mathfrak{a}_1^{-1}\mathfrak{a}_2^{-1})=-\lt(\mathfrak{a}_1\mathfrak{a}_2)\lt(\mathfrak{a}_1\mathfrak{a}_2^{-1}). \]

The proof is complete.
\end{proof}

\subsubsection{Figure $\Theta$ graph}\label{subsec:fig-theta}

We first consider the case that the graph $\Sigma_X$ is like $\Theta$ in \cref{fig:theta}. Without loss of generality, we suppose each edge of the graph $\Sigma_X$ has length $2$ (for example the symmetric three-funnel case). Suppose 
$\lt(\mathfrak{a}_1)=Az^{-2}$, $\lt(\mathfrak{a}_2)=B z^{-2}$ and $\lt(\mathfrak{a}_1\mathfrak{a}_2)=Cz^{-2}$ with $A,B,C\in\C^*$. For the weighted graph, we can define three weights by
\begin{equation*}
h_j=2+\frac{2\alpha_j}{\log(1/|z|)},
\end{equation*}
with $\alpha_j$'s determined by
\begin{equation*}
\alpha_1+\alpha_2=\log|A|,\quad \alpha_2+\alpha_3=\log|B|,\quad \alpha_1+\alpha_3=\log|C|.
\end{equation*}
For a reduced word $\gamma$ of word length two, from the choice of $\alpha_j$ we can verify that the weighted length $\ell_h$ satisfies
\begin{equation}\label{equ:ell-h-gamma}
\ell_h(\gamma)= 2\log |\lt(\gamma)|/\log(1/|z|)=\ell_0(\gamma,z).
\end{equation}
For example for the word $\frak{a}_1\frak{a}_2^{-1}$, from the Fricke relation, we have $\tr(\frak{a}_1\frak{a}_2^{-1})=\tr(\frak{a}_1)\tr(\frak{a}_2)-\tr(\frak{a}_1\frak{a}_2)$, then $\lt(\frak{a}_1\frak{a}_2^{-1})=\lt(\frak{a}_1)\lt(\frak{a}_2)$ which implies $\ell_0(\frak{a}_1\frak{a}_2^{-1},z)=h_1+2h_2+h_3=\ell_h(\frak{a}_1\frak{a}_2^{-1})$.
By \cref{equ:identity-2gen}, we know this relation also holds for any cyclically reduced $\gamma$.
Therefore, the weighted Ihara zeta function of the weight $h$ gives the intermediate zeta function $Z_0$.

For this example, by \cref{prop:ZI}, the zeta function $Z_0(\Gamma,z,s)$ can be computed and is given by $(-2abc + a^2 + b^2 + c^2 - 1) (2abc + a^2 + b^2 + c^2 - 1)$ with $a=e^{-s(h_2+h_3)/2}, b=e^{-s(h_1+h_3)/2}, c=e^{-s(h_1+h_2)/2}$.

\subsubsection{Dumbbell graph}
The case that the graph $\Sigma_X$ is like a dumbbell as in \cref{fig:dumbbell} is similar to \cref{subsec:fig-theta}.  Suppose each edge of the graph has length two. For example, $\lt(\mathfrak{a}_1)=Az^{-1}$, $\lt(\mathfrak{a}_2)=Bz^{-1}$ and $\lt(\mathfrak{a}_1\mathfrak{a}_2)=Cz^{-4}$ with $A,B,C\in\C^*$. We define the three weights by
\begin{equation*}
h_j=2+\frac{2\beta_j}{\log(1/|z|)},
\end{equation*}
with $\beta_j$'s determined by
\begin{equation*}
\beta_1=\log|A|,\quad \beta_{2}=\log|B|,\quad \beta_1+2\beta_{3}+\beta_{2}=\log|C|.
\end{equation*}
Similar to the previous case, the weighted Ihara zeta function equals the intermediate zeta function $Z_0$.

For this example, by \cref{prop:ZI} the zeta function $Z_0(\Gamma,z,s)$ can be computed and is given by $-(c - 1)(c + 1)(a - 1)(a + 1)(4a^2b^4c^2 - a^2c^2 + a^2 + c^2 - 1)$ with $a=e^{-sh_1/2}, b=e^{-sh_{3}/2},c=e^{-sh_{2}/2}$. For $t=0$, $s=\frac{1}{2}\log 2$ is the solution with the maximal real part.
\begin{exam}\label{exa.o-o}
If $A=B=C=1$, then the intermediate zeta function coincides with the Ihara zeta function.
\end{exam}

\subsubsection{Figure eight graph}\label{sec:figure-eight}

The case when the graph $\Sigma_X$ is like figure $8$ (as shown in \cref{fig:eight}) is not directly applicable for weighted Ihara zeta function, since \cref{equ:ell-h-gamma} does not hold. But using the identity \cref{equ:identity-2gen}, we can compute
\begin{equation*}
\ag{
	-\log Z_0(\Gamma,z,s)&=\sum\limits_{j=1}^{\infty}\sum\limits_{[\gamma]\in\calP}\frac{1}{j}e^{-2js\log|\lt(\gamma)|/\log(1/|z|)}\\
	&=\sum\limits_{n=1}^{\infty}\frac{1}{n}\sum\limits_{j\#(\gamma)=n} e^{-js\log|\lt(\gamma)^2|/\log(1/|z|)}
}
\end{equation*}
Consider the matrix $W\in M_{12\times 12}$ given by
$W(\mathfrak{a}_{i_1}\mathfrak{a}_{i_2}, \mathfrak{a}_{i_2}\mathfrak{a}_{i_3})=\exp\left({-s\frac{\log |\lt(\mathfrak{a}_{i_1}\mathfrak{a}_{i_2})|+\log|\lt(\mathfrak{a}_{i_2}\mathfrak{a}_{i_3})|}{2\log(1/|z|)}}\right)$, then the above expression is equal to 
\begin{equation*}
\sum\limits_{n=1}^{\infty}\frac{1}{n} \tr W^n=-\log \det(I-W).
\end{equation*}
\begin{rem*}
This is indeed the path zeta function in Section 4 in \cite{hortonWhatAreZeta2006}.
Consider the matrix $V\in M_{4\times 4}$ given by
$V(\mathfrak{a}_{i_1},\mathfrak{a}_{i_2})=\exp\left({-s\frac{\log |\lt(\mathfrak{a}_{i_1}\mathfrak{a}_{i_2})|}{\log(1/|z|)}}\right)$ for $\mathfrak{a}_{i_1}\mathfrak{a}_{i_2}\neq id$. Then the above expression is also equal to
\begin{equation*}
	\sum\limits_{n=1}^{\infty}\frac{1}{n} \tr V^n=-\log \det(I-V).
\end{equation*}
\end{rem*}

\begin{exam}\label{exa.funnel-torus}
Example of funneled torus in \cite[Section 16]{borthwickSpectralTheoryInfiniteArea2016}. Here the matrices are given by 
\[ S_1=\begin{pmatrix}
	e^{\ell/2} & 0 \\ 0 & e^{-\ell/2}
\end{pmatrix},\ S_2=\begin{pmatrix}
	\cosh(\ell/2)-\cos\phi\ \sinh(\ell/2) & \sin^2\phi\ \sinh(\ell/2) \\ \sinh(\ell/2) & \cosh(\ell/2)+\cos\phi\ \sinh(\ell/2)
\end{pmatrix}  \]
with $\phi$ the angle between the two shortest geodesics.
We do the change of variable with $z=e^{-\ell/2}$ and obtain a Schottky family. \footnote{The verification of condition $(\bigstar)$ can be done similarly as in \cref{exa.borth}. }
Moreover, 
\[ \lt(S_1)=z^{-1},\ \lt(S_2)=z^{-1},\ \lt(S_1S_2)=(1+\cos\phi)z^{-2}/2. \]
The graph is figure eight with edge length $2$. 

From the first paragraph of \cref{sec:figure-eight}, with $\phi=\pi/2$ the middle zeta function can be computed and is given by 
\[ Z_0(\Gamma,z,s)= -(-a + 2b + 1)  (a + 2b - 1)  (a - 1)^2, \]
with $a=e^{-2s},\ b=e^{-s(2-\log 2/\log(1/|z|))}$
. Hence the Hausdorff dimension $\delta_z$ of the limit set is given by the first zero of
\[ e^{-2s}+2e^{-s(2-\log 2/\log(1/|z|)}-1=0 \]
dividing by $\log(1/|z|)$.
For the case $\ell=10$ or $z=e^{-5}$, the first zero of $Z_0(\Gamma,e^{-5},s)$ dividing by $\log(1/|z|)$ is close to $0.115$, which is close to numerics in \cite[Fig 16.8]{borthwickSpectralTheoryInfiniteArea2016}. The main term $\frac{\log 3}{10}$ of \cite{dangHausdorffDimension2024} gives $0.1099$.

\end{exam}

\begin{rem}
In \cite[Theorem 1.1]{weichResonanceChainsGeometric2015}, the author considered the case of three funnels with width given by $n_1\ell,n_2\ell,n_3\ell$ with $n_1,n_2,n_3\in\N$ and $\ell\rightarrow \infty$. Under an extra triangle condition $n_i+n_j>n_k$ with $\{i,j,k\}=\{1,2,3\}$, the convergence to the Ihara zeta function is proved in \cite{weichResonanceChainsGeometric2015}. Actually, the triangle condition says that the graph is figure $\Theta$. If $n_1+n_2=n_3$, then the graph is figure eight. If $n_1+n_2<n_3$, then the graph is figure dumbbell. For all these three cases, by the change of variable with $z=e^{-\ell/2}$, our computation of Ihara zeta functions/intermediate zeta function $Z_0$ and the convergence all work.

We also remark that the type of graph (figure eight, dumbbell, and figure $\Theta$) depends on the choice of the lengths of the generators, while the topological type of the surface depends on the configuration of the Schottky discs. In the case of two generators $\mathfrak{a}_1$ and $\mathfrak{a}_2$, if $D_{\mathfrak{a}_1}$ and $D_{\mathfrak{a}_1^{-1}}$ are adjacent (\cref{fig:dumbbell,fig:theta}), then the hyperbolic surface is a three-funnel surface; if $D_{\mathfrak{a}_1}$ and $D_{\mathfrak{a}_1^{-1}}$ are separated by $D_{\mathfrak{a}_2}$ and $D_{\mathfrak{a}_2^{-1}}$ (\cref{fig:eight}), then the hyperbolic surface is a funneled torus.

\end{rem}

 \subsection{Symmetric three-funnel surface}

In this section, we discuss how our main theorem applies to the symmetric three-funnel surface. In particular, we recover the result of \cite{pollicottZerosSelbergZeta2019a}.

\begin{exam}\label{exa.borth}

    Recall $X(\ell)$ from introduction. In \cite[Section 16]{borthwickSpectralTheoryInfiniteArea2016}, an explicit form is given by 
    \[ S_1=\begin{pmatrix}
        \cosh(\ell/2) & \sinh(\ell/2) \\ \sinh(\ell/2) & \cosh(\ell/2)
    \end{pmatrix},\ S_2=\begin{pmatrix}
        \cosh(\ell/2) & a^{-1}\sinh(\ell/2) \\ a\sinh(\ell/2) & \cosh(\ell/2)
    \end{pmatrix} \]
    and $a\in\R$ such that $\tr(S_1S_2)=-2\cosh(\ell/2)$. We do the change of variable with $z=e^{-\ell/4}$, then the group becomes a Schottky family with
    \[ \lt(S_1)=\lt(S_2)=-\lt(S_1S_2)=z^{-2}. \]

    This family satisfies condition $(\bigstar)$ can be verified in the following way: \begin{equation*}
    \tr(S_1S_2)=2((z^{-2}+z^2)/2)^2+(a+a^{-1})((z^{-2}-z^2)/2)^2=-2(z^{-2}+z^2)/2.
\end{equation*}
So we obtain a solution \[a^{-1}=-1+2z+O(|z^2|).\]
The fixed points of $S_1$ are at $1$ (attracting) and $-1$ (repelling). The fixed points of $S_2$ are at $a^{-1}$ and $-a^{-1}$. The cross-ratios of these four points are  
\[ \left( \frac{1\pm a^{-1}}{1-\pm a^{-1}} \right)^2, \]
which have order greater than $-2$.
The contracting ratios of $S_1$ are $S_2$ are $z^4$. From these, we can verify that this family satisfies condition $(\bigstar)$.

    The graph is two vertices with three edge connecting them with each edge of length $2$. 
    From the discussion of the two generator case, the intermediate zeta function $Z_0$ with $M=0$ is exactly the graph zeta function for symmetric three-funnel surface. 
    
    Since \cref{prop:speed-conv} also works for $\SL_2(\R)$-case, combined with the computation of $Z_0$, we obtain the convergence of zeta functions in \cref{thm:three-funnel}. The statement of Hausdorff dimension follows directly.
\end{exam}

\begin{exam}\label{exa.PV19}
   We want to compute the zeta function for other $M$ and explain that we recover the result of \cite{pollicottZerosSelbergZeta2019a}.

Continue from the previous example. The three-funnel surfaces can be described by four discs, each of radius $z^2$.  Therefore the separation between $d_{na}(\frak{a}_i\infty,(D_{\frak{a}_i})^c)=e^{-2}$ for the generators.

After doing computations using the algorithm in \cref{prop:lc-cocycle} and \cref{cor:analytic-middle}, we obtain:\\
Let $x_1=e^{-2s}$, $x_2=e^{sz^2/2\log(1/|z|)}$, $x_3=e^{sz^4/2\log(1/|z|)}$. Then
\begin{equation*}
    Z_2(\Gamma,z,s)= ((x_2^4 + x_1^4 (-1 + x_2^2)^2 + x_1^2 (x_2^2 - 2 x_2^4))^2 (x_2^4 + 
   x_1^4 (-1 + x_2^2)^2 - 2 x_1^2 (x_2^2 + x_2^4)))/x_2^{12}
\end{equation*}
and
\begin{equation*}
\begin{aligned}
    Z_4(\Gamma,z,s)=\frac{1}{x_3^{18} x_2^{12}}& (x_1^8 - 4 x_3^2 x_1^8 + x_3^8 x_1^4 (-1 + x_1^2)^2 + 
    x_3^4 (-2 x_1^6 + 6 x_1^8) + x_3^3 x_1^4 (-2 + x_1^2) x_2^2 - 
    2 x_3^5 x_1^4 (-1 + x_1^2) x_2^2\\
    &\,\,\,+ x_3^7 x_1^2 (-1 + x_1^2)^2 x_2^2 - 
    x_3^6 (-1 + x_1^2) (4 x_1^6 + x_2^4 - x_1^2 x_2^4))^2 \\
    &(x_1^8 - 4 x_3^2 x_1^8 + 
    x_3^8 x_1^4 (-1 + x_1^2)^2 + x_3^4 (-2 x_1^6 + 6 x_1^8) + 
    4 x_3^5 x_1^4 (-1 + x_1^2) x_2^2 \\
    &\,\,\,- 2 x_3^7 x_1^2 (-1 + x_1^2)^2 x_2^2 - 
    2 x_3^3 x_1^4 (1 + x_1^2) x_2^2 - x_3^6 (-1 + x_1^2) (4 x_1^6 + x_2^4 - x_1^2 x_2^4)).
\end{aligned}
\end{equation*}

In particular, factorizing $Z_2$, we can recover the result of \cite{pollicottZerosSelbergZeta2019a} by \cref{cor:zero-ZM}. In $Z_2$, if we let $x_2=1$, then we obtain the Ihara zeta function of the symmetric three-funnel surface, which is predicted by \cref{lem:M'-M}.
\end{exam}

\begin{rem}
    From previous computations, we conjecture that the term $x_1$ in $Z_{2M}(\Gamma,z,s)$ has degree $6\cdot 2^M$ with coefficient $(1-x_{M+1}^2)^{3\cdot 2^M}$, where $x_{M+1}=e^{s z^{2M}/2\log(1/|z|)}$. This conjecture gives some support to the fractal Weyl law, that is the number of resonances in the region $\{s\in\C:\ \Re s>-C, |\Im s|\leq T \}$ for any fixed $C>0$ grows asymptotic to $T^{1+\delta_\Gamma}$ with $\delta_\Gamma$ the critical exponent of the Schottky group $\Gamma$.

    Below is a heuristic argument: In the region $\Re s \in [-C,\delta], \Im s \in[2T\pi,2(T+1)\pi]$ with $ T\approx |z|^{-2M}$, the leading coefficient $(1-x_{M+1}^2)^{3\cdot 2^M}$ is of constant size. Since $z$ is small, the terms $x_{j+1}=e^{s z^{2j}/2\log(1/|z|)}$ with $j\geq 1$ are almost constant. The only variable in $Z_{2M}$ is $x_1$. Hence as a polynomial on $x_1$, $Z_{2M}$ is of degree $6\cdot 2^M$ and the ratio between the coefficients and the leading coefficient is bounded. Therefore this polynomial has $6\cdot 2^M$ zeros of bounded size. Then $s=-\frac{1}{2}\log x_1$ is in the box $\{\Re s\in [-C,\delta],\ \Im s\in[2T\pi,2(T+1)\pi] \}$. The predicted numbers of zeros by the fractal Weyl law in this region is 
    \[(|z|^{-2M})^{\delta}\approx (|z|^{-2M})^{\log 2/2\log(1/|z|)}=2^{M}. \]

        For the region $\Re s\in [-C,\delta], \Im s \in[2T\pi,2(T+1)\pi]$ with $T\ll |z|^{-2M}$, the leading coefficient $(1-x_{M+1}^2)^{3\cdot 2^M}$ is so small, hence $Z_{2M}$ may have large zeros of $x_1$ and the corresponding $s$ is outside the region. Since $Z_{2M'}(s)$ for some $M'<M$ with $M'\approx \log T/2\log(1/|z|)$ is a good approximation to $Z_{2M}(s)$ in the region, the number of zeros of $Z_{2M}(s)$ can be computed from $Z_{2M'}(s)$, which is approximately $2^{M'}\approx T^\delta$. 

    In conclusion, with the intermediate zeta function $Z_{2M}$, for $T\leq |z|^{-2M}$, we observe the number of zeros in the region $\{\Re s\in[-C,\delta], \Im s\in [2T\pi,2(T+1)\pi]\}$ is approximately $T^\delta$, which fits the prediction of the fractal Weyl law.

\end{rem}

\section{Appendix}

\subsection{Critical exponent of the non-Archimedean Poincar\'e series}
\label{sec.appendix}

In this Appendix, we explain that the critical exponent of the Poincar\'e series is equal to the growth rate of the number of periodic geodesics for non-Archimedean $k$.

Let $\Gamma$ be a Schottky group in $\PGL_2(k)$ of $g$ generators. Recall from \cite[Theorem II.3.18]{poineauBerkovichCurvesSchottky2021} the non-wandering domain $O$ of $\Gamma$ on $\P^{1,an}_k$, the Mumford curve $X=\Gamma\backslash O$ and their skeleta $\Sigma_O$ and $\Sigma_X$. Then $\Sigma_X$ is a graph of rank $g$. Let $p$ be the quotient map from $O$ to $X$, then $p^{-1}\Sigma_X=\Sigma_O$ and $\Sigma_X=p(\Sigma_O)=\Gamma\backslash \Sigma_O$. Recall that we have a distance $d_a$ on $\Sigma_O$, the $\Gamma$ action on $\Sigma_O$ preserves the distance and we have a quotient distance on $\Sigma_X$.

Recall the definition of length periodic primitive geodesic in \cref{lem:na length}. Let $\cG_\Gamma(L)$ be the number of periodic primitive geodesics of length less than $L$ in $\Sigma_X$, which is also the least translation length of the corresponding conjugacy class on $\Sigma_O$. Let $\cN_\Gamma(L,o)$ be the number of $\gamma\in\Gamma$ such that $d_a(o,\gamma o)\leq L$ for $o\in \Sigma_O$. 

Since $\Sigma_O$ is a locally finite tree with a metric, it is CAT($-1$). We can apply \cite{roblinFonctionOrbitaleGroupes2002} to obtain the critical exponent of the Poincar\'e series
\[  \delta(\Gamma)=\lim_{L\rightarrow\infty}\frac{1}{L}\log\cN_\Gamma(L,o).  \]

Moreover, by \cite[Corollaire 2]{roblinFonctionOrbitaleGroupes2002}
\begin{prop}
    We have
    \[ \lim_{L\rightarrow\infty}\frac{1}{L}\log\cG_\Gamma(L)=\delta(\Gamma). \]
\end{prop}

By \cite[Theorem 2.10]{hortonWhatAreZeta2006}, we have 

\begin{cor}
    The first zero of Ihara zeta function of the graph is equal to the critical exponent.
\end{cor}

\subsection{Proof of \cref{prop:expansion-der-cycl}}\label{sec.proof}

Before the proof, we state three elementary lemmas about estimates of the coefficients of the Laurent series after multiplication, division, and logarithm.

\begin{lem}\label{lem:laurent-mult}
    Suppose $f(t),g(t)\in \Ct$ are two formal Laurent series:
    \begin{equation*}
        f(t)= t^m\sum_{j=0}^{\infty} f_j t^j\quad \text{with}\quad |f_j|\leq e^{C(j+1)}\,\,\text{for}\,\,j\in \mathbb{Z}_{\geq 0};
    \end{equation*}
    \begin{equation*}
        g(t) =t^{m'}\sum_{j=0}^{\infty} g_j t^j\quad \text{with} \quad |g_j|\leq e^{C(j+1)} \,\,\text{for}\,\,j\in \mathbb{Z}_{\geq 0}.
    \end{equation*}
    Then $f(t)g(t)\in\Ct$ has the Laurent series expansion
    \begin{equation*}
        f(t)g(t) = t^{m+m'}\sum\limits_{j=0}^{\infty} h_j t^j\quad \text{with} \quad |h_j|\leq (j+1)e^{C(j+2)}\,\,\text{for}\,\,j\in\mathbb{Z}_{\geq 0}.
    \end{equation*}
\end{lem}

\begin{proof}
    We have
    \begin{equation*}
        |h_n|=\left|\sum\limits_{j+k=n}f_jg_k \right|\leq (n+1)e^{C(n+2)}.\qedhere
    \end{equation*}
\end{proof}

\begin{lem}\label{lem:laurent-divide}
    Suppose $f(t),g(t)\in \Ct$ are two formal Laurent series with 
    \begin{equation*}
        f(t)= t^m\sum_{j=0}^{\infty} f_j t^j,\quad f_0\neq 0,\, |f_j|\leq e^{C(j+1)};
    \end{equation*}
    \begin{equation*}
        g(t) =t^{m'}\sum_{j=0}^{\infty} g_j t^j,\quad |g_0|\geq e^{-A},\, |g_j|\leq e^{C(j+1)}.
    \end{equation*}
    Then $f(t)/g(t)\in\Ct$ has the Laurent series expansion
    \begin{equation*}
        f(t)/g(t)=t^{m-m'}\sum\limits_{j=0}^{\infty} h_j t^j,\quad |h_j|\leq e^{(3C+C_1+A)(j+1)}
    \end{equation*}
    for some universal constant $C_1$. 
\end{lem}
\begin{proof}
    We may first take $\tilde{g}_0=1$ and $|\tilde{g}_j|=|g_j/g_0|\leq e^{C(j+1)+A}$. Then
    \begin{equation*}
        \frac{g_0t^{m'}}{g(t)}=\sum\limits_{n=0}^{\infty}\left(-\sum\limits_{j=1}^{\infty}\tilde{g}_jt^j\right)^n=\sum\limits_{n=0}^{\infty} g'_n t^n
    \end{equation*}
    where
    \begin{equation*}
        |g'_n| \leq \sum_{j_1+\cdots+j_k=n} |\tilde{g}_{j_1}\cdots \tilde{g}_{j_k}| \leq e^{C_0n}e^{2Cn+An}\leq e^{(2C+C_0+A)n}
    \end{equation*}
    for some universal constant $C_0$.
    Finally we multiply by $f(t)$:
    \begin{equation*}
        \frac{f(t)}{g(t)}=\frac{f(t)}{g_0t^{m'}}\frac{g_0t^{m'}}{g(t)} = g_0^{-1}t^{m-m'}\left(\sum_{j=0}^{\infty} f_j t^j\right)\left(\sum_{n=0} g_n' t^n\right)=t^{m-m'}\sum\limits_{j=0}^{\infty} h_j t^j
    \end{equation*}
    where
    \begin{equation*}
        |h_n|\leq |g_0|^{-1}\sum_{j+k=n}|f_jg'_k|\leq (n+1)e^{A}e^{C(n+1)} e^{(2C+C_0+A)n}\leq e^{(3C+C_1+A)(n+1)}
    \end{equation*}
    for some universal constant $C_1$.
\end{proof}

\begin{lem}\label{lem:laurent-log}
    Suppose $f(t)\in \Ct$ has the expansion
    \begin{equation*}
        f(t)=t^m\sum_{j=0}^{\infty} f_j t^j,\quad f_0\neq 0,\quad |f_j|\leq e^{C(j+1)}.
    \end{equation*}
    Then as a formal series, $\log(f(t))$ has the expansion
    \begin{equation*}
        \log(f(t))=m\log t+ \log(f_0)+\sum_{j=0}^{\infty} h_j f_0^{-j} t^j,\quad |h_j|\leq e^{(3C+C_0)j}
    \end{equation*}
    where $C_0$ is a universal constant.
\end{lem}
\begin{proof}
    We use the formula for $\log$:
    \begin{equation*}
    \begin{aligned}
    \log(f(t))&=m\log t+ \log(f_0)+\log\left(1+\sum\limits_{j=1}^{\infty}\frac{f_j}{f_0}t^j\right)\\
    &=m\log t+ \log(f_0)+\sum_{n=1}^{\infty}\frac{(-1)^{n+1}}{n}\left(\sum\limits_{j=1}^{\infty}\frac{f_j}{f_0}t^j\right)^n.
    \end{aligned}
    \end{equation*}
    Therefore,
    \begin{equation*}
        |h_n|\leq\sum_{j_1+\cdots+j_k=n}\frac{1}{k}|f_{j_1}\cdots f_{j_k}|\cdot|f_0|^{n-k}\leq e^{C_0 n +3Cn}.
    \end{equation*}
\end{proof}

\begin{proof}[Proof of \cref{prop:expansion-der-cycl}]

First we check the estimate for $\mathfrak{a}_{i_2}(\cdots(\mathfrak{a}_{i_N}(\infty)))$. Each entry of $\mathfrak{a}_{i_2}\cdots \mathfrak{a}_{i_{N}}$ is a sum of at most $2^{N-1}$ terms, with each term a product of $N-1$ entries of generators. By subadditivity and submultiplicativity of the hybrid norm, for each entry $\mathfrak{a}_{i_2}\cdots \mathfrak{a}_{i_{N}}(i,l)$ we have
\begin{equation*}
\sup_{1\leq i,l\leq 2}\|\mathfrak{a}_{i_2}\cdots \mathfrak{a}_{i_{N}}(i,l)\|_{\mathrm{hyb}}\leq 2^{N-1} \sup_{\substack{\mathfrak{a}\in\mathcal{A}\\
1\leq i,l\leq 2}}\|\mathfrak{a}(i,l)\|_{\mathrm{hyb}}^{N-1}\lesssim e^{CN}.
\end{equation*}
Suppose $\mathfrak{a}_{i_2}\cdots \mathfrak{a}_{i_{N}}=\begin{pmatrix}
    a&b\\
    c&d
\end{pmatrix}$. 
For some $0<r<1/e$ independent of $N$ and $|t|<r$, the Schottky discs $D_{\mathfrak{a}}$ are disjoint for $\mathfrak{a}\in\mathcal{A}$. If $a=\sum_{n\geq m} a_nt^n\in \Ct$ with $a_{m}\neq 0$, then $\|a\|_{\mathrm{hyb}}\leq e^{CN}$ implies that $e^{-m}\leq \|a\|_{\mathrm{hyb}}\leq e^{CN}$ and
\begin{equation*}
    |a_z|\leq \sum\limits_{n\leq 0} |a_n|(r/2)^{n}+\sum\limits_{n>0}|a_n|e^{-n}\leq \max\{(re/2)^{m},1\}\|a\|_{\mathrm{hyb}}\leq  e^{C'N},\quad \frac{r}{2}\leq |z|\leq 1/e. 
\end{equation*}
In the last step, we use $(re/2)^{m}=e^{m\log(r/2)}$ and $\log(r/2)<0$.

Let $f(z)=\sum\limits_{j=0}^{\infty}c_j z^j$ where $c_j$ is the coefficient in the Laurent series $c=t^m\sum_{j=0}^{\infty} c_jt^j \in \Ct$ with $c_0\neq 0$. Since $\infty$ does not lie in $D_{\overline{\mathfrak{a}}_{i_N},z}$ and $a/c=\mathfrak{a}_{i_2}\cdots \mathfrak{a}_{i_N}(\infty)$, $a_z/c_z$ is in the disc $D_{\mathfrak{a}_{i_2}\cdots\mathfrak{a}_{i_N},z}\subset D_{\mathfrak{a}_{i_2},z}$, which we assume is bounded uniformly from $\infty$. We conclude that $|a_z/c_z|\leq C$ and $c_z\neq 0$ for $|z|<r$. Thus $\log|f(z)|$ is harmonic in the disc $D(0,r)$. 
Since $ad-bc=1$, we have
\begin{equation*}
   1\leq |a_zd_z|+|b_zc_z|\leq C(|a_z|+|b_z|)|c_z|\leq e^{CN}|c_z|,\quad \frac{r}{2}\leq |z|<r.
\end{equation*}
From the Laurent series expansion $c=t^{m}\sum_{j=0}^{\infty} c_jt^j \in \Ct$, we have
\begin{equation*}
    e^{-m} \leq \|c\|_{\mathrm{hyb}}\leq e^{CN}, 
\end{equation*}
which implies $m\geq -CN$.

By mean value theorem on $D(0,r/2)$, we have
\[\log|c_0|=\log|f(0)|\geq \min_{|z|=r/2}\log|f(z)|=\min_{|z|=r/2}\log|c_z/z^m|\geq -CN - m\log(r/2) \geq -C'N.\]
In the last step we use $\log(r/2)<0$ and $m\geq -CN$. Therefore, $|c_0|\geq e^{-CN}$.
Recall $\|a\|_{\mathrm{hyb}}\leq e^{CN}$ and $\|c\|_{\mathrm{hyb}}\leq e^{CN}$, we have a Laurent series expansion
\begin{equation}\label{equ:aj ll}
    \mathfrak{a}_{i_2}(\cdots(\mathfrak{a}_{i_N}(\infty)))=\frac{a}{c}=t^{m'}\sum_{j=0}^{\infty}a_j' t^j
\text{ with }|a_j'|\leq e^{C(N+1)j},
\end{equation} where in the last step we use \cref{lem:laurent-divide}.

Suppose $\mathfrak{a}_{i_1}=\begin{pmatrix}
    a'&b'\\
    c'&d'
\end{pmatrix}$, then $\mathfrak{a}_{i_1}'(w)=(c'w+d')^{-2}$. We would like to show
\begin{equation*}
        \mathfrak{a}_{i_1}'(\mathfrak{a}_{i_2}(\cdots\mathfrak{a}_{i_{N-1}}(\mathfrak{a}_{i_N}(\infty))))=t^m \sum\limits_{j=0}^{\infty} a_j t^j ,\quad |a_j|\leq e^{C(N+1)j}.
\end{equation*}
By \cref{lem:laurent-mult}, it suffices to show a similar expansion for $(\mathfrak{a}_{i_2}\cdots\mathfrak{a}_{i_N}(\infty)+d'/c')^{-1}$. By \cref{lem:laurent-divide} and \cref{equ:aj ll}, it suffices to show the leading coefficient $\pi_0$ in the expansion
\begin{equation*}
    \mathfrak{a}_{i_2}\cdots\mathfrak{a}_{i_N}(\infty)+d'/c'=t^m\sum\limits_{j=0}^{\infty}\pi_j t^j,\quad \pi_0\neq 0
\end{equation*}
satisfies $|\pi_0|\geq  e^{-CN}$.
By \cref{prop:separtion}, $\mathfrak{a}_{i_2}(\cdots(\mathfrak{a}_{i_N}(\infty)))\in \mathfrak{a}_{i_2}D_{\mathfrak{a}_{i_3},z}\subset |z|^{v'}D_{\mathfrak{a}_{i_2},z}$ for some $v'>0$ and $0<|z|<r$. Hence the point $ \mathfrak{a}_{i_2}(\cdots(\mathfrak{a}_{i_N}(\infty)))$ is $|z|^v$ separated from $D_{\overline{\mathfrak{a}}_{i_1},z}$ (which contains $-d'/c'$) for some $v>0$ and $0<|z|<r$. By considering the separation with non-Archimedean norm, we know $|m|\leq C$ is uniformly bounded. Consider the holomorphic function $f(z)=\sum\limits_{j=0}^{\infty}\pi_j z^j$, then $\log|f(z)|$ is a harmonic function in a small disc $|z|<r$. By mean value theorem,
\begin{equation*}
    \log |\pi_0|\geq \min_{|z|=r/2} \log|f(z)|\geq \min_{|z|=r/2} \log|z^mf(z)|- m\log(r/2)\geq \nu\log(r/2)-m\log(r/2).
\end{equation*}
Since $m$ and $\nu$ are uniform, we actually prove $|\pi_0|\geq e^{-C}$ holds for some uniform constant $C>0$. Therefore \cref{lem:laurent-divide} gives \cref{eq:expansion-derivative-cocycle}.
Similarly, since the discs are bounded uniformly away from $\infty$, we have
\begin{equation*}
    \log |\pi_0|\leq \max_{|z|=r/2} \log|f(z)|\leq \max_{|z|=r/2} \log|z^mf(z)|- m\log(r/2)\leq \log C-m\log(r/2).
\end{equation*}
Thus $|\pi_0|\leq e^C$ for some uniform constant $C>0$.
Since we know the leading coefficient $a_0$ in \cref{eq:expansion-derivative-cocycle} is given by $\pi_0^{-2}$ times the leading coefficient of $c'^{-2}$. We conclude
\begin{equation*}
    |a_0|\geq e^{-2C'}|\pi_0|^{-2}\geq e^{-2C-2C'}.
\end{equation*}

\cref{eq:expansion-lM} then follows from \cref{eq:expansion-derivative-cocycle} and the lower bound $|a_0|\geq e^{-C}$.
\end{proof}

\bibliographystyle{alpha}

%\bibliography{bibfile}
\bibliography{main}

\bigskip

 \noindent 
	Jialun Li. {\it CNRS-Centre de Math\'ematiques Laurent Schwartz, \'Ecole Polytechnique, Palaiseau, France.}  \\
	email: {\tt jialun.li@cnrs.fr} 
		
		\bigskip  

     \noindent 
	Carlos Matheus. {\it CNRS-Centre de Math\'ematiques Laurent Schwartz, \'Ecole Polytechnique, Palaiseau, France.}  \\
	email: {\tt carlos.matheus@math.cnrs.fr } 
		
		\bigskip  

\noindent 
	Wenyu Pan. {\it Department of Mathematics, University of Toronto, Toronto, Cananda}  \\
	email: {\tt wenyup.pan@utoronto.ca} 
		
		\bigskip  
    \noindent 
	Zhongkai Tao. {\it Department of Mathematics, Evans Hall, University of California, Berkeley, CA 94720, USA}  \\
	email: {\tt ztao@math.berkeley.edu} 
		
		\bigskip

\end{document}